\documentclass[11pt]{amsart}
\usepackage{amssymb, latexsym, graphicx, mathrsfs, enumerate, amsmath, amsthm, textcomp, url}
\usepackage{color}
\input xy
\xyoption{all}

\newlength{\margins}
\usepackage[top=\margins,bottom=\margins,left=\margins,right=\margins]{geometry}

\numberwithin{equation}{section}
\newtheorem{Thm}{Theorem}[section]
\newtheorem{Prop}[Thm]{Proposition}
\newtheorem{Cor}[Thm]{Corollary}
\newtheorem{Lem}[Thm]{Lemma}
\newtheorem{Conj}[Thm]{Conjecture}

\theoremstyle{definition}
\newtheorem{Def}[Thm]{Definition}

\newtheorem{Rmk}[Thm]{Remark}

\newcommand{\bQ}{\overline{\mathbb{Q}}}

\newcommand{\bF}{\overline{\mathbb{F}}}

\newcommand{\C}{\mathbb{C}}
\newcommand{\R}{\mathbb{R}}
\newcommand{\Q}{\mathbb{Q}}

\newcommand{\Z}{\mathbb{Z}}

\newcommand{\F}{\mathbb{F}}

\newcommand{\lra}{\longrightarrow}

\newcommand{\A}{\mathbb{A}}

\newcommand{\T}{\mathcal{T}}

\newcommand{\diag}{{\rm diag}}
\newcommand{\ds}{\displaystyle}

\begin{document}

\title[Siegel series and  intersection numbers]
{A  reformulation of the Siegel series and  intersection numbers}


\author[Sungmun Cho and Takuya Yamauchi]{Sungmun Cho and Takuya Yamauchi}
\thanks{The first author was   supported by JSPS KAKENHI Grant No. 16F16316, 
 Samsung Science and Technology Foundation under Project Number SSTF-BA1802-03, and NRF 2018R1A4A 1023590.}
\thanks{The second author is partially supported by JSPS KAKENHI Grant Number (B) No.19H01778.}
\address{Sungmun Cho \\ 
Department of Mathematics, POSTECH, 77, Cheongam-ro, Nam-gu, Pohang-si, Gyeongsangbuk-do, 37673, KOREA}

\email{sungmuncho12@gmail.com}
\address{Takuya Yamauchi \\ 
Mathematical Inst. Tohoku Univ.\\
 6-3,Aoba, Aramaki, Aoba-Ku, Sendai 980-8578, JAPAN}
\email{tyamauchi@m.tohoku.ac.jp}

\maketitle

\begin{abstract}
 


In this paper, we will explain a conceptual reformulation and inductive formula of the Siegel series. 
Using this, we will explain that both sides of the local intersection multiplicities of \cite{GK1} and the Siegel series have the same inherent structures, beyond matching values.

As an application, we will prove a new identity between the intersection number of two modular correspondences over $\F_p$ and the sum of the Fourier coefficients of the Siegel-Eisenstein series for $\mathrm{Sp}_4/\Q$ of weight $2$, which is independent of  $p \left(> 2\right)$.
In addition, we will explain a description of the local intersection multiplicities of the special cycles over $\F_p$ on the supersingular locus of the `special fiber'  of the Shimura varieties for 
$\mathrm{GSpin}(n,2), n\le 3$  in terms of the Siegel series directly.

\end{abstract}





\tableofcontents

\section{Introduction}\label{sectionintro}
\subsection{On the Gross-Keating's formula}\label{subsectionintro}
In \cite{GK1} Gross and Keating studied the arithmetic intersection number of three modular correspondences which are regarded 
as cycles in the self-product $Y_0(1)\times Y_0(1)/\Z$ of the moduli stack $Y_0(1)$ of elliptic curves over any scheme. 
They described it purely in terms of certain invariants of a ternary quadratic form  created by themselves to formulate it.
This invariant has
been generalized to quadratic forms of any degree over a local field, and is nowadays
called the Gross-Keating invariant.
They already  expected in the introduction of their paper that the  arithmetic intersection number coincides with the sum of the
Fourier coefficients of the derivative of the Siegel-Eisenstein series of weight 2 and of degree 3, which has been studied 
thoroughly in \cite{A}.
Kudla in \cite{Kudla1} later proposed a general program (local version) to make a connection between the  local intersection multiplicity of special cycles on 
an integral model of the Shimura varieties for $\mathrm{GSpin}(n,2)$
and the derivative of the local factor of a  Fourier coefficient of  the Siegel-Eisenstein series of weight $(n+2)/2$ and of degree $n+1$.
In the latter object, such local factor is called the Siegel series.


The program has been vastly studied by a series of papers  \cite{Kudla1}, \cite{KRY1},\cite{KRY2},\cite{KR2},\cite{KR1} for $0\le n\le 3$, and \cite{BY} for general $n$, when the dimension of the arithmetic intersection is zero.
A strategy to compute the local intersection multiplicities which had been taken in these papers is to reduce them to the case of Gross-Keating. 
On the other hand, a computation of the Siegel series side is based on Katsurada's paper \cite{Kat}. 
The relation between both sides then follows by a direct comparison.
Note that beyond direct comparison, there had not been known an evidence or structure to conceptually yield the equality between them.

 Therefore, in order to understand Kudla's program in conceptual way toward the higher dimensional case of the arithmetic intersection, it would be  important  to have a better understanding on the relation between  the local intersection multiplicity of Gross and Keating in \cite{GK1} and the Siegel series.
 In \cite{GK1}, a key step is to derive an inductive formula (Lemma 5.6 of \cite{GK1}) for the local intersection multiplicities over $W(\bF_p)$ 
at a prime $p$, which involves the local intersection multiplicity on the special fiber at $p$ (Lemma 5.7 of \cite{GK1}).  
As we will compare this inductive formula with our result in the next  subsection, we describe the precise form of their inductive formula.
 
Let $(L, q_L)$ be an anisotropic quadratic lattice over $\mathbb{Z}_p$ of rank $3$. Then the Gross-Keating invariant of $(L, q_L)$ consists of three integers $\mathrm{GK}(L)=(a_1, a_2, a_3)$ with $a_1\leq a_2 \leq a_3$.
If we denote by $\alpha_p(a_1, a_2, a_3)$  the local intersection multiplicity associated to $(L, q_L)$ (see (3.18) of \cite{GK1}), then it satisfies the following inductive formula, with respect to the Gross-Keating invariant:
(cf. the proof of Theorem \ref{thmaniso2}):
\begin{equation}\label{eqintro1}
\alpha_p(a_1, a_2, a_3)=\alpha_p(a_1, a_2, a_3-2)+\mathcal{T}_{a_1, a_2}.
\end{equation}
Here, $\mathcal{T}_{a_1, a_2}$ is the local intersection multiplicity of two cycles on the special fiber of the setting of Gross-Keating.

\subsection{On the Siegel series}\label{onss}
On the other hand,  the Siegel series is of great importance in automorphic forms, such as the study of conjectures related to automorphic $L$-functions  and  the construction of automorphic forms of level $1$, so-called Ikeda lift.
We refer to the first several paragraphs of \cite{IK2} for more introductory discussion about the important usages of the Siegel series in this context.
Theories of the Siegel series have been developed by many people such as Kitaoka and Shimura. It was Katsurada in \cite{Kat} who firstly found the explicit formula of the Siegel series for $\mathbb{Z}_p$. 
However, as mentioned in the introduction of \cite{IK2}, his formula is  complicated and it is not clear which invariant of a quadratic form determines the Siegel series. 

Recently, Ikeda and Katsurada in \cite{IK2} obtained the formula of the Siegel series over any finite extension of $\mathbb{Z}_p$.
Furthermore, they show that the Siegel series is completely determined by the Gross-Keating invariant with extra data, called the Extended Gross-Keating datum, for any quadratic form over any finite extension of $\mathbb{Z}_p$.

The Siegel series is usually described as a polynomial. 
An explicit formula of  the Siegel series given in \cite{Kat} and \cite{IK2} is to determine the coefficients of this polynomial. 
On the other hand, theoretical interpretation of these coefficients had not been known yet.

\subsection{Reformulation of the Siegel series}
Our first main result is to reformulate the Siegel series over any finite extension of $\mathbb{Z}_p$ in Theorem \ref{eqldf} and Corollary \ref{rmkse}. 
The Siegel series can be defined as an integral of certain volume form on a $p$-adic manifold (cf. Definitions \ref{deflocaldensity} and \ref{defse}) associated to a quadratic lattice $(L, q_L)$ over $\mathfrak{o}$, where $\mathfrak{o}$ is the ring of integers of a finite field extension of $\mathbb{Q}_p$ (for any $p$). 
It is usually denoted by $\mathcal{F}_L(X)$, as a polynomial of $X$.
Let $n$ be the rank of $L$ so that the Gross-Keating invariant consists of $n$-integers $\mathrm{GK}(L)=(a_1, \cdots, a_n)$ satisfying $a_i\leq a_j$ with $i\leq j$.

\begin{Thm}\label{main-1}(Corollary \ref{rmkse})
we have the following description of the Siegel series:
\begin{multline*}
\mathcal{F}_L(X)=\\
(1-X)\cdot
\sum_{\substack{0\leq b\leq \frac{|\mathrm{GK}(L)|}{2},\\ n_0\leq a\leq n}} \left(\#\mathcal{S}_{(L, a^{\pm}, b)}\cdot
f^{b\cdot(n+1)}X^{2b}\cdot (1+\chi(a^{\pm})f^{n-a/2}X)\prod_{1\leq i < n-a/2}(1-f^{2i}X^2)\right).
\end{multline*}
\end{Thm}
Here, $|\mathrm{GK}(L)|=a_1+\cdots +a_n$,
$n_0$ is the number of $0$'s in $\mathrm{GK}(L)=(a_1, \cdots, a_n)$ (cf. Proposition \ref{propzero}),
$f$ is the cardinality of the residue field of $\mathfrak{o}$, 
$\chi$ is defined just below Equation (\ref{eq37}),
and
$\mathcal{S}_{(L, a^{\pm}, b)}$ is the set of certain (depending on two integers $a^{\pm}$ and $b$) quadratic lattices containing $L$, which can be found at Equation (\ref{sab}).

This gives a conceptual and theoretical interpretation of  the coefficients of $\mathcal{F}_L(X)$
 as a weighted sum of certain number of quadratic lattices. 
The method  used in Theorem \ref{main-1} (Corollary \ref{rmkse}) is based on another geometric nature of the Siegel series involving the stratification of a $p$-adic scheme, geometric description of each stratum, Grassmannian, and lattice counting argument.
This  is
largely different from the known techniques in this context\footnote{Remark in page 444 of \cite{Kat} says `it seems very interesting problem to prove Theorems 4.1 and 4.2  directly from the local theory of quadratic forms'.
Here, Theorems 4.1 and 4.2 are main results of \cite{Kat}, which give an explicit formula of the Siegel series over $\mathbb{Z}_p$.
Our method can be understood in the spirit of the problem proposed by Katsurada.}.

Using the result of \cite{IK2}, we then  obtain an inductive formula of the Siegel series, with respect to the Gross-Keating invariant, under Conjecture \ref{conj4} concerning about quadratic forms (which is verified to be true when $p$ is odd or when $(L, q_L)$ is anisotropic over $\mathbb{Z}_2$ in Lemmas \ref{l4}-\ref{l5}).

We describe our inductive formula more precisely. 
If we choose the integer $d$ characterized by the condition $a_{n-d}<\underbrace{a_{n-d+1}=\cdots =a_n}_d$, then we can associate certain lattice $L^{(d, n)}$ containing $L$ whose rank is also $n$. 
To be more precise, for a reduced basis  $(e_1, \cdots, e_n)$ of $L$ given in Definition \ref{def3.2}, the lattice $L^{(d, n)}$ is spanned by $(e_1, \cdots, e_{n-d}, \underbrace{\frac{1}{\pi}\cdot e_{n-d+1}, \cdots, \frac{1}{\pi}\cdot e_{n}}_d)$. Here, $\pi$ is a uniformizer in $\mathfrak{o}$.
A second main theorem of the current paper is the following: 

\begin{Thm}\label{main0}(Theorem \ref{mainthm})
Assume that Conjecture \ref{conj4} is true.
 If  $L^{(d, n)}$ is a quadratic lattice,
then we have the following inductive formula, with respect to the Gross-Keating invariant, of the Siegel series $\mathcal{F}_L(X)$:

\begin{multline*}
\mathcal{F}_L(X)=\sum_{m=1}^{d} \left( c_m \cdot f^{\left(n+1\right)m}\cdot X^{2m}\cdot\sum_{L'\in \mathcal{G}_{L, d, m}}\mathcal{F}_{L'}(X)\right)+\\
\displaystyle (1-X)(1-f^{d}X)^{-1}\cdot \left(\prod_{i=1}^{d}(1-f^{2i}X^2)\right)\cdot \mathcal{F}_{L_0^{(d, n)}}(f^{d}X),
\end{multline*}
where $c_m=(-1)^{m-1}f^{m(m-1)/2}$.
Here, $f$ is the cardinality of the residue field of $\mathfrak{o}$.
For $L'\in \mathcal{G}_{L, d, m}$, 
\[
\left\{
  \begin{array}{l}
 \mathrm{GK}(L) \succ \mathrm{GK}(L');\\
|\mathrm{GK}(L')|=|\mathrm{GK}(L)|-2m;\\
 \mathrm{GK}(L_0^{(d, n)})=\mathrm{GK}(L)^{(n-d)}.
    \end{array} \right.
\]

Note that notion of 
$L^{(d, n)}$, $\mathcal{G}_{L, d, m}$, and $\binom{m}{k}_f$ 
can be found at the beginning of Section \ref{sectionifss}.
Notion of  $L_0^{(d, n)}$  can be found at Remark \ref{rlx}.(1).
\end{Thm}

Here, $\mathcal{G}_{L, d, m}$ is identified with Grassmannian to classify the set of $m$-dimensional subspaces of the vector space 
of dimension $d$ (given by $L^{(d, n)}/L$)  over a finite field $\mathfrak{o}/(\pi)$, whose order is $\binom{d}{m}_f$.

\subsection{The comparison between Gross-Keating's formula and the Siegel series}\label{subseccompa}
Since both sides, Gross-Keating's formula and the Siegel series, have inductive formulas, it is natural to ask whether or not there is a relation between them.

If we restrict ourselves to an anisotropic quadratic lattice over $\mathbb{Z}_p$ of rank $n$, which covers  the case of Gross-Keating, then we have more refined and simpler inductive formula (Theorem \ref{indfaniso3}) as follows:

\begin{equation}\label{eqintro2}
\mathcal{F}_L(X)=\left\{
  \begin{array}{l l}
  p^{n+1}\cdot X^2\cdot \mathcal{F}_{L^{(n)}}(X)+(1-X)(1+pX)\cdot \mathcal{F}_{L^{(n)}_0}(pX)   & \quad    \text{if $2\leq n \leq 4$};\\
  p^{2}\cdot X^2\cdot \mathcal{F}_{L^{(1)}}(X)+(1-X)(1+pX)  & \quad    \text{if $n=1$}.
    \end{array} \right.
\end{equation}

After comparing both inductive formulas, we obtain the following result:

\begin{Thm}\label{mainthmintro}(cf. Theorems \ref{thmaniso1} and \ref{thmaniso2})
The inductive formula of Gross-Keating in Equation (\ref{eqintro1}) and   the derivative of Equation (\ref{eqintro2})  at $1/p^2$ with $n=3$  \textbf{do match} each other.
 As a direct consequence, we have
\[
\left\{
  \begin{array}{l l}
\alpha_p(a_1,a_2,a_3)=c_1\cdot \mathcal{F}'_L(1/p^2);\\
\mathcal{T}_{a_1, a_2}=c_2\cdot \mathcal{F}'_{L_0^{(3)}}(1/p),
    \end{array} \right.\]
for  explicitly computed constants $c_1$ and $c_2$.
Here $\mathrm{GK}(L_0^{(3)})=(a_1, a_2)$.
\end{Thm}

This theorem shows that both sides of the local intersection multiplicity and the Siegel series have the same inherent structures, that is, the same inductive formula, beyond matching their values. 
In addition, it gives us a new observation that the local intersection multiplicity on the special fiber can also be described in terms of the derivative of the Siegel series.

 \subsection{Applications to intersection numbers over finite fields}\label{subsecmodpintro}
 Since the above theorem matches both sides on special fibers,
 we can naturally consider their applications over finite fields.
We explain two consequences in this line: intersection numbers over finite fields and the local intersection multiplicities on special fibers of  $\mathrm{GSpin}(n,2)$ Shimura varieties with $n\leq 3$ (in the case of zero dimension of the arithmetic intersection).
 
\subsubsection{Intersection numbers over finite fields} Since we obtained a new description of the local intersection multiplicity on the special fiber in the setting of Gross-Keating in terms of the derivative of the Siegel series, it is natural to compute the intersection numbers of two modular correspondences on finite fields.  
 
More precisely, let $\varphi_m$ be the modular polynomial in $\Z[x,y]$ of degree $m$ whose irreducible factor corresponds to  
an affine model of the modular curves $Y_0(m/n^2)$ for some $n^2|m$ in $Y_0(1)\times Y_0(1)$ (cf. \cite{V1}). 
Then for positive integers $m_1,m_2$ and a prime $p$, we define the intersection number over the finite field $\F_p$ or over the complex field as follows:
\begin{equation}\label{dimension}
(T_{m_1,p},T_{m_2,p}):={\rm length}_{\F_p}\F_p[x,y]/(\varphi_{m_1},\varphi_{m_2}),\ 
(T_{m_1,\C},T_{m_2,\C}):={\rm length}_{\C}\C[x,y]/(\varphi_{m_1},\varphi_{m_2}).
\end{equation}
We compare the above two intersection numbers by using 
Theorem \ref{mainthmintro} on the supersingular locus and the theory of quasi-canonical lifts on the ordinary locus.
The following theorem is our result:

\begin{Thm}\label{new}(Proposition \ref{ip} and Theorem \ref{newid}) 
The intersection number $(T_{m_1,p},T_{m_2,p})$ is finite if and only if $m_1m_2$ is not a square. In addition, 
if $m_1m_2$ is not a square and $p$ is odd, then 
$$(T_{m_1,p},T_{m_2,p})=(T_{m_1,\C},T_{m_2,\C}).$$
\end{Thm}

We refer to Remark \ref{conrad} for a discussion with $p=2$.\\
Since $(T_{m_1,\C},T_{m_2,\C})$ is the sum of the Fourier coefficients of the Siegel-Eisenstein series for $\mathrm{Sp}_4/\Q$ by Proposition 2.4 of \cite{GK1}, 
the intersection number on the special fiber is also the sum of the Fourier coefficients of the Siegel-Eisenstein series for $\mathrm{Sp}_4/\Q$.
Furthermore  the intersection number is independent of the characteristic of a finite field with $p>2$,
whereas the local intersection multiplicities highly depend on $p$.\\

The above theorem yields a new interpretation on a classical object $\mathbb{Z}[\frac{1}{2}][x,y]/(\varphi_{m_1}, \varphi_{m_2})$.
Note that the two main objects to be analyzed in  \cite{GK1}  are geometric interpretations of 
\[
\C[x,y]/(\varphi_{m_1},\varphi_{m_2})~~~~\textit{      and     }~~~~~~~~~~~~ \Z[x,y]/(\varphi_{m_1},\varphi_{m_2},\varphi_{m_3}).
\]
Namely, the dimension of the first object is the intersection number of two modular correspondences over $\C$ and (log of) the cardinality of the second object is the arithmetic intersection number of three modular correspondences over $\Z$. 
In this context, the $\Z$-module $\mathbb{Z}[\frac{1}{2}][x,y]/(\varphi_{m_1}, \varphi_{m_2})$ has the following interesting interpretations:

\begin{Thm}\label{new!!}(Theorem \ref{flat}) 
\begin{enumerate}
\item $\mathbb{Z}[\frac{1}{2}][x,y]/(\varphi_{m_1}, \varphi_{m_2})$ is a free $\mathbb{Z}[\frac{1}{2}]$-module. 
\item 
The rank of $\mathbb{Z}[\frac{1}{2}][x,y]/(\varphi_{m_1}, \varphi_{m_2})$, as a $\mathbb{Z}[\frac{1}{2}]$-module,  is equal to 
\[
\frac{1}{288}
 \sum_{T\in {\rm Sym}_{2}(\Z)_{> 0} \atop {\rm diag}(T)=(m_1, m_2)}c(T).
\]
Here, $c(T)$ is the Fourier coefficient of the Siegel-Eisenstein series for ${\rm Sp}_4(\Z)$ of weight 2 with respect to the $(2\times 2)$- half-integral symmetric matrix $T$.
\end{enumerate}

\end{Thm}

\subsubsection{The local intersection multiplicities on the special fiber in orthogonal Shimura varieties} 
Let $p$ be an odd prime. 
Let $T'$ be a positive definite half-integral symmetric matrix over $\Z_{(p)}$ of size $n+1$ for a non-negative integer $n$. 
As explained in subsection \ref{subsectionintro}, the local intersection multiplicity on the arithmetic intersection $\mathcal{Z}=
\mathcal{Z}(T')$ of 
some special cycles associated to $T'$ on an integral model $\mathcal{M}$ over $\Z_{(p)}$ of Shimura varieties for 
$G=\mathrm{GSpin}(n,2)$, when the dimension of $\mathcal{Z}$ is zero,  is reduced to that of Gross-Keating.  
We refer Section \ref{sectionin} for the notation and the detailed explanation. 
For any geometric point $\xi$ on $\mathcal{Z}$ let us define the local intersection multiplicity at $\xi$:   
$e(\mathcal{Z},\xi)={\rm length}_{\Z_{(p)}}\mathcal{O}_{\mathcal{Z},\xi}.$  
Similarly for any positive definite half-integral symmetric matrix $T$ over $\Z_{(p)}$ of size $n$ one can also define 
the intersection $\mathcal{Z}(T)_{\F_p}$ of 
some special cycles associated to $T$ on $\mathcal{M}_{\F_p}$. 
For any geometric $\xi$ on $\mathcal{Z}(T)_{\F_p}$ let us define the local intersection multiplicity at $\xi$:   
$e(\mathcal{Z}(T)_{\F_p},\xi)={\rm length}_{\F_p}\mathcal{O}_{\mathcal{Z}(T)_{\F_p},\xi}.$  
Let $(a_0,\ldots,a_{n})$ be the GK-invariant of $T'\otimes \Z_p$. 
It turns out that $e(\mathcal{Z},\xi)$ depends only on $(a_0,\ldots,a_{n})$ and therefore we may put 
$e(\mathcal{Z},\xi)=e(a_0,\ldots,a_{n}).$ 
Further when $\mathcal{Z}(T')$ is of dimension zero, the initial integer $a_0$ has to be zero and we will see that 
$e(a_0,\ldots,a_{n})=\alpha_p(0,\ldots,0,a_1,\ldots,a_{n})$ where the zeros in front of $a_1$ are added until 
the number of all entries are 3.

In this context, we obtain the following result:

\begin{Thm}\label{main1}(Proposition \ref{special2} and Theorem \ref{main-ind}) Assume that $p$ is odd and $0\leq n \leq 3$.
\begin{enumerate}  
\item Assume that $\mathcal{Z}(T)$ is of dimension zero. Then the local intersection multiplicity $e(\mathcal{Z},\xi)$ has the same inductive formula which 
are parallel to  the inductive formula induced from Equation $\left(\ref{eqintro2}\right)$. 
\item Assume that $T$ represents 1. Then $\mathcal{Z}(T)_{\F_p}$ is of dimension zero and for any geometric point $\xi$ on 
$\mathcal{Z}(T)_{\F_p}$,   
the local intersection multiplicity $e(\mathcal{Z}(T)_{\F_p},\xi)$ is described in terms of the derivative of the Siegel series for a suitable anisotropic quadratic lattice of rank 2.
\end{enumerate}
\end{Thm}

Since we  have a better understanding on the latter objects that are Siegel series, it is natural to ask if our comparison argument between two inductive formulas in Gross-Keating's case can be extended to the general case.

\subsection{Speculation}\label{specul}
\subsubsection{} 
As Kudla expected, the local intersection multiplicity in higher dimensional case of the arithmetic intersection  on $\mathrm{GSpin}(n,2)$ Shimura varieties is believed to match with the derivative of the Siegel series for a suitable quadratic lattice.
The comparison results of Theorem \ref{mainthmintro} and Theorem \ref{main1}.(1) seem to imply that 
 there should be an inductive formula in geometric side which is parallel to that on the Siegel series side.
Thus our inductive formula of the Siegel series given in Theorem \ref{main0} would be the inductive formula that the local intersection multiplicity is expected to satisfy with.

Recently, Chao Li and Wei Zhang proved the Kudla-Rapoport conjecture  in \cite{LZ}, which is a unitary version of the above Kudla's conjecture.
Their proof relies on the principle  to compare inductive structures of both sides, rather than to compute them directly.

\subsubsection{}
Since we have an interpretation of the local intersection multiplicity on (the supersingular locus of) the special fiber of  $\mathrm{GSpin}(n,2)$ Shimura varieties in terms of the Siegel series in Theorem \ref{main1}, 
we will be able to relate the intersection numbers on the special fiber of $\mathrm{GSpin}(n,2)$ Shimura varieties with the sum of the Fourier coefficients of the Siegel-Eisenstein series with suitable weight and degree.
We expect that 
the  intersection number of the special cycles at the special fiber of  $\mathrm{GSpin}(n,2)$-Shimura variety
is independent of $p$ (possibly away from bad primes), which turns to be the sum of the Fourier coefficients of the Siegel-Eisenstein series. 
This observation is parallel to Theorems \ref{new}-\ref{new!!}. 
This would imply that an associated arithmetic intersection  is flat over $\mathbb{Z}$ (possibly away from bad primes).
 

\subsection{Organizations}
We will organize this paper as follows. 
After fixing notations in Section \ref{sectionnss},
we will derive a conceptual study of the Siegel series in Sections \ref{sectionld}-\ref{sectionifss}.
In Section  \ref{sectionssaql}, we will explain a refined formulation of the Siegel series for anisotropic quadratic lattices over $\mathbb{Z}_p$.
Section  \ref{reGK} is devoted to compare both sides of the local intersection multiplicity  of \cite{GK1} and the Siegel series.
In Sections \ref{sectionapp}-\ref{sectionin}, we will explain two applications in the context of   intersection numbers (or multiplicities) over a finite field.
In Appendix, we list up explicit examples for the intersection numbers related to Section \ref{sectionapp}.  \\

\textbf{Acknowledgments.} We would like to express our deep appreciation to Professors T. Ikeda and H. Katsurada for many fruitful discussions and suggestions.
We also thank Professors B. Conrad, B. Gross, B. Howard,   R. Schulze-Pillot, and W. Zhang 
for helpful discussions and corrections in Theorems \ref{new} and \ref{mainthm}. 
Special thanks are own to Professor S. Yokoyama for computing intersection numbers for modular correspondences over both  a finite field
and the complex  field in Appendix and also to Professor Chul-hee Lee for pointing out our mistakes on the table in Appendix.  
Finally we thank the referee for helpful suggestions and comments which substantially helped with the presentation of  our paper.


\section{Notations}\label{sectionnss}
\begin{itemize}
\item Let $F$ be a finite field extension of $\mathbb{Q}_p$ with $\mathfrak{o}$  its ring of integers and $\kappa$  its residue field.
Let $\pi$ be a uniformizer in $\mathfrak{o}$.
Let $f$ be the cardinality of the finite field $\kappa$.

\item For an element $x\in F$, the exponential order of $x$ with respect to the maximal ideal in $\mathfrak{o}$ is written by $\mathrm{ord}(x)$.

\item Let $e=\mathrm{ord}(2)$.
Thus if $p$ is odd, then $e=0$.

\item We consider an $\mathfrak{o}$-lattice $L$ with a quadratic form $q_L : L \rightarrow \mathfrak{o}$.
Here, an $\mathfrak{o}$-lattice means a finitely generated free $\mathfrak{o}$-module. 
Such a quadratic form $q_L$ is called \textit{an integral quadratic form} and 
such a pair  $(L, q_L)$  is called  \textit{a quadratic lattice}.
We sometimes say that $L$ is a quadratic lattice by omitting $q_L$, if this does not cause confusion or ambiguity.
Similarly, we define a quadratic space $(V, q_L\otimes_{\mathfrak{o}}1)$ for $V=L\otimes_{\mathfrak{o}}F$ and
sometimes say that $V$ is a quadratic space.
If there is no ambiguity, then we simply use $q_L$, rather than $q_L\otimes 1$,  to stand for the quadratic form defined on $V$. 
We assume that $V$ is nondegenerate with respect to the quadratic form $q_L\otimes 1$.
More generally, for a commutative $\mathfrak{o}$-algebra $R$, we define a quadratic $R$-module $L\otimes_{\mathfrak{o}} R$ in the same manner.
We sometimes call it a quadratic $R$-lattice.

\item For a  quadratic lattice $(L, q_L)$ over $\mathfrak{o}$,
the quadratic form $\bar{q}_L$ on $L\otimes_\mathfrak{o}\kappa$ is defined to be $q_L~mod~\pi$.

\item For two $\mathfrak{o}$-lattices $L$ and $M$,
assume that $L$ and $M$ have the same rank and that $L\supseteq M$.
Then we denote by $[L:M]$ the length of the torsion module $L/M$ so that the cardinality of $L/M$ is $f^{[L:M]}$.

\item The fractional ideal generated by $q_L(X)$ as $X$ runs through $L$ will be called the \textit{norm} of $L$ and written $N(L)$.

\item For matrices $X$ and $Y$ with entries in $F$, the matrix $\begin{pmatrix} X& 0 \\ 0 & Y \end{pmatrix}$ is denoted by $X\bot Y$
 and is called the orthogonal sum of $X$ and $Y$.

\item Let $(a_1, \cdots, a_m)$ and $(b_1, \cdots, b_n)$ be  non-decreasing sequences consisting of non-negative integers in $\mathbb{Z}$.
Then $(a_1, \cdots, a_m)\cup  (b_1, \cdots, b_n)$ is defined as the non-decreasing sequence  $(c_1, \cdots, c_{n+m})$
such that  $\{c_1, \cdots, c_{n+m}\}=\{a_1, \cdots, a_m\}\cup  \{b_1, \cdots, b_n\}$ as multisets.


\item For $\underline{a}=(a_1, \cdots, a_n)$ with each $a_i$ an element of $\mathbb{Z}$,
the sum $a_1+\cdots +a_n$ is denoted by $|\underline{a}|$.

\item For $\underline{a}=(a_1, \cdots, a_n)$ with each $a_i$ an element of $\mathbb{Z}$,
the first $m$-tuple $(a_1, \cdots, a_m)$ with $m\leq n$ is denoted  by $\underline{a}^{(m)}$.

\item Let $H=\begin{pmatrix} 0& 1/2 \\ 1/2 & 0 \end{pmatrix}$ and let $H_k$ be the orthogonal sum of the $k$-copies of $H$.
Then $H_k$ defines a quadratic lattice of rank $2k$.
We denote such a lattice by  $(H_k, q_k)$, if this does not cause confusion or ambiguity.
Similarly we define the quadratic space $(W_k,q_k\otimes 1)$, where $W_k=H_k\otimes_{\mathfrak{o}}F$, and sometimes say that $W_k$ is a quadratic space.
If there is no ambiguity, then we simply use $q_k$, rather than $q_k\otimes 1$,  to stand for the quadratic form defined on $W_k$. 

\item For a symmetric matrix $B$ of size $n\times n$, we say that 
 $B$ is  half-integral over $\mathfrak{o}$ if 
each non-diagonal entry  multiplied by $2$ and each diagonal entry of $B$ are in $\mathfrak{o}$.
We sometimes say that $B$ is half-integral, by omitting `over $\mathfrak{o}$', if it does not cause ambiguity.

\item We say that $B$ is non-degenerate if the determinant of $B$ is nonzero.

\item For $U\in \mathrm{GL}_n(\mathfrak{o})$ and $B$,  where $B$ is a  half-integral symmetric matrix over $\mathfrak{o}$  of size $n\times n$,
we set  $B[U]={}^tUBU$.
Here, ${}^tU$ is the matrix transpose of $U$.

\item We say that two half-integral matrices $B$ and $B'$ are $equivalent$ if $B'=B[U]$ for a certain $U\in \mathrm{GL}_n(\mathfrak{o})$.

\item For  two quadratic $R$-lattices $L$ and $L'$, where $R$ is a commutative $\mathfrak{o}$-algebra, we say that an $R$-linear map $f : L \rightarrow L'$ is \textit{isometry} if it is injective and preserves the associated quadratic forms, i.e. $q_L(x)=q_{L'}(f(x))$ for any $x\in L$.

\item 
Assume that $k\geq n$, where  $n$ is the rank of $L$.
We define $\mathrm{O}_{\mathfrak{o}}(L, H_k)$ as the affine scheme defined over $\mathfrak{o}$ such that
$\mathrm{O}_{\mathfrak{o}}(L, H_k)(R)$, the set of $R$-points of $\mathrm{O}_{\mathfrak{o}}(L, H_k)$ for any commutative $\mathfrak{o}$-algebra $R$,
is  the set of $R$-linear maps (not necessarily injective) from $L\otimes R$
to  $H_k\otimes R$ preserving the associated quadratic forms.
If $R$ is a flat $\mathfrak{o}$-domain, then $\mathrm{O}_{\mathfrak{o}}(L, H_k)(R)$ is the set of isometries (i.e. injective) from the quadratic $R$-lattice $L\otimes R$ to the quadratic $R$-lattice $H_k\otimes R$, which will be proved in the next lemma.


\end{itemize}

\begin{Lem}\label{lemgeneric}
Assume that $k\geq n$, where  $n$ is the rank of $L$.
If $R$ is a flat $\mathfrak{o}$-domain, then an $R$-linear map from  $L\otimes R$ to $H_k\otimes R$ preserving the associated quadratic forms is injective.
Thus the generic fiber of $\mathrm{O}_{\mathfrak{o}}(L, H_k)$, denoted by $\mathrm{O}_{F}(V, W_k)$ with $V=L\otimes_{\mathfrak{o}}F$ and $W_k=H_k\otimes_{\mathfrak{o}}F$, represents the set of isometries from the quadratic space $V$ to the quadratic space $W_k$.
\end{Lem}

\begin{proof}
Let $\varphi: L\otimes R\rightarrow H_k\otimes R$ be an $R$-linear map preserving the associated quadratic forms for a flat $\mathfrak{o}$-domain $R$.
We choose $v\in L\otimes R$ such that $\varphi (v)=0$. Our goal is  to show that $v=0$.

Assume that $v\neq 0$ in $L\otimes R$. 
Let $R_0$ be the quotient field of $R$. Note that the characteristic of  $R_0$ is $0$ since $R$ is flat over $\mathfrak{o}$.

If we let $\tilde{v}=v\otimes 1\in L\otimes R_0$, then $\tilde{v}$ is nonzero. Thus we can choose a basis, say $\mathcal{B}$, of an $R_0$-vector space 
$L\otimes R_0$ of dimension $n$ involving $\tilde{v}$.
We may assume that the last vector in $\mathcal{B}$ is $\tilde{v}$.

We write $\tilde{\varphi}=\varphi\otimes 1 : L\otimes R_0 \rightarrow H_k\otimes R_0$ such that $\tilde{\varphi}(\tilde{v})=0$. 
If we express an $R_0$-linear map $\tilde{\varphi}$ as a matrix $T$ of size $(2k \times n)$ with respect to a basis $\mathcal{B}$, then the last column vector of $T$ is zero.
We now consider the following matrix equation:
\[
q_L={}^tT\cdot q_k\cdot T.
\]
Here, $q_L$ (respectively $q_k$) is the symmetric matrix associated to $L\otimes R_0$ (respectively $H_k\otimes R_0$) with suitable sets of basis.
Thus both $q_L$ and $q_k$ are nondegenerate.
This  contracts to the given setting since the determinant of $q_L$ is nonzero, whereas that of the right hand side is $0$.
Therefore $\tilde{v}=0$ which implies that $v=0$. 
\end{proof}

Let $B$ be a non-degenerate half-integral symmetric matrix over $\mathfrak{o}$ of size $n \times n$.
We will define the Gross-Keating invariant for $B$ below.
The definition is taken from \cite{IK1}.

\begin{Def}[Definitions 0.1 and 0.2 in \cite{IK1}]\label{gkdef}
\begin{enumerate}
\item 
We express $B=\begin{pmatrix}b_{ij}\end{pmatrix}$.
Let $S(B)$ be the set of all non-decreasing sequences $(a_1, \cdots, a_n)\in \mathbb{Z}^n_{\geq 0}$ such that
\[
\begin{array}{l l}
   \mathrm{ord}(b_{ii}) \geq a_i   & \quad    \text{($1 \leq i \leq n$)};\\
   \mathrm{ord}(2b_{ij}) \geq (a_i+a_j)/2   & \quad    \text{($1 \leq i \leq j \leq n$)}.
    \end{array}\]
   Put
   \[\mathbf{S}(\{B\})=\bigcup_{U\in \mathrm{GL}_n(\mathfrak{o})}S(B[U]).\]
The Gross-Keating invariant $\mathrm{GK}(B)$ of $B$ is the greatest element of $\mathbf{S}(\{B\})$ with respect to the lexicographic order
$\succeq$ on $\mathbb{Z}^n_{\geq 0}$.
Here, the lexicographic order $\succeq$ on $\mathbb{Z}^n_{\geq 0}$  is the following (cf. the paragraph following Definition 0.1 of \cite{IK1}).
Choose two elements $(a_1, \cdots, a_n)$ and $(b_1, \cdots, b_n)$ in $\mathbb{Z}^n_{\geq 0}$.
Let $i$ be the first integer over which $a_i$ differs from $b_i$ (so that $a_j=b_j$ for any $j<i$).
If $a_i>b_i$, then we say that $(a_1, \cdots, a_n) \succ (b_1, \cdots, b_n)$.
Otherwise, we say that $(a_1, \cdots, a_n) \prec (b_1, \cdots, b_n)$.

\item The symmetric matrix $B$ is called $optimal$ if $\mathrm{GK}(B)\in S(B)$.

\item If $B$ is a symmetric matrix associated to a quadratic lattice $(L, q_L)$, then  $\mathrm{GK}(L)$, called the Gross-Keating invariant of  $(L, q_L)$, is defined by $\mathrm{GK}(B)$.
$\mathrm{GK}(L)$ is independent of the choice of a matrix $B$.

\end{enumerate}
\end{Def}

It is known that the set $\mathbf{S}(\{B\})$ is  finite (cf. \cite{IK1}), which explains well-definedness of $\mathrm{GK}(B)$.
We can also see that $\mathrm{GK}(B)$ depends on the equivalence class of  $B$.
In general, it is a difficult question to check whether or not a given matrix $B$ is optimal.
Ikeda and Katsurada introduced so-called `reduced form' associated to $B$ and  showed that it is optimal.
We use a reduced form several times in this paper and thus
 provide its detailed definition through the following series of definitions \ref{def3.0}-\ref{def3.2}.
 They are taken from \cite{IK1} and \cite{IK2} for synchronization.
 
In \cite{IK1}, they defined a reduced form when $p=2$.
However, their definition and main theorems  hold for any $p$, which were explained in the initial version of their paper posted on arXiv. 
Thus, in the following, we will not make restriction on $p$, unless otherwise stated.

\begin{Def}[Definition 3.1 in \cite{IK2}]\label{def3.0}
Let $\underline{a}=(a_1, \cdots, a_n)$ be a non-decreasing sequence of non-negative integers.
Write $\underline{a}$ as
\[\underline{a}=(\underbrace{m_1, \cdots, m_1}_{n_1}, \cdots, \underbrace{m_r, \cdots, m_r}_{n_r} )\]
with $m_1<\cdots <m_r$ and $n=n_1+\cdots + n_r$.
For $s=1, 2, \cdots, r$, put
\[n_s^{\ast}=\sum_{u=1}^{s}n_u,\]
and
\[ I_s=\{ n_{s-1}^{\ast}+1, n_{s-1}^{\ast}+2, \cdots, n_{s}^{\ast}\}. \]
Here, we let $n_0^{\ast}=0$.
\end{Def}

Let $\mathfrak{S}_n$ be the symmetric group of degree $n$.
Let $\sigma\in \mathfrak{S}_n$ be an involution i.e. $\sigma^2=id$.

\begin{Def}[Definition 3.1 in \cite{IK1}]\label{def3.1}
For a non-decreasing sequence of non-negative integers $\underline{a}=(a_1, \cdots, a_n)$, we set
\[\mathcal{P}^0=\mathcal{P}^0(\sigma)=\{i|1\leq i \leq n, i=\sigma(i)\}, \]
\[\mathcal{P}^+=\mathcal{P}^+(\sigma)=\{i|1\leq i \leq n, a_i>a_{\sigma(i)}\}, \]
\[\mathcal{P}^-=\mathcal{P}^-(\sigma)=\{i|1\leq i \leq n, a_i<a_{\sigma(i)}\}. \]
We say that an involution $\sigma\in \mathfrak{S}_n$ is an $\underline{a}$-admissible involution if the following three conditions are satisfied:
\begin{enumerate}
\item $\mathcal{P}^0$ has at most two elements.
If $\mathcal{P}^0$ has two distinct elements $i$ and $j$, then
$a_i\not\equiv a_j$ mod $2$, and
\[a_i=\mathrm{max}\{a_j| j\in \mathcal{P}^0\cup \mathcal{P}^+,   a_j\equiv a_i \text{ mod } 2   \}.  \]

\item For $s=1, \cdots, r$, we have
\[\#(I_s\cap \mathcal{P}^+)\leq 1, ~~~ \#(I_s\cap \mathcal{P}^-)+\#(I_s\cap \mathcal{P}^0)\leq 1.\]

\item If $i\in \mathcal{P}^-$, then
\[a_{\sigma(i)}=\mathrm{min}\{a_j | j\in \mathcal{P}^+, ~ a_j>a_i,  a_j\equiv a_i \text{ mod } 2   \}.  \]
Similarly, if $i\in \mathcal{P}^+$, then
\[a_{\sigma(i)}=\mathrm{max}\{a_j | j\in \mathcal{P}^-, ~ a_j<a_i,  a_j\equiv a_i \text{ mod } 2   \}.  \]



\end{enumerate}

\end{Def}

\begin{Def}[Definition 3.2 in \cite{IK1}]\label{def3.2}
Write $B=\begin{pmatrix}b_{ij}\end{pmatrix}$.
Let $\underline{a}\in S(B)$.
Let $\sigma\in \mathfrak{S}_n$ be an $\underline{a}$-admissible involution.
We say that $B$ is a reduced form of GK-type $(\underline{a}, \sigma)$ if the following conditions are satisfied:
\begin{enumerate}
\item If $i \notin \mathcal{P}^0$, $j=\sigma(i)$, and $a_i\leq a_j$, then
\[\mathrm{GK}\begin{pmatrix}\begin{pmatrix}b_{ii} & b_{ij}\\ b_{ij}&b_{jj}\end{pmatrix}\end{pmatrix}=(a_i, a_j).\]
Note that if $p=2$ then this condition is equivalent to the following condition (by Proposition 2.3 of \cite{IK1}).
\[\left\{
  \begin{array}{l l}
  \mathrm{ord}(2b_{ij})=\frac{a_i+a_{j}}{2}   & \quad    \text{if $i\notin \mathcal{P}^0$, $j=\sigma(i)$};\\
   \mathrm{ord}(b_{ii})=a_i   & \quad    \text{if $i\in \mathcal{P}^-$}.
    \end{array} \right.\]

\item if $i\in \mathcal{P}^0$, then
\[\mathrm{ord}(b_{ii})=a_i.\]

\item If $j\neq i, \sigma(i)$, then
\[\mathrm{ord}(2b_{ij})>\frac{a_i+a_j}{2}.\]
\end{enumerate}

\end{Def}

\begin{Thm}[Corollary 5.1 in \cite{IK1}]\label{thm5.1}
A reduced form is optimal.
More precisely, if $B$ is a reduced form of GK-type $(\underline{a}, \sigma)$, then
$$\mathrm{GK}(B)=\underline{a}.$$
\end{Thm}


\begin{Rmk}\label{rmk1}
Assume that $p=2$.
In this remark, we will explain the existence of a reduced form and the uniqueness of an involution up to equivalence.
\begin{enumerate}
\item For any given non-decreasing sequence of non-negative integers $\underline{a}=(a_1, \cdots, a_n)$,
there always exists an $\underline{a}$-admissible involution (cf. the paragraph following Definition 3.1 of \cite{IK1}).

\item For any given non-degenerate half-integral symmetric matrix $B$ over $\mathfrak{o}$,
there always exist  a $\mathrm{GK}(B)$-admissible involution $\sigma$ and
a reduced form of GK type $(\mathrm{GK}(B), \sigma)$ which is equivalent to  $B$
(cf. Theorem 4.1 of \cite{IK1}).

\item 
we say that two $\underline{a}$-admissible involutions are equivalent if they are conjugate by an element of $\mathfrak{S}_{n_1}\times \cdots \times \mathfrak{S}_{n_r}$.
Here, we follow the notation introduced in Definition \ref{def3.0} to specify the integers $n_1, \cdots, n_r$.
If $\sigma$ is an $\underline{a}$-admissible involution, then the equivalence class of $\sigma$ is determined by 
\begin{equation}\label{eqset}
\#(\mathcal{P}^+\cap I_s), ~~\textit{  }~~~ \#(\mathcal{P}^-\cap I_s),  ~~\textit{  }~~~  \#(\mathcal{P}^0\cap I_s)
\end{equation}
for $1\leq s \leq r$ (cf. the paragraph following Remark 4.1 in \cite{IK1}).

\item Let $\sigma$ and $\tau$ be  $\mathrm{GK}(B)$-admissible involutions 
associated to reduced forms of  GK types $(\mathrm{GK}(B), \sigma)$ and $(\mathrm{GK}(B), \tau)$, respectively, which are equivalent to a given symmetric matrix $B$.
 Then $\sigma$ and $\tau$ are equivalent (cf. Theorem 4.2 of \cite{IK1}).
 Therefore, the above sets in (\ref{eqset}) for $B$ are independent of the choice of a 
 $\mathrm{GK}(B)$-admissible involution with a reduced form.
\end{enumerate}
\end{Rmk}

We list a few  facts about the Gross-Keating invariant below.

\begin{Rmk}\label{rgk}
\begin{enumerate}
\item If $p$ is odd, then a diagonal matrix, whose diagonal entries are $u_i\pi^{a_i}$ with  $u_i\in \mathfrak{o}^{\times}$ and $a_i
\leq a_j$ for $i< j$, is a reduced form  (cf. Remark 1.1 of \cite{IK1}) and the Gross-Keating invariant is $(a_1, \cdots, a_n)$. 
Note that any half-integral symmetric matrix is diagonalizable for $p$ odd. 

\item For the half-integral symmetric matrix $H_k$ of rank $2k$, we have
$$\mathrm{GK}(H_k)=(0, \cdots, 0).$$

\item 
If there is an isometry from $(L, q_L)$ of rank $n$ to $(H_k, q_k)$,
then $n\leq 2k$ and
$$\mathrm{GK}(L) \succeq \mathrm{GK}(H_k)^{(n)}=(0, \cdots, 0)$$
by Lemma 1.2 of \cite{IK1}.

\item The first integer of $\mathrm{GK}(L)$ is 
the exponential order of a generator of $N(L)$ (cf. Lemma B.1 of \cite{Y}).

\item Consider $L \subseteq L' \subseteq V$ such that  $[L':L]=b$.
Then
\[|\mathrm{GK}(L')|=|\mathrm{GK}(L)|-2b\]
by Theorem 0.1 of \cite{IK1}.
Here, $|\mathrm{GK}(L)|=a_1+\cdots +a_n$ for $\mathrm{GK}(L)=(a_1, \cdots, a_n)$.
\end{enumerate}
\end{Rmk}

\section{Local densities}\label{sectionld}
The Siegel series of a quadratic lattice $(L, q_L)$ can be defined in terms of the local density associated to two quadratic lattices $(L, q_L)$ and $(H_k, q_k)$ (cf. Definition \ref{defse}). 
From this section to the end, we assume that $k\geq n$, where $n$ is the rank of $L$.
The purpose of this section is to reformulate the local density (and the Siegel series) in terms of certain lattice counting problem conceptually, whose explicit form  is given in Theorem \ref{eqldf} and Corollary  \ref{rmkse}.

\subsection{Local density and primitive local density}\label{subsectionldpld}
We define the following notions:
\[
\left\{
  \begin{array}{l}
  \textit{$\mathfrak{Q}$: the $F$-vector space of quadratic forms defined on $V$};\\
  \textit{$\mathfrak{M}$: the set of $F$-linear maps from $V$ to $W_k$};\\
  \textit{$\mathfrak{M}_L$: the set of  $\mathfrak{o}$-linear maps from $L$ to $H_k$};\\
  \textit{$\mathfrak{Q}_L$: the free $\mathfrak{o}$-module of integral quadratic forms defined on $L$}.
    \end{array} \right.
\]
Here, we remind that $V=L\otimes_{\mathfrak{o}}F$
and $W_k=H_k\otimes_{\mathfrak{o}}F$.

Regarding  $\mathfrak{M}$ and $\mathfrak{Q}$ as varieties over $F$,
let $\omega_{\mathfrak{M}, L}$ and $\omega_{\mathfrak{Q}, L}$ be nonzero,  translation-invariant forms on   $\mathfrak{M}$ and $\mathfrak{Q}$,
 respectively, with normalizations
$$\int_{\mathfrak{M}_L}|\omega_{\mathfrak{M}, L}|=1 \mathrm{~and~}  \int_{\mathfrak{Q}_L}|\omega_{\mathfrak{Q}, L}|=1.$$
Let $\mathfrak{M}^{\ast}$ be the set of injective linear maps from $V$ to $W$.
We can also regard $\mathfrak{M}^{\ast}$ as an open subvariety of $\mathfrak{M}$.
Thus $\mathfrak{M}^{\ast}$ is  a (not necessarily affine) nonsingular variety over $F$. 
Define a map $\rho : \mathfrak{M}^{\ast} \rightarrow \mathfrak{Q}$ by $\rho(m)=q_k\circ m$.
Here $q_k$ is the  quadratic form defined on $W_k$.
Then the inverse image of $q_L\otimes_{\mathfrak{o}}1$
 along the map $\rho$ is $\mathrm{O}_F(V, W_k)$, which 
represents the set of isometries from the quadratic space $V$ to the quadratic space $W_k$
(cf. Lemma \ref{lemgeneric}).
Here $q_L$ is the quadratic form defined on $L$.
It is easy to see that the morphism $\rho$ is representable as a morphism of schemes over $F$ and smooth by showing the surjectivity of the differential of $\rho$ over the Zariski tangent space on any closed point.

\begin{Def}\label{diff}
We will define a differential $\omega_{L}^{\mathrm{ld}}$ on $\mathrm{O}_F(V, W_k)$
associated to $\omega_{\mathfrak{M}, L}$ and $\omega_{\mathfrak{Q}, L}$. 
This is taken from Section 3 of \cite{GY}. 
Smoothness of the morphism $\rho : \mathfrak{M}^{\ast} \rightarrow \mathfrak{Q}$ induces the following short exact sequence of locally free sheaves on $\mathfrak{M}^{\ast}$ (cf. Proposition 5 of \cite{BLR}):
\begin{equation*}\label{eqshort}
0\rightarrow \rho^{\ast}\Omega_{\mathfrak{Q}\slash F}\rightarrow \Omega_{\mathfrak{M}^{\ast}\slash F} \rightarrow \Omega_{\mathfrak{M}^{\ast}\slash \mathfrak{Q}} \rightarrow 0.
\end{equation*}

This gives rise to an isomorphism
\begin{equation*}\label{eqtop}
\rho^{\ast}\left(\bigwedge\limits^{\mathrm{top}}\Omega_{\mathfrak{Q}\slash F}\right)\otimes
\bigwedge\limits^{\mathrm{top}}\Omega_{\mathfrak{M}^{\ast}\slash \mathfrak{Q}}\simeq 
\bigwedge\limits^{\mathrm{top}}\Omega_{\mathfrak{M}^{\ast}\slash F}.
\end{equation*}

Let $\omega_{L}\in \bigwedge\limits^{\mathrm{top}}\Omega_{\mathfrak{M}^{\ast}\slash \mathfrak{Q}}(\mathfrak{M}^{\ast})$ be such that 
$\rho^{\ast}\omega_{\mathfrak{Q}, L}\otimes \omega_{L}=\omega_{\mathfrak{M}, L}|_{\mathfrak{M^{\ast}}}$. 
We then denote by $\omega_{L}^{\mathrm{ld}}$  the restriction of $\omega_{L}$ to $\mathrm{O}_F(V, W_k)$. 
We sometimes write $\omega_{L}^{\mathrm{ld}}=\omega_{\mathfrak{M}, L}/\rho^{\ast}\omega_{\mathfrak{Q}, L}$.
\end{Def}

\begin{Lem}\label{volumeform}
We have the following equation: 
\[
\int_{\mathrm{O}_{\mathfrak{o}}(L, H_k)(\mathfrak{o})}|\omega_L^{\mathrm{ld}}|=\lim_{N\rightarrow \infty} f^{-N\cdot dim \mathrm{O}_F(V, W)}
\#\mathrm{O}_{\mathfrak{o}}(L, H_k)(\mathfrak{o}/\pi^N \mathfrak{o}),
\]
where the limit stabilizes for $N$ sufficiently large.
Here, $dim \mathrm{O}_F(V, W_k)=dim \mathrm{O}(W_k)-dim \mathrm{O}(V^{\perp})=
dim \mathfrak{M}-dim \mathfrak{Q}=2kn-(n^2+n)/2$.
\end{Lem}

The lemma is well-known and a generalization of Lemma 3.4 of \cite{GY}. 
If $\mathfrak{M}$ has a group structure, then the proof  is explained in  page 5 of \cite{Han} or in pages 119-120 of  \cite{Tam}.
The general case is explained in \cite{Yu1}, which is not available on the internet.
Since the proof of the general case is not different from the case handled in  \cite{Han} and  \cite{Tam},
 we sketch a main idea of the proof.

\begin{proof}
The right hand side can be identified with the following:
\[
\lim_{U_q\rightarrow q_L}\frac{\int_{\rho^{-1}(U_q)\cap \mathfrak{M}_L}|\omega_{\mathfrak{M}, L}|}{\int_{U_q}|\omega_{\mathfrak{Q}, L}|},
\]
where $U_q$ is an open neighborhood of $q_L$ of the form $q_L+\pi^N\mathfrak{Q}_L$ for $N> 0$ so that 
$\int_{q_L+\pi^N\mathfrak{Q}_L}|\omega_{\mathfrak{Q}, L}|=f^{-N\cdot dim \mathfrak{Q}}$ and 
$\int_{\rho^{-1}(q_L+\pi^N\mathfrak{Q}_L)\cap \mathfrak{M}_L}|\omega_{\mathfrak{M}, L}|=f^{-N\cdot dim \mathfrak{M}}\cdot \#\mathrm{O}_{\mathfrak{o}}(L, H_k)(\mathfrak{o}/\pi^N \mathfrak{o})$.
It is well known that  isomorphism classes of  quadratic lattices are locally constant. 
In other words, any quadratic form contained in $q_L+\pi^N\mathfrak{Q}_L$, for sufficiently large integer $N$, is isometric to $q_L$ by Theorem 2 of \cite{Dur}. 
Thus Fubini's theorem yields that the fraction of these two integrals with sufficiently large integer $N$ is the same as the left hand side.
\end{proof}

\begin{Def}\label{deflocaldensity}
The local density associated to the pair of two quadratic lattices $L$ and $H_k$, denoted by $\alpha(L, H_k)$, is defined as 
$$\alpha(L, H_k)=\int_{\mathrm{O}_{\mathfrak{o}}(L, H_k)(\mathfrak{o})}|\omega_L^{\mathrm{ld}}|.
$$
\end{Def}

We define the subfunctor $\mathrm{O}_{\mathfrak{o}}^{prim}(L, H_k)$ of $\mathrm{O}_{\mathfrak{o}}(L, H_k)$ such that
 $\mathrm{O}_{\mathfrak{o}}^{prim}(L, H_k)(R)$, the set of $R$-points for a commutative $\mathfrak{o}$-algebra $R$, is the set of elements in $\mathrm{O}_{\mathfrak{o}}(L, H_k)(R)$ whose at least one $n\times n$-minor, as a linear map from 
 $L\otimes_{\mathfrak{o}}R$ to $H_k\otimes_{\mathfrak{o}}R$, is a unit in $R$.
In particular, if $R=\mathfrak{o}$, then 
$\mathrm{O}_{\mathfrak{o}}^{prim}(L, H_k)(\mathfrak{o})$ is the set of elements in $\mathrm{O}_{\mathfrak{o}}(L, H_k)(\mathfrak{o})$ whose reduction modulo $\pi$ is injective   from $L\otimes_{\mathfrak{o}}\kappa$ to $H_k\otimes_{\mathfrak{o}}\kappa$.
Each element in $\mathrm{O}_{\mathfrak{o}}^{prim}(L, H_k)(R)$ is called \textit{a primitive isometry}.
We will show that $\mathrm{O}_{\mathfrak{o}}^{prim}(L, H_k)$ is an open (not necessarily affine) subscheme of $\mathrm{O}_{\mathfrak{o}}(L, H_k)$ (cf. Corollary \ref{corsmooth}) and that  $\mathrm{O}_{\mathfrak{o}}^{prim}(L, H_k)(\mathfrak{o})$ is open in $\mathrm{O}_{\mathfrak{o}}(L, H_k)(\mathfrak{o})$ in terms of inherent $p$-adic topology (cf. Lemma \ref{padicopen}).

Consider a lattice $L'$  in $V$ containing $L$.
Let $q_{L'}$ be the quadratic form defined on $L'$, whose restriction to $L$ is the same as $q_L$.
Here, $q_{L'}$ is abuse of notation since it may not be an integral quadratic form.
We can regard  $\mathrm{O}_{\mathfrak{o}}^{prim}(L', H_k)(\mathfrak{o})$ as a  subset of $\mathrm{O}_{\mathfrak{o}}(L, H_k)(\mathfrak{o})$   induced by the restriction to $L$.
Since any linear map in $\mathrm{O}_{\mathfrak{o}}(L, H_k)(\mathfrak{o})$ is injective (i.e. isometry) by Lemma \ref{lemgeneric}, 
 we have the following stratification on $\mathrm{O}_{\mathfrak{o}}(L, H_k)(\mathfrak{o})$:
\begin{equation}\label{eqldset}
\mathrm{O}_{\mathfrak{o}}(L, H_k)(\mathfrak{o})=\bigsqcup_{L\subseteq L' \subseteq V} \mathrm{O}_{\mathfrak{o}}^{prim}(L', H_k)(\mathfrak{o}).
\end{equation}

In the following lemma, we will show that the norm $N(L')$ of $L'$ should be contained in the ring $\mathfrak{o}$, in order that $\mathrm{O}_{\mathfrak{o}}^{prim}(L', H_k)(\mathfrak{o})$ is nonempty.
This directly implies that the above  is a finite disjoint union.
\begin{Lem}\label{lemnorm}
The condition $N(L')\subseteq \mathfrak{o}$ is equivalent to the existence of a primitive isometry from $(L', q_{L'})$ to $(H_k, q_k)$.
\end{Lem}
\begin{proof}
If we have a primitive isometry from $(L', q_{L'})$ to $(H_k, q_k)$,
then the fact  $N(H_k)=\mathfrak{o}$ yields
$N(L')\subseteq \mathfrak{o}$.

Conversely, 
assume that $N(L')\subseteq \mathfrak{o}$.
The symmetric matrix $B'$ of the quadratic lattice $L'$ with respect to any basis is then  half-integral over $\mathfrak{o}$.
We consider the half-integral symmetric matrix $\begin{pmatrix} 0&\frac{1}{2}\cdot id_n \\ \frac{1}{2}\cdot id_n& B' \end{pmatrix}$.
Here, $id_n$ is the $(n \times n)$-identity matrix, where $n$ is the rank of $L'$.
Note that $\begin{pmatrix} 0&\frac{1}{2}\cdot id_n \\ \frac{1}{2}\cdot id_n& 0 \end{pmatrix}$ is an half-integral symmetric matrix of the quadratic lattice $H_n$.
We claim that  $\begin{pmatrix} 0&\frac{1}{2}\cdot id_n \\ \frac{1}{2}\cdot id_n& B' \end{pmatrix}$ is equivalent to the matrix $\begin{pmatrix} 0&\frac{1}{2}\cdot id_n \\ \frac{1}{2}\cdot id_n& 0 \end{pmatrix}$ over $\mathfrak{o}$.
This induces the existence of a primitive isometry from $(L', q_{L'})$ to $(H_k, q_k)$.

Our claim follows from the matrix equation:
\[
\begin{pmatrix} 0&\frac{1}{2}\cdot id_n \\ \frac{1}{2}\cdot id_n& 0 \end{pmatrix}=
\begin{pmatrix} id_n &0 \\ {}^tX& id_n \end{pmatrix}
\begin{pmatrix} 0&\frac{1}{2}\cdot id_n \\ \frac{1}{2}\cdot id_n& B' \end{pmatrix}
\begin{pmatrix} id_n &X \\ 0& id_n \end{pmatrix}.
\]
Here, $X$ is any matrix with entries in $\mathfrak{o}$
such that $X+{}^tX+2B'=0$.
\end{proof}

This lemma, combined with Remark \ref{rgk} induces the following description of the existence of a primitive isometry in terms of the Gross-Keating invariant:
\begin{Cor}\label{primex}
There exists  a primitive isometry from $(L', q_{L'})$ to $(H_k, q_k)$ if and only if
\[\mathrm{GK}(L') \succeq \mathrm{GK}(H_k)^{(n)}=(0, \cdots, 0).\]

\end{Cor}

\begin{Lem}\label{padicopen}
Assume that $N(L')\subseteq \mathfrak{o}$.
The set $\mathrm{O}_{\mathfrak{o}}^{prim}(L', H_k)(\mathfrak{o})$ is open in $\mathrm{O}_{\mathfrak{o}}(L, H_k)(\mathfrak{o})$ in terms of the $p$-adic topology.
\end{Lem}

\begin{proof}
It suffices to show that 
 $\mathrm{O}_{\mathfrak{o}}(L', H_k)(\mathfrak{o})$ is open in $\mathrm{O}_{\mathfrak{o}}(L, H_k)(\mathfrak{o})$
 and that  $\mathrm{O}_{\mathfrak{o}}^{prim}(L, H_k)(\mathfrak{o})$ is open in $\mathrm{O}_{\mathfrak{o}}(L, H_k)(\mathfrak{o})$.
 
Let us identify $\mathrm{End}_{\mathfrak{o}}(L, H_k)(\mathfrak{o})$ with the set of $2k\times n$-matrices with entries in $\mathfrak{o}$.
Let $d=[L':L]$. Since $L \subset L' \subset \frac{1}{\pi^d} L$,  $\mathrm{End}_{\mathfrak{o}}(L', H_k)(\mathfrak{o})$ contains
$\mathrm{End}_{\mathfrak{o}}(\frac{1}{\pi^d} L, H_k)(\mathfrak{o})$.
The latter is identified with 
 $\pi^d\cdot \mathrm{End}_{\mathfrak{o}}(L, H_k)(\mathfrak{o})$ as a subset of $\mathrm{End}_{\mathfrak{o}}(L, H_k)(\mathfrak{o})$.
Here, $\pi^d\cdot \mathrm{End}_{\mathfrak{o}}(L, H_k)(\mathfrak{o})$ is defined as  $\{\pi^d X, X\in \mathrm{End}_{\mathfrak{o}}(L, H_k)(\mathfrak{o})\}$.
Therefore, as an $\mathfrak{o}$-module,  $\mathrm{End}_{\mathfrak{o}}(L', H_k)(\mathfrak{o})$ has a finite index in $\mathrm{End}_{\mathfrak{o}}(L, H_k)(\mathfrak{o})$ and so is open.

Since  $\mathrm{O}_{\mathfrak{o}}(L', H_k)(\mathfrak{o})$ is the intersection of $\mathrm{End}_{\mathfrak{o}}(L', H_k)(\mathfrak{o})$ and $\mathrm{O}_{\mathfrak{o}}(L, H_k)(\mathfrak{o})$ inside $\mathrm{End}_{\mathfrak{o}}(L, H_k)(\mathfrak{o})$,
it is open in $\mathrm{O}_{\mathfrak{o}}(L, H_k)(\mathfrak{o})$.

For the second claim, 
since  $\mathrm{O}_{\mathfrak{o}}^{prim}(L, H_k)(\mathfrak{o})$ is the intersection of $\mathrm{End}^{prim}_{\mathfrak{o}}(L, H_k)(\mathfrak{o})$ and $\mathrm{O}_{\mathfrak{o}}(L, H_k)(\mathfrak{o})$ inside $\mathrm{End}_{\mathfrak{o}}(L, H_k)(\mathfrak{o})$,
as in the above case,
 it suffices to show that 
$\mathrm{End}^{prim}_{\mathfrak{o}}(L, H_k)(\mathfrak{o})$ is open in $\mathrm{End}_{\mathfrak{o}}(L, H_k)(\mathfrak{o})$.
Here 
$\mathrm{End}^{prim}_{\mathfrak{o}}(L, H_k)(\mathfrak{o})$ consists of endomorphisms from $L$ to $H_k$ whose at least one $n\times n$-minor is a unit in $\mathfrak{o}$.
We consider the reduction map modulo $\pi$ from $\mathrm{End}_{\mathfrak{o}}(L, H_k)(\mathfrak{o})$
to $\mathrm{End}_{\kappa}(L\otimes \kappa, H_k\otimes \kappa)$.
The induced quotient topology on the latter  is  the discrete topology on a finite set.
Let $\mathrm{End}^{prim}_{\kappa}(L\otimes \kappa, H_k\otimes \kappa)$ be the (open) subset of $\mathrm{End}_{\kappa}(L\otimes \kappa, H_k\otimes \kappa)$ consisting of  endomorphisms from $L\otimes \kappa$ to $H_k\otimes \kappa$ whose at least one $n\times n$-minor is nonzero in $\kappa$. Then 
$\mathrm{End}^{prim}_{\mathfrak{o}}(L, H_k)(\mathfrak{o})$ is the inverse image of $\mathrm{End}^{prim}_{\kappa}(L\otimes \kappa, H_k\otimes \kappa)$ under the reduction map. Thus it is open.
\end{proof}

\begin{Prop}
The local density $\alpha(L, H_k)$ is  expressed as the following sum:
\begin{equation}\label{eqld}
\alpha(L, H_k)
=\sum_{ L \subseteq L' \subseteq V} f^{[L':L]\cdot(n+1-2k)}\int_{\mathrm{O}_{\mathfrak{o}}^{prim}(L', H_k)(\mathfrak{o})}|\omega_{L'}^{\mathrm{ld}}|.
\end{equation}

\end{Prop}
\begin{proof}
Since the set $\mathrm{O}_{\mathfrak{o}}^{prim}(L', H_k)(\mathfrak{o})$ is open in $\mathrm{O}_{\mathfrak{o}}(L, H_k)(\mathfrak{o})$
 by Lemma \ref{padicopen},
Equation (\ref{eqldset}) yields the following identity:
\[
\alpha(L, H_k)=\sum_{L\subseteq L' \subseteq V} \int_{\mathrm{O}_{\mathfrak{o}}^{prim}(L', H_k)(\mathfrak{o})}|\omega_L^{\mathrm{ld}}|.
\]
Let $d=[L':L]$. 
Recall that $\omega_{L}^{\mathrm{ld}}=\omega_{\mathfrak{M}, L}/\rho^{\ast}\omega_{\mathfrak{Q}, L}$ and 
 $\omega_{L'}^{\mathrm{ld}}=\omega_{\mathfrak{M}, L'}/\rho^{\ast}\omega_{\mathfrak{Q}, L'}$ in Definition \ref{diff}.
By the normalizations of $\omega_{\mathfrak{M}, L}$ and $\omega_{\mathfrak{Q}, L}$ given at the beginning of this subsection,
we have
\[
\left\{
  \begin{array}{l}
\omega_{\mathfrak{M}, L}=\pi^{2kd}\cdot \omega_{\mathfrak{M}, L'};\\
\omega_{\mathfrak{Q}, L}=\pi^{d(n+1)}\cdot\omega_{\mathfrak{Q}, L'}.
    \end{array} \right.
\]
Thus, $\omega_{L}^{\mathrm{ld}}=\pi^{2kd-(n+1)d}\omega_{L'}^{\mathrm{ld}}$. 
 Equation (\ref{eqld}) follows from this observation.
\end{proof}

\begin{Def}\label{defdecompositionldpld}
We define the primitive local density associated to the quadratic lattices $L$ and $H_k$, denoted by
$\alpha^{prim}(L, H_k)$, as follows:
\[\alpha^{prim}(L, H_k)=\int_{\mathrm{O}_{\mathfrak{o}}^{prim}(L, H_k)(\mathfrak{o})}|\omega_{L}^{\mathrm{ld}}|.\]
\end{Def}

Thus Equation (\ref{eqld}) is written as follows:
\begin{equation}\label{equationldprimld}
\alpha(L, H_k)=
 \sum_{ L \subseteq L' \subseteq V} f^{[L':L]\cdot(n+1-2k)}\cdot \alpha^{prim}(L', H_k). 
 \end{equation}
Here, the sum runs over all quadratic lattices $L'$ such that $N(L')\subseteq \mathfrak{o}$ and thus is finite.
We remark that this formula is  well-known in  classical literatures using different terminologies (cf. Lemma 3 of \cite{K83}).

\subsection{A formula of the primitive local density}\label{subsectionfpld}

We assume that $N(L)\subseteq \mathfrak{o}$ and thus the set $\mathrm{O}_{\mathfrak{o}}^{prim}(L, H_k)(\mathfrak{o})$ is nonempty.
In this subsection, we describe a formula of  the primitive local density $\alpha^{prim}(L, H_k)$ by proving smoothness of 
$\mathrm{O}^{prim}_{\mathfrak{o}}(L, H_k)$   as a scheme defined over $\mathfrak{o}$.


We first explain another description of  $\mathrm{O}^{prim}_{\mathfrak{o}}(L, H_k)$ in terms of the fiber of a certain morphism between two smooth schemes over $\mathfrak{o}$.

Let $\underline{M}_{\mathfrak{o}}(L, H_k)$ be the functor from the category of flat $\mathfrak{o}$-algebras to the category of sets such that 
$\underline{M}_{\mathfrak{o}}(L, H_k)(R)$, the set of $R$-points for a flat $\mathfrak{o}$-algebra $R$, is the set of  $R$-linear maps from $L\otimes_{\mathfrak{o}}R$
to $H_k\otimes_{\mathfrak{o}}R$ by ignoring the associated quadratic forms.
Then the functor $\underline{M}_{\mathfrak{o}}(L, H_k)$ is uniquely represented by a flat $\mathfrak{o}$-algebra which is a polynomial ring over $\mathfrak{o}$ of $2kn$ variables. 
Thus we can now talk of $\underline{M}_{\mathfrak{o}}(L, H_k)(R)$  for any  (not necessarily flat) $\mathfrak{o}$-algebra $R$.

Let $\underline{M}^{\ast}_{\mathfrak{o}}(L, H_k)$ be the subfunctor of $\underline{M}_{\mathfrak{o}}(L, H_k)$ such that
 $\underline{M}^{\ast}_{\mathfrak{o}}(L, H_k)(R)$, the set of $R$-points for a commutative $\mathfrak{o}$-algebra $R$, is the set of  $R$-linear maps from $L\otimes_{\mathfrak{o}}R$
to $H_k\otimes_{\mathfrak{o}}R$ whose at least one  $n\times n$-minor is a unit in $R$.
Then it is easy to see that $\underline{M}^{\ast}_{\mathfrak{o}}(L, H_k)$ is an open  subscheme of $\underline{M}_{\mathfrak{o}}(L, H_k)$
which yields smoothness of $\underline{M}^{\ast}_{\mathfrak{o}}(L, H_k)$
(cf. Section 3.2 of \cite{C1}).
Note that $\underline{M}^{\ast}_{\mathfrak{o}}(L, H_k)$ is not necessarily affine, but has finite affine covers (given by each $n\times n$-minor).
In particular, if $R=\mathfrak{o}$, then 
$\underline{M}^{\ast}_{\mathfrak{o}}(L, H_k)(\mathfrak{o})$ is the set of  $\mathfrak{o}$-linear maps from  $L$ to $H_k$  whose reduction modulo $\pi$ is injective from 
$L\otimes_{\mathfrak{o}}\kappa$ to $H_k\otimes_{\mathfrak{o}}\kappa$.

Let $\underline{Q}_L$ be the affine space of dimension $n(n+1)/2$ defined over $\mathfrak{o}$ such that
$\underline{Q}_L(R)$, the set of $R$-points for a commutative $\mathfrak{o}$-algebra $R$,
is the set of  quadratic forms on $L\otimes_{\mathfrak{o}}R$ whose 
coefficients are in $R$.
Similarly, we define $\underline{Q}_{H_k}$ as the affine space of dimension $2k(2k+1)/2$ defined over $\mathfrak{o}$.


Let $R$ be a flat $\mathfrak{o}$-algebra. As a matrix, each element of $\underline{Q}_L(R)$ is given by a symmetric matrix $(a_{ij})$
of size $n\times n$ such that each nondiagonal entry $a_{ij}$ with $i\neq j$ is of the form $1/2\cdot a_{ij}'$ for $a_{ij}'\in R$
and each diagonal entry $a_{ii}$ is contained in $R$.

Then we consider the following morphism
\[ \underline{Q}_{H_k} \times \underline{M}^{\ast}_{\mathfrak{o}}(L, H_k) \longrightarrow   \underline{Q}_L, (q, m) \mapsto q\circ m.\]
Here, $q\in \underline{Q}_{H_k}(R)$ and $m\in \underline{M}^{\ast}_{\mathfrak{o}}(L, H_k)(R)$ for a flat $\mathfrak{o}$-algebra $R$.
It is easy to see that the above morphism is well-defined and represented by a  morphism of schemes over $\mathfrak{o}$
(cf. the last paragraph of the proof of Theorem 3.4 in \cite{C2}).
The above  morphism induces the morphism
\[\rho : \underline{M}^{\ast}_{\mathfrak{o}}(L, H_k) \longrightarrow   \underline{Q}_L, m \mapsto q_k\circ m.\]

\begin{Thm}\label{thmsmoothness}
The morphism 
$\rho : \underline{M}^{\ast}_{\mathfrak{o}}(L, H_k) \rightarrow   \underline{Q}_L$
is smooth.
\end{Thm}
\begin{proof}
  The theorem follows from Lemma 5.5.1 of \cite{GY} and the following lemma.
  \end{proof}

      \begin{Lem}\label{lemsmoothness}
      The morphism $\rho \otimes \kappa : \underline{M}^{\ast}_{\mathfrak{o}}(L, H_k)\otimes \kappa \rightarrow \underline{Q}_L\otimes \kappa$
      is smooth.
      \end{Lem}

\begin{proof}
 The proof is based on Lemma 5.5.2 in \cite{GY}.
    It suffices to show that, for any $m \in \underline{M}^{\ast}_{\mathfrak{o}}(L, H_k)(\bar{\kappa})$,
    the induced map on the Zariski tangent space $\rho_{\ast, m}:T_m \rightarrow T_{\rho(m)}$ is surjective.
Here,  $\bar{\kappa}$ is the algebraic closure of  $\kappa$.

We introduce     another functor on the category of  flat $\mathfrak{o}$-algebras.
   Define $T(R)$ to be the set of  $(n \times n)$-matrices $y$ such that each entry is of the form $1/2\cdot y_{ij}$ with $y_{ij}\in R$.
Then this functor is represented by an affine space over $\mathfrak{o}$.

We now compute the map $\rho_{\ast, m}$ explicitly.
If we identify $T_m$ with $\underline{M}_{\mathfrak{o}}(L, H_k)(\bar{\kappa})$ and $T_{\rho(m)}$ with $\underline{Q}_L(\bar{\kappa})$, then
 $$\rho_{\ast, m} : T_m \longrightarrow T_{\rho(m)}, X\mapsto m^t\cdot q_k \cdot X+X^t\cdot q_k \cdot m.$$
We explain how to compute  $X\mapsto m^t\cdot q_k \cdot X+X^t\cdot q_k \cdot m$ explicitly.
Firstly, we formally compute $X\mapsto m^t\cdot q_k \cdot X$.
It is of the form $1/2\cdot Y\in T(\bar{\kappa})$, where $Y$ is an $n\times n$-matrix with entries in $\bar{\kappa}$. 
Then we formally compute $1/2\cdot Y+1/2\cdot {}^tY$.
It is of the form $Z$, whose  diagonal entries, denoted by $z_{ii}$'s,  are in $\bar{\kappa}$ and
whose nondiagonal entries are of the form $1/2\cdot z_{ij}$ with $z_{ij}\in \bar{\kappa}$ such that $z_{ij}=z_{ji}$.
More precisely, if we write $Y=(y_{ij})$, then $z_{ii}=y_{ii}$ and $z_{ij}=y_{ij}+y_{ji}$ for $i\neq j$.
Thus $Z$ is an element of $\underline{Q}_L(\bar{\kappa})$.

 To prove the surjectivity of $\rho_{\ast, m} : T_m \longrightarrow T_{\rho(m)}$, it suffices to show the following two statements:
       \begin{itemize}
   \item[(1)] $X \mapsto  m^t\cdot q_k \cdot X$ defines a surjection
   $\underline{M}_{\mathfrak{o}}(L, H_k)(\bar{\kappa}) \rightarrow T(\bar{\kappa})$;
   \item[(2)] $Y \mapsto {}^t Y + Y$ defines a surjection $T(\bar{\kappa}) \rightarrow \underline{Q}_L(\bar{\kappa})$.
\end{itemize}

These two arguments are direct from the construction of $T(\bar{\kappa})$.
\end{proof}

\begin{Cor}\label{corsmooth}
The scheme $\mathrm{O}^{prim}_{\mathfrak{o}}(L, H_k)$ is an open subscheme $\mathrm{O}_{\mathfrak{o}}(L, H_k)$.
In addition, it is   smooth over $\mathfrak{o}$.
\end{Cor}

\begin{proof}
The scheme $\mathrm{O}^{prim}_{\mathfrak{o}}(L, H_k)$ is defined as the fiber of $q_L$ along the smooth morphism $\rho$.
Thus $\mathrm{O}^{prim}_{\mathfrak{o}}(L, H_k)$ is the intersection of $\mathrm{O}_{\mathfrak{o}}(L, H_k)$ and $\underline{M}^{\ast}_{\mathfrak{o}}(L, H_k)$ inside $\underline{M}_{\mathfrak{o}}(L, H_k)$.
The fact that $\underline{M}^{\ast}_{\mathfrak{o}}(L, H_k)$ is an open  subscheme of $\underline{M}_{\mathfrak{o}}(L, H_k)$ yields the first statement. 
We note that $\mathrm{O}^{prim}_{\mathfrak{o}}(L, H_k)$ has  finite affine covers 
given by $n\times n$-minors, each of which is an open subscheme of  $\mathrm{O}_{\mathfrak{o}}(L, H_k)$ as well.

For the second statement,
since  smoothness is stable under  base change, $\mathrm{O}^{prim}_{\mathfrak{o}}(L, H_k)$ is a smooth scheme over $\mathfrak{o}$.
\end{proof}

The  special fiber of $\mathrm{O}^{prim}_{\mathfrak{o}}(L, H_k)$ is $\mathrm{O}_{\kappa}^{prim}(\bar{q}_L, \bar{q}_k)$,
where $\bar{q}_L~\left(= q_L ~mod~\pi\right)$ is a quadratic form defined on $L\otimes_{\mathfrak{o}}\kappa$ and $\bar{q}_k ~\left(= q_k~ mod~\pi\right)$ is a quadratic form defined on $H_k\otimes_{\mathfrak{o}}\kappa$.
In particular, $\mathrm{O}^{prim}_{\mathfrak{o}}(L, H_k)(\kappa)$ is  the set of isometries from the quadratic space $(L\otimes_{\mathfrak{o}}\kappa, \bar{q}_L)$
to the quadratic space $(H_k\otimes_{\mathfrak{o}}\kappa,  \bar{q}_k)$.

We finally state the following formula of the primitive local density. 

\begin{Thm}\label{tpld}
Assume that $N(L)\subseteq \mathfrak{o}$.
Then  the primitive local density $\alpha^{prim}(L, H_k)$ is described as follows:
\[\alpha^{prim}(L, H_k)=f^{-dim \mathrm{O}_F(V, W_k)}\cdot \#\mathrm{O}^{prim}_{\kappa}(\bar{q}_L, \bar{q}_k)(\kappa).  \]
Here, $dim \mathrm{O}_F(V, W_k)=2kn-(n^2+n)/2$
and $\#\mathrm{O}_{\kappa}^{prim}(\bar{q}_L, \bar{q}_k)(\kappa)$
stands for the cardinality of the set $\mathrm{O}_{\kappa}^{prim}(\bar{q}_L, \bar{q}_k)(\kappa)$.
 
\end{Thm}

\begin{proof}
 Theorem \ref{thmsmoothness} yields the following short exact sequence of locally free sheaves on $\underline{M}^{\ast}_{\mathfrak{o}}(L, H_k)$ as in Definition \ref{diff} (cf. Proposition 5 of \cite{BLR}):
\[ 0\longrightarrow \rho^{\ast}\Omega_{\underline{Q}_L/\mathfrak{o}} \longrightarrow \Omega_{\underline{M}^{\ast}_{\mathfrak{o}}(L, H_k)/\mathfrak{o}}
\longrightarrow \Omega_{\underline{M}^{\ast}_{\mathfrak{o}}(L, H_k)/\underline{Q}_L} \longrightarrow 0. \]

Let $\omega_{\underline{M},L}$ (respectively $\omega_{\underline{Q},L}$) be a differential of top degree on $\underline{M}_{\mathfrak{o}}(L, H_k)$ (respectively $\underline{Q}_L$) over $\mathfrak{o}$, whose reduction on the special fiber is nowhere zero. 
Since $\underline{M}^{\ast}_{\mathfrak{o}}(L, H_k)$ is an open subscheme of $\underline{M}_{\mathfrak{o}}(L, H_k)$,
the restriction of $\omega_{\underline{M},L}$ to $\underline{M}^{\ast}_{\mathfrak{o}}(L, H_k)$
 is a differential of top degree on $\underline{M}^{\ast}_{\mathfrak{o}}(L, H_k)$ whose reduction on the special fiber is nowhere zero.


We define $\omega_{L}^{\mathrm{can}}$ to be $\omega_{\underline{M},L}/\rho^{\ast}\omega_{\underline{Q},L}$.
Here,  we refer to Definition \ref{diff} for the notion of $\omega_{\underline{M},L}/\rho^{\ast}\omega_{\underline{Q},L}$.
Then  $\omega_{L}^{\mathrm{can}}$ is a differential  of top degree on $\mathrm{O}^{prim}_{\mathfrak{o}}(L, H_k)$ over $\mathfrak{o}$, whose reduction on the special fiber is nowhere zero.

Since $\underline{M}_{\mathfrak{o}}(L, H_k)(\mathfrak{o})=\mathfrak{M}_L$ and $\underline{Q}_L(\mathfrak{o})=\mathfrak{Q}_L$,
 the differentials $\omega_{\mathfrak{M}, L}$  (respectively $\omega_{\mathfrak{Q}, L}$) and $\omega_{\underline{M},L}$ (respectively $\omega_{\underline{Q},L}$) determine the same measure on $\mathfrak{M}$ (respectively $\mathfrak{Q}$). 
 Here, see the beginning of subsection \ref{subsectionldpld} for the notions of $\omega_{\mathfrak{M}, L}$, $\omega_{\mathfrak{Q}, L}$, $\mathfrak{M}_L$, $\mathfrak{M}$, $\mathfrak{Q}_L$, $\mathfrak{Q}$.
Thus the description of $\omega_{L}^{\mathrm{ld}}$ given in Definition \ref{diff} shows that $\omega_{L}^{\mathrm{ld}}$ and $\omega_{L}^{\mathrm{can}}$ determine the same measure on $\mathrm{O}_{F}(V, W_k)(F)$.
This is parallel to Section 7.1 of \cite{GY}.
Note that a differential of top degree on  $\mathrm{O}^{prim}_{\mathfrak{o}}(L, H_k)$, whose reduction on the special fiber is nowhere zero,  is not unique.
But such two forms determine the same measure on $\mathrm{O}^{prim}_{\mathfrak{o}}(L, H_k)(F)$ ($=\mathrm{O}_{F}(V, W_k)(F)$ by Lemma \ref{lemgeneric}) due to smoothness of the scheme $\mathrm{O}^{prim}_{\mathfrak{o}}(L, H_k)$. 
For more explanation about this, see Section 5 of Chapter 3  in \cite{Pop}.

The primitive local density is the volume of $\mathrm{O}^{prim}_{\mathfrak{o}}(L, H_k)(\mathfrak{o})$ with respect to the measure $|\omega_{L}^{\mathrm{ld}}|$ (cf. Definition \ref{defdecompositionldpld}).
Since $\mathrm{O}^{prim}_{\mathfrak{o}}(L, H_k)(\mathfrak{o})$ is the set of $\mathfrak{o}$-points of the smooth scheme $\mathrm{O}^{prim}_{\mathfrak{o}}(L, H_k)$ and 
$\alpha^{prim}(L, H_k)(\kappa)=\mathrm{O}^{prim}_{\kappa}(\bar{q}_L, \bar{q}_k)(\kappa)$, the formula
 directly follows from Theorem 2.2.5 of \cite{Weil} or Theorem 5.3 of Chapter 3 of \cite{Pop}.
\end{proof}


\subsection{Reformulation  of the local density}\label{subsectionrld}

Using Theorem \ref{tpld} and Corollary \ref{primex}, we can reformulate Equation (\ref{eqld}) of the local density as follows:
\begin{equation}\label{eqld'}
\alpha(L, H_k)=
f^{-2kn+(n^2+n)/2}\cdot \sum_{\substack{L \subseteq L' \subseteq V, \\ \mathrm{GK}(L')\succeq (0, \cdots, 0)}}
f^{[L':L]\cdot(n+1-2k)}\cdot \#\mathrm{O}_{\kappa}^{prim}(\bar{q}_{L'}, \bar{q}_k)(\kappa).
\end{equation}

Recall that $\mathrm{O}_{\kappa}^{prim}(\bar{q}_{L'}, \bar{q}_k)(\kappa)$ is the set of
isometries from the quadratic space $(L'\otimes\kappa, \bar{q}_{L'})$
to the quadratic space $(H_k\otimes\kappa, \bar{q}_k)$.
Let $L'\otimes\kappa=\bar{L}'_0\bot \bar{L}'_1$, where $\bar{L}'_1 \left(= \mathrm{Rad}\left(L'\otimes\kappa\right)\right)$ is the radical of $L'\otimes\kappa$
so that the restriction of the quadratic form $\bar{q}_{L'}$ on $\bar{L}'_0$ is nonsingular.
We assign the following notion $a^{\pm}$ with an integer $a$ to $L'$ with respect to  $\bar{L}'_0$:
\begin{equation}\label{a}
\left\{
  \begin{array}{l l}
  a^+    & \quad  \textit{if $a=dim~\bar{L}'_0$ is even and $\bar{L}'_0$ is split};\\
  a^-    & \quad  \textit{if $a=dim~\bar{L}'_0$ is even and $\bar{L}'_0$ is nonsplit};\\
    a (=a^+=a^-)   & \quad  \textit{if $a=dim~\bar{L}'_0$ is odd}.
    \end{array} \right.
\end{equation}

Here, $dim~\bar{L}'_0$ is the dimension of $\bar{L}'_0$ as a $\kappa$-vector space.
If $a=0$, then we understand $0^+=0^-$.
By Exercise 4 in Section 5.6 of \cite{Kit}, 
  $\#\mathrm{O}_{\kappa}^{prim}(\bar{q}_{L'}, \bar{q}_k)(\kappa)$  is completely determined by three ingredients,  $a^{\pm}$, $n$, and $2k$.
Thus we can denote it by $\#\mathrm{O}^{prim}_{\kappa}(a^{\pm}, n, 2k)$.
Here,
\[
\left\{
  \begin{array}{l}
\textit{$2k$ is the dimension of the nondegenerate split quadratic  space $(H_k\otimes \kappa, \bar{q}_{\kappa})$};\\
\textit{$n$ is the dimension of the (possibly degenerate) quadratic  space $(L'\otimes \kappa, \bar{q}_{L'})$};\\
\textit{$a^{\pm}$ is as explained above such that $a$ is the dimension of a maximal nonsingular subspace of $L'\otimes \kappa$}.
    \end{array} \right.
\]
Note that $n$ is the rank of $L$.
In the following proposition, we will describe the integer $a$ in terms of $\mathrm{GK}(L)$.

\begin{Prop}\label{propzero}
The integer $a$, which is defined to be the dimension of  $\bar{L}'_0$ as a $\kappa$-vector space,
is the same as the number of $0$'s in $\mathrm{GK}(L')$.
\end{Prop}
\begin{proof}
This directly follows from observation of   a reduced form associated to the quadratic lattice $L$.
\end{proof}

Since each  summand in  Equation (\ref{eqld'}) is  determined by the number of $0$'s in $\mathrm{GK}(L')$ (with the signature $\pm$) and $[L':L]$, 
 we explain bounds of  these two objects in this paragraph.
Let $[L':L]=b$ and let $n_0$ be the number of $0$'s in $\mathrm{GK}(L)=(a_1, \cdots, a_n)$.
Remark \ref{rgk} yields that $|\mathrm{GK}(L')|=|\mathrm{GK}(L)|-2b$.
The integer $b$ is then nonnegative and at most $|\mathrm{GK}(L)|/2$
since $\mathrm{GK}(L')\succeq (0, \cdots, 0)$ (cf. Remark \ref{rgk}).
Let $\mathrm{GK}(L')=(a_1', \cdots, a_n')$
and let $a$ be the number of $0$'s in $\mathrm{GK}(L')$
(cf. Proposition \ref{propzero}).
If $b$ is positive, then the integer $a$ is at least $max\left\{n_0, n- |\mathrm{GK}(L')|\right\}$.


We introduce a new notion  $\mathcal{S}_{(L, a^{\pm}, b)}$
to denote the set of  all quadratic lattices $L'$ including $L$ for which the summands in Equation (\ref{eqld'}) are  equal.
More precisely,

\begin{equation}\label{sab}
\left\{
  \begin{array}{l}
\mathcal{S}_{(L, a^{+}, b)}=\left\{L' \left(\supseteq L\right) | \mathrm{GK}(L')\succeq (0, \cdots, 0), [L':L]=b,
\textit{$a^+$ is assigned to $L'$}\right\};\\
\mathcal{S}_{(L, a^{-}, b)}=\{L' \left(\supseteq L\right) | \mathrm{GK}(L')\succeq (0, \cdots, 0), [L':L]=b,
\textit{$a^-$ is assigned to $L'$}\}.
    \end{array} \right.
\end{equation}
If $a$ is odd or $0$, then $\mathcal{S}_{(L, a^{+}, b)}=\mathcal{S}_{(L, a^{-}, b)}$.
 Note that $\mathcal{S}_{(L, a^{\pm}, b)}$ is empty if $b > \frac{|\mathrm{GK}(L)|}{2}$ by Remark \ref{rgk}.
Therefore Equation (\ref{eqld'}) is now reformulated as follows:
\begin{equation}\label{eqld''}
\alpha(L, H_k)=
f^{-2kn+(n^2+n)/2}\cdot
\sum_{\substack{0\leq b\leq \frac{|\mathrm{GK}(L)|}{2},\\ n_0\leq a\leq n}} f^{b\cdot(n+1-2k)}\cdot \#\mathcal{S}_{(L, a^{\pm}, b)}\cdot
\#\mathrm{O}_{\kappa}^{prim}(a^{\pm}, n, 2k).
\end{equation}

Here, if $a$ is odd or $0$, then we ignore one of $\mathcal{S}_{(L, a^{+}, b)}$ or $\mathcal{S}_{(L, a^{-}, b)}$.
If $a$ is even and positive, then we count the summands involving $\mathcal{S}_{(L, a^{+}, b)}$ and $\mathcal{S}_{(L, a^{-}, b)}$ separately.

In the above equation, the number $\#\mathrm{O}_{\kappa}^{prim}(a^{\pm}, n, 2k)$ is already well-known as follows
(cf. Exercise 4 in Section 5.6 of \cite{Kit}):
\begin{equation}\label{eq37}
\mathrm{O}_{\kappa}^{prim}(a^{\pm}, n, 2k)=f^{2kn-(n^2+n)/2}(1-f^{-k})(1+\chi(a^{\pm})f^{n-a/2-k})\Pi_{1\leq i < n-a/2}(1-f^{2i-2k}).
\end{equation}
Here,
\[
\chi(a^{\pm})=\left\{
  \begin{array}{l l}
  0    & \quad  \textit{if $a$ is odd};\\
  1    & \quad  \textit{if $a$ is even and $a^+$ is assigned or if $a=0$};\\
   -1   & \quad  \textit{if $a(>0)$ is even and $a^-$ is assigned}.
    \end{array} \right.
\]

Using Equation (\ref{eqld''}) combined with the above description of $\#\mathrm{O}_{\kappa}^{prim}(a^{\pm}, n, 2k)$,
we finally obtain the  local density formula as follows:

\begin{Thm}\label{eqldf}
For a quadratic lattice $L$, we have 
\begin{multline*}
\alpha(L, H_k)=\\
(1-f^{-k})\cdot
\sum_{\substack{0\leq b\leq \frac{|\mathrm{GK}(L)|}{2},\\ n_0\leq a\leq n}} \left(\#\mathcal{S}_{(L, a^{\pm}, b)}\cdot
f^{b\cdot(n+1-2k)}\cdot (1+\chi(a^{\pm})f^{n-a/2-k})\prod_{1\leq i < n-a/2}(1-f^{2i-2k})\right).
\end{multline*}
Here, if $a$ is odd or $0$, then we ignore one of $\mathcal{S}_{(L, a^{+}, b)}$ or $\mathcal{S}_{(L, a^{-}, b)}$.
If $a$ is even and positive, then we count the summands involving $\mathcal{S}_{(L, a^{+}, b)}$ and $\mathcal{S}_{(L, a^{-}, b)}$ separately.

\end{Thm}


\begin{Def}\label{defse}
For a  quadratic lattice $L$, the Siegel series is defined to be the polynomial $\mathcal{F}_L(X)$ of $X$ such that
\[
\mathcal{F}_L(f^{-k})=\alpha(L, H_k)
\]
for all integer $k$ such that  $k\geq n$, where $n$ is the rank of $L$.
\end{Def}

\begin{Cor}\label{rmkse}
From the  formula of Theorem \ref{eqldf},
we have the following description of the Siegel series:
\begin{multline*}
\mathcal{F}_L(X)=\\
(1-X)\cdot
\sum_{\substack{0\leq b\leq \frac{|\mathrm{GK}(L)|}{2},\\ n_0\leq a\leq n}} \left(\#\mathcal{S}_{(L, a^{\pm}, b)}\cdot
f^{b\cdot(n+1)}X^{2b}\cdot (1+\chi(a^{\pm})f^{n-a/2}X)\prod_{1\leq i < n-a/2}(1-f^{2i}X^2)\right).
\end{multline*}

\end{Cor}

Thus, each coefficient of the polynomial $\mathcal{F}_L(X)$ is determined by 
 the order of the set $\mathcal{S}_{(L, a^{\pm}, b)}$.

\section{Inductive formulas of the Siegel series}\label{sectionifss}
A main goal of  this section is to find  an inductive formula of the Siegel series, by detailed investigation of the set of lattices $\mathcal{S}_{(L, a^{\pm}, b)}$.

For a quadratic lattice $L$ with $\mathrm{GK}(L)=(a_1, \cdots, a_n) \left(\neq (0, \cdots, 0)\right)$,
we choose a basis  $(e_1, \cdots, e_n)$  of $L$ 
such that with respect to this basis the symmetric matrix of the quadratic lattice $L$ is optimal.
Let $d$ be the integer such that $a_{n-d}< \underbrace{a_{n-d+1}=\cdots =a_n}_{d}$.
If $a_1=\cdots =a_n$, then we let $d=n$. 
We denote the  lattice $\left(\supset L\right)$ having a basis 
$$(e_1, \cdots, e_{n-d}, \underbrace{\frac{1}{\pi}\cdot e_{n-d+1}, \cdots, \frac{1}{\pi}\cdot e_n}_{d})$$ by 
\[
\left\{
  \begin{array}{l l}
 L^{(d, n)}   & \quad  \textit{if $d>1$};\\
 L^{(n)} ~(\textit{or $L^{(1, n)}$})  & \quad  \textit{if $d=1$}.
    \end{array} \right.
\]
 
We assume that the quadratic form on $L^{(d, n)}$, naturally induced by the quadratic form on $L$, is an integral quadratic form, that is,  $L^{(d, n)}$ is a quadratic lattice.
For example, if $a_n\geq 2$, then $L^{(d, n)}$ is a quadratic lattice.
Although it is not required that  $a_n\geq 2$, the integers $a_i$'s should satisfy the condition   that $a_n=\cdots = a_{n-d+1}\geq 1$, in order that $L^{(d, n)}$ is a quadratic lattice.


Let $\bar{V}_{L, d}=L^{(d, n)}/L$ be a $\kappa$-vector space of dimension $d$. 
Then  each lattice between $L$ and $L^{(d, n)}$ bijectively corresponds to each subspace of $\bar{V}_{L, d}$.
More precisely,  the set of all lattices $L'$ between $L$ and $L^{(d, n)}$ with degree $[L':L]=m$, where $0\leq m \leq d$,
equals the set of subspaces of $\bar{V}_{L, d}$ of dimension $m$.
We denote the former set of lattices by $\mathcal{G}_{L, d, m}$.
For example, $\mathcal{G}_{L, d, 0}=\{L\}$ and 
$\mathcal{G}_{L, d, d}=\{L^{(d, n)}\}$.
The latter set is   the Grassmannian,  denoted by  $G(m, d)$, whose cardinality is well known to be $\binom{d}{m}_f$.
Here, 
\[\binom{d}{m}_f=\frac{[(d)!]_f}{[m!]_f[(d-m)!]_f},\]
where \[[m!]_f=\prod_{t=1}^{m}\frac{f^t-1}{f-1}, \textit{    for any positive integer $m$}.\]
We write that $[0!]_f=1$ so that $\binom{d}{d}_f=\binom{d}{0}_f=1$.
For example, if $m=1$, then $\binom{d}{1}_f=\frac{f^{d}-1}{f-1}$.
Thus we have the following formula:
\begin{equation}\label{eqgrass}
\#\mathcal{G}_{L, d, m}=\binom{d}{m}_f.
\end{equation}

In the following lemma, we explain a property of a lattice including $L$, but not an element of $\mathcal{G}_{L, d, m}$.

\begin{Lem}\label{l1}
Consider a lattice  $L'$  in $V$ containing $L$.
If $L'$ does not contain any lattice in $\mathcal{G}_{L, d, m}$ for $1\leq m \leq d$ (equivalently $L'$ does not contain any lattice in $\mathcal{G}_{L, d, 1}$),
then there exists a direct summand $M'$ of $L'$ such that $L'=M'\oplus \underbrace{\mathfrak{o} e_{n-d+1}\oplus \cdots \oplus \mathfrak{o} e_n}_{d}$
and $L=M\oplus \underbrace{\mathfrak{o} e_{n-d+1}\oplus \cdots \oplus \mathfrak{o} e_n}_{d}$, where $M=L\cap M'$.
\end{Lem}

\begin{proof}
Let $l=n-d$.
For such $L'$, we denote the image of $e_i$, with $l+1\leq i \leq n$,   in  $L'/\pi L'$  by $\bar{e}_i$.
Let $\bar{V}_{d}'$ be the subspace of $L'/\pi L'$, as a $\kappa$-vector space,  spanned by $\bar{e}_{l+1}, \cdots,  \bar{e}_{n}$.
Since  $L'$ does not contain any lattice in $\mathcal{G}_{L, d, m}$ for $1\leq m \leq d$, 
the vectors $(\bar{e}_{l+1}, \cdots,  \bar{e}_{n})$ are linearly independent  
and thus 
the dimension of the vector space $\bar{V}_d'$ is  $d$.

Thus  there are $l$-vectors $(\bar{e}'_1, \cdots, \bar{e}_{l}')$ in $L'/\pi L'$ having
$(\bar{e}'_1, \cdots, \bar{e}_{l}', \bar{e}_{l+1}, \cdots,  \bar{e}_{n})$ as a basis.
Choose $(e'_1, \cdots, e_{l}')$ in $L'$ as  preimages of $(\bar{e}'_1, \cdots, \bar{e}_{l}')$.
By Nakayama's lemma, $(e'_1, \cdots, e_{l}', e_{l+1}, \cdots, e_n)$ is a basis of $L'$ as an $\mathfrak{o}$-module.

Let $M'$ be the submodule of $L'$ spanned by $(e'_1, \cdots, e_{l}')$ so that $L'/M'=\mathfrak{o}e_{l+1}\oplus \cdots \oplus \mathfrak{o}e_n$.
We consider the following short exact sequence:
\[1\rightarrow L\cap M' \rightarrow L \rightarrow L'/M'\rightarrow 1.\]
This short exact sequence  splits since there exists a section from $L'/M'$ to $L$
such that $e_{i}$  maps to $e_{i}$ with $l+1\leq i \leq n$.
Thus $L\cong (L\cap M')\oplus (\mathfrak{o}e_{l+1}\oplus \cdots \oplus \mathfrak{o}e_n)$ as $\mathfrak{o}$-modules.
Since this isomorphism is induced from the inclusions, we can identify $L$ with  $(L\cap M')\oplus (\mathfrak{o}e_{l+1}\oplus \cdots \oplus \mathfrak{o}e_n)$ 
as submodules of $L'$.
Let $M=L\cap M'$.
This completes the proof.
\end{proof}

\begin{Lem}\label{l2}
For a basis $(e_1, \cdots, e_n)$ of $L$,
consider a direct summand $M^{\dag}$ of $L$ such that $L=M^{\dag}\oplus (\underbrace{\mathfrak{o} e_{n-d+1}\oplus \cdots \oplus \mathfrak{o} e_n}_d)$.
Then there is a basis of $M^{\dag}$ consisting of the column vectors of the matrix
$\begin{pmatrix} id_{n-d}\\x \end{pmatrix}$, where $id_{n-d}$ is the identity matrix of size $n-d$ and $x\in \mathrm{M}_{d\times (n-d)}(\mathfrak{o})$.
\end{Lem}

\begin{proof}
A basis of $M^{\dag}\oplus (\mathfrak{o} e_{n-d+1}\oplus \cdots \oplus \mathfrak{o} e_n)$ is given by the column vectors of the matrix
$\begin{pmatrix} x_1&0 \\ x_2&id_{d} \end{pmatrix}$ with entries in $\mathfrak{o}$, where $x_1$ is a square matrix of size $n-d$.
Since this matrix is invertible, $x_1$ is invertible  over $\mathfrak{o}$ as well.
Thus we can choose another basis for  $L=M^{\dag}\oplus (\mathfrak{o} e_{n-d+1}\oplus \cdots \oplus \mathfrak{o} e_n)$, given by the column vectors  of the matrix
$\begin{pmatrix} x_1&0 \\x_2&id_{d} \end{pmatrix}
\begin{pmatrix} x_1^{-1}&0 \\0&id_{d} \end{pmatrix}
=\begin{pmatrix} id_{n-d}&0 \\x_2x_1^{-1}&id_{d} \end{pmatrix}$.
Let $x=x_2x_1^{-1}$.
This completes the proof.
\end{proof}

\begin{Rmk}\label{rlx}
\begin{enumerate}
\item In the situation of Lemma \ref{l1}, a direct summand $M$ of $L$ has a basis given by the column vectors of a matrix $\begin{pmatrix} id_{n-d}\\x \end{pmatrix}$
with $x\in \mathrm{M}_{d\times (n-d)}(\mathfrak{o})$ by Lemma \ref{l2}.
Such a lattice $M$ is denoted by  
 $L^{(d, n)}_x$,  in order to emphasize both roles of $x$ and  $(d, n)$, so that
\[
L=L^{(d, n)}_x\oplus \underbrace{\mathfrak{o} e_{n-d+1}\oplus \cdots \oplus \mathfrak{o} e_n}_{d}.
\] 

 If $d=1$, then we sometimes write $L^{(n)}_x$ to denote $L^{(1,n)}_x$.

 The simplest case of $L^{(d, n)}_x$ is when $x$ is the zero vector.
 In this case,
\[
L^{(d, n)}_0=\mathfrak{o} e_{1}\oplus \cdots \oplus \mathfrak{o} e_{n-d}.
\] 
The lattice  $L^{(d, n)}_0$ will be crucially used 
in our inductive formula of the Siegel series (cf. Theorems \ref{mainthm} and \ref{indfaniso3}).

\item A basis of $L$ such that with respect to this basis the symmetric matrix of $L$ is  optimal (reduced, respectively) is called  \textit{an optimal basis} (\textit{a  reduced basis}, respectively).

\item
If $(e_1, \cdots, e_n)$ is an optimal basis (resp. reduced basis) of $L$, then
Theorem 0.2 (resp. Lemma 4.6) of \cite{IK1} yields that
$$L^{(d, n)}_x\oplus \underbrace{\mathfrak{o} e_{n-d+1}\oplus \cdots \oplus \mathfrak{o} e_n}_{d}$$ forms an optimal basis  (resp. reduced basis). 



\item Ikeda and Katsurada imposed extra invariant to a quadratic lattice $L$  in addition to the Gross-Keating invariant and called it `Extended Gross-Keating datum', denoted by $\mathrm{EGK}(L)$ in \cite{IK1}.
They defined $\mathrm{EGK}(L)$ based on an optimal form of $L$, which turns to be independent of the choice of an optimal form (cf. Theorem 0.4 and Definition 6.3 of \cite{IK1}).
For an explicit description of $\mathrm{EGK}(L)$, we refer to Definition 6.3 of \cite{IK1}.
A main contribution of their another paper \cite{IK2} is to prove that  the Siegel series $\mathcal{F}_L(X)$ is completely determined by $\mathrm{EGK}(L)$ (cf.  Theorem 1.1 of \cite{IK2}).

\item Since column vectors of the matrix 
$\begin{pmatrix} id_{n-d}&0 \\x &id_{d}\end{pmatrix}$ form an optimal basis of
 $L$,
 Theorem 0.3 of \cite{IK1} yields  that $\mathrm{GK}(L^{(d, n)}_0)=\mathrm{GK}(L^{(d, n)}_x)$. 
In addition, by Definition 6.3 of \cite{IK1}, one can easily see that $\mathrm{EGK}(L^{(d, n)}_0)=\mathrm{EGK}(L^{(d, n)}_x)$ 
since 
$\mathrm{EGK}(L)$ is independent from the  choice of an optimal basis.
Therefore, the argument of  the above (4) implies that 
 $$\mathcal{F}_{L^{(d, n)}_0}(X)=\mathcal{F}_{L^{(d, n)}_x}(X).$$\newline
\end{enumerate}
\end{Rmk}


From now on until the end of this section, we work with the following choice of a basis of a quadratic lattice $L$:
\[
\textit{$(e_1, \cdots, e_n)$ is } \left\{
  \begin{array}{l l}
 \textit{an optimal basis} & \quad  \textit{if $p$ is odd};\\
 \textit{a reduced basis} & \quad  \textit{if $p$ is even}.
    \end{array} \right.
\]
Before analyzing $\#\mathcal{S}_{(L, a^{\pm}, b)}$ in Proposition \ref{prop3},
we will introduce one conjecture regarding quadratic forms modulo $\pi$ in Conjecture  \ref{conj4}, and prove it when $p$ is odd or when $L$ is anisotropic over $\mathbb{Z}_2$ in Lemmas \ref{l4}-\ref{l5}.

We write $L=M\oplus N$, where $M$ (respectively $N$) is spanned by $(e_{1}, \cdots, e_{n-d})$ (respectively $(\underbrace{e_{n-d+1}, \cdots, e_n}_{d})$).
We consider a lattice $L'$ containing $L$ in $L\otimes_{\mathfrak{o}}F$ of the form $L'=M'\oplus N$,
where $M'$ is a lattice containing $M$ in $M\otimes_{\mathfrak{o}}F$.
Let $q_{L'}$ be the quadratic form on $L'$ which is naturally induced from $q_L$ on $L$ such that $q_{L'}|_{L}=q_L$.
Similarly, we define  the quadratic form $q_{M'}$ defined on $M'$.

In the case that $q_{M'}$  is integral (equivalently, $\mathrm{GK}(M')\succeq (0, \cdots, 0)$), 
we define $\bar{M}'=M'/\pi M'$ to be the quadratic space over $\kappa$ having  the quadratic form $q_{M'}$ modulo $\pi$.
Similarly, in the case that  $q_{L'}$ on $L'$  is integral,  we define the quadratic space $\bar{L}'=L'/\pi L'$.

If $q_{L'}$ is integral, then $q_{M'}$ is also integral. Conversely, we propose the following conjecture.

\begin{Conj}\label{conj4}
If $q_{M'}$ on $M'$  is integral, then $q_{L'}$ on $L'$  is integral (thus they are equivalent).
In this situation,
the dimension of  $\bar{L}'$ modulo the radical, is the same as the dimension of  $\bar{M}'$ modulo the radical.
In other words,  the number of $0$'s in $\mathrm{GK}(L')$ is the same as 
the number of $0$'s in $\mathrm{GK}(M')$ (cf. Proposition \ref{propzero}).
\end{Conj}

We think that the conjecture is true in the general case. 
In the following, we will prove it in two cases, when $p$ is odd or when 
$(L, q_L)$ is an anisotropic $\mathbb{Z}_2$-lattice.

\begin{Lem}\label{l4}
The conjecture is true for an odd prime $p$.
\end{Lem}

\begin{proof}
Let $B=\begin{pmatrix}
a&b\\{}^tb&c
\end{pmatrix}$ be an optimal form with respect to $(e_1, \cdots, e_n)$, where
the size of $a$ is $(n-d)\times (n-d)$ and the size of $c$ is $d\times d$.
A symmetric matrix of $L'$ with respect to the decomposition $L'=M'\oplus N$ such that a basis of $N$ is fixed by $(\underbrace{e_{n-d+1}, \cdots, e_n}_{d})$ is of the form
$\begin{pmatrix}
{}^tx&0\\0&id
\end{pmatrix}\cdot
\begin{pmatrix}
a&b\\{}^tb&c
\end{pmatrix}\cdot
\begin{pmatrix}
x&0\\0&id
\end{pmatrix}=
\begin{pmatrix}
{}^tx\cdot a\cdot x &{}^tx\cdot b\\{}^tb\cdot x&c
\end{pmatrix}$ for certain $x\in \mathrm{GL}_{n-d}(F)$. 
Here, ${}^tx\cdot a\cdot x$ is a symmetric matrix of $M'$. 
Thus it suffices to show that the exponential order of each entry of $2\cdot {}^tx\cdot b$ is at least $1$.

Since $p$ is odd, we can choose another basis of $M$ given by the column vectors of $y\in \mathrm{GL}_{n-d}(\mathfrak{o})$ such that ${}^ty\cdot a\cdot y$ is an optimal and diagonal matrix by Remark \ref{rgk}.(1).
Then by Theorems 0.2-0.3 of \cite{IK1}, the symmetric matrix 
$\begin{pmatrix}
{}^ty\cdot a\cdot y &{}^ty\cdot b\\{}^tb\cdot y&c
\end{pmatrix}$ is also optimal with respect to a basis  forming   $L=M\oplus N$.
Thus we may and do assume that $a$ is a diagonal matrix with the $i$-th diagonal entry $u_i\pi^{a_i}$, where $u_i$ is a unit in $\mathfrak{o}$.

Since we assumed that  $M'$ is a quadratic lattice containing $M$ in $M\otimes_{\mathfrak{o}}F$ with $p$ odd,
the lattice $M'$ is contained in the dual lattice of $M$, which is defined to be the set $\{v\in M\otimes_{\mathfrak{o}}F | b_{q_M}(v, L)\in \mathfrak{o}\}$.
Here, $b_{q_M}$ is the symmetric bilinear form associated to the  quadratic form $q_M$ such that $b_{q_M}(v,v)=q_M(v)$.
The dual lattice of $M$ is spanned by $(\pi^{-a_1}e_1, \cdots, \pi^{-a_{n-d}}e_{n-d})$.
Thus, if $x_{n-d}$ is the diagonal matrix of size $n-d$ whose $i$-th diagonal entry is  $\pi^{-a_i}$, then a matrix $x$ determining $M'$ is of the form $x_{n-d}\cdot x'$, where $x'\in \mathrm{GL}_{n-d}(F)\cap \mathrm{M}_{n-d}(\mathfrak{o})$.

On the other hand, the exponential order of each entry of $2\cdot {}^tx_{n-d}\cdot b$ is at least $1$ since $B$ is optimal and $a_i<a_{n-d+1}=a_n$ for any $i \left(\leq n-d\right)$.
Therefore, the exponential order of each entry of $2\cdot {}^t(x')\cdot {}^tx_{n-d}\cdot b$ is at least $1$ since $x'\in  \mathrm{M}_{n-d}(\mathfrak{o})$.
This completes the proof.
\end{proof}

\begin{Lem}\label{l5}
The conjecture is true for an anisotropic quadratic $\mathbb{Z}_p$-lattice with any $p$.
\end{Lem}

\begin{proof}
Let $L$ be an anisotropic quadratic lattice over $\mathbb{Z}_p$ and let $\mathrm{GK}(L)=(a_1, \cdots, a_n)$ so that $n$ is at most $4$.
It is well known that there is a unique maximal quadratic lattice in $L\otimes_{\mathfrak{o}}F$ by Theorem 91:1 of \cite{O}.
Thus it suffices to prove the conjecture when $M'$ is   the maximal quadratic lattice inside $M\otimes_{\mathfrak{o}}F$.

We can easily prove  that any anisotropic quadratic lattice whose Gross-Keating invariant consists of $0$ and $1$ is maximal
since only two among $a_i$'s have the same parity (cf.  Proposition \ref{parity}).

As in the proof of Lemma \ref{l4}, 
let $B=\begin{pmatrix}
a&b\\{}^tb&c
\end{pmatrix}$ be a reduced form with respect to $(e_1, \cdots, e_n)$, 
where
the size of $a$ is $(n-d)\times (n-d)$ and the size of $c$ is $d\times d$.
Let $x_{n-d}$ be the diagonal matrix of size $n-d$ whose $i$-th diagonal entry is  $\pi^{-[a_i/2]}$. 
The proof of Proposition 3.2 in \cite{IK1} implies that 
the Gross-Keating invariant of  ${}^tx_m\cdot a \cdot x_m$  consists of $0$ and $1$  
 so that the associated quadratic lattice is maximal.
A simple calculation shows that the exponential order of each entry of $2\cdot {}^tx_m\cdot b$ is at least $1$. This completes the proof.
\end{proof}

\begin{Rmk}\label{rmk6}
In the proof of the above lemma, the only place to use the assumption of $\mathfrak{o}=\mathbb{Z}_p$ is that any anisotropic quadratic lattice whose Gross-Keating invariant consists of $0$ and $1$ is maximal which follows from Proposition \ref{parity}.
If this is true for a general $\mathfrak{o}$, then the proof of the lemma works so that the conjecture is true for an anisotropic quadratic lattice over $\mathfrak{o}$.
\end{Rmk}

Using Lemmas \ref{l1}-\ref{l2} and assuming Conjecture \ref{conj4}, we will  explain  an inductive formula of $\# \mathcal{S}_{(L, a^{\pm}, b)}$ in the following proposition.

\begin{Prop}\label{prop3}
Assume that Conjecture \ref{conj4} is true. 
Let $b$ be an integer such that
$0\leq b\leq \frac{|\mathrm{GK}(L)|}{2}$.
Then for any integer $b'$ with  $b'\geq b$, we have the following formula:
\begin{multline*}
\# \mathcal{S}_{(L, a^{\pm}, b)}=\\
\sum_{m=1}^{d} \left(c_m \cdot \sum_{L'\in \mathcal{G}_{L, d, m}}\# \mathcal{S}_{(L', a^{\pm}, b-m)}\right)+
f^{d(b-(n-d)b')}\sum_{x\in M_{d\times (n-d)}(\mathfrak{o}/\pi^{b'}\mathfrak{o})}\# \mathcal{S}_{(L^{(d, n)}_x, a^{\pm}, b)},
\end{multline*}
where 
$c_m=(-1)^{m-1}f^{m(m-1)/2}$.
Here, if $b-m<0$, then we understand $\# \mathcal{S}_{(L', a^{\pm}, b-m)}=0$.
And $L^{(d, n)}_x$ for $x\in M_{d\times (n-d)}(\mathfrak{o}/\pi^{b'}\mathfrak{o})$ stands for a lattice associated to any representative of $x$ in $M_{d\times (n-d)}$. It depends on such representative but  
$\#\mathcal{S}_{(L^{(d, n)}_x, a^{\pm}, b)}$ does not.

If $\# \mathcal{S}_{(L, a^{\pm}, b)}=0$, then each $\# \mathcal{S}_{(L', a^{\pm}, b-m)}$ and  $\#\mathcal{S}_{(L^{(d, n)}_x, a^{\pm}, b)}$ are zero as well.

\end{Prop}

\begin{proof}

Note that $\mathcal{S}_{(L', a^{\pm}, b-m)} \subseteq  \mathcal{S}_{(L, a^{\pm}, b)}$ for $L'\in \mathcal{G}_{L, d, m}$.
We choose a lattice 
$$L^{\dag}\in \mathcal{S}_{(L, a^{\pm}, b)} \setminus \bigcup_{L'\in \mathcal{G}_{L, d, 1}} \mathcal{S}_{(L', a^{\pm}, b-1)},$$
if the right hand side is nonempty.
By Lemmas \ref{l1}-\ref{l2} and Remark  \ref{rlx}.(1),
there exists a direct summand $L^{(d, n)}_x$ of $L$ such that 
\[
\left\{
  \begin{array}{l}
    L=L^{(d, n)}_x\oplus \underbrace{\mathfrak{o} e_{n-d+1}\oplus \cdots \oplus \mathfrak{o} e_n}_d;\\
L^{\dag}=(L^{(d, n)}_x)^{\dag}\oplus \underbrace{\mathfrak{o} e_{n-d+1}\oplus \cdots \oplus \mathfrak{o} e_n}_d,
    \end{array} \right.
\]
as an $\mathfrak{o}$-lattice
(not as a quadratic $\mathfrak{o}$-lattice).
Here, $(L^{(d, n)}_x)^{\dag}$ is a direct summand of $L^{\dag}$ satisfying the condition that 
$L\cap (L^{(d, n)}_x)^{\dag}=L^{(d, n)}_x$.
Note that if $(e_1, \cdots, e_n)$ is an optimal basis (resp. reduced basis) of $L$, then the above basis for $L$ is also optimal (resp. reduced) by Remark \ref{rlx}.(3).
Thus we can use Conjecture \ref{conj4} with $L=L^{(d, n)}_x\oplus \underbrace{\mathfrak{o} e_{n-d+1}\oplus \cdots \oplus \mathfrak{o} e_n}_d$. 
More precisely, 
since $[L^{\dag}:L]=[(L^{(d, n)}_x)^{\dag}:L^{(d, n)}_x]=b$,
the lattice  $(L^{(d, n)}_x)^{\dag}$ is contained in $\mathcal{S}_{(L^{(d, n)}_x, a^{\pm}, b)}$ by Conjecture \ref{conj4}.
Thus by Conjecture \ref{conj4} again, we have that  

\begin{equation}\label{eqsla}
\mathcal{S}_{(L, a^{\pm}, b)} \setminus \bigcup_{L'\in \mathcal{G}_{L, d, 1}} \mathcal{S}_{(L', a^{\pm}, b-1)} =
\bigcup_{x\in M_{d\times n-d}(\mathfrak{o})} \mathcal{S}_{(L^{(d, n)}_x, a^{\pm}, b)}\oplus \underbrace{\mathfrak{o} e_{n-d+1}\oplus \cdots \oplus \mathfrak{o} e_n}_d.
\end{equation}

Here, $\mathcal{S}_{(L^{(d, n)}_x, a^{\pm}, b)}\oplus \underbrace{\mathfrak{o} e_{n-d+1}\oplus \cdots \oplus \mathfrak{o} e_n}_d$  is the set of  lattices  $\{M'\oplus \underbrace{\mathfrak{o} e_{n-d+1}\oplus \cdots \oplus \mathfrak{o} e_n}_d | M'\in \mathcal{S}_{(L^{(d, n)}_x, a^{\pm}, b)}\}.$
The above equation holds even when the left hand side is empty, by assuming Conjecture \ref{conj4}.
This confirms the last sentence of the proposition.

Since $[M':L^{(d, n)}_x]=b$ for $M'\in \mathcal{S}_{(L^{(d, n)}_x, a^{\pm}, b)}$, we can see that 
$$\mathcal{S}_{(L^{(d, n)}_x, a^{\pm}, b)}\oplus \underbrace{\mathfrak{o} e_{n-d+1}\oplus \cdots \oplus \mathfrak{o} e_n}_d
=\mathcal{S}_{(L^{(d, n)}_y, a^{\pm}, b)}\oplus \underbrace{\mathfrak{o} e_{n-d+1}\oplus \cdots \oplus \mathfrak{o} e_n}_d
\textit{ if } x\equiv y \textit{ mod $\pi^b$}.$$
Thus Equation (\ref{eqsla}) can be expressed as follows:

\begin{equation}\label{eqsla2}
\mathcal{S}_{(L, a^{\pm}, b)} \setminus \bigcup_{L'\in \mathcal{G}_{L, d, 1}} \mathcal{S}_{(L', a^{\pm}, b-1)} =
\bigcup_{x\in M_{d\times (n-d)}(\mathfrak{o}/\pi^{b'}\mathfrak{o})} \mathcal{S}_{(L^{(d, n)}_x, a^{\pm}, b)}\oplus \underbrace{\mathfrak{o} e_{n-d+1}\oplus \cdots \oplus \mathfrak{o} e_n}_d
\end{equation}
for any integer $b'$ such that $b' \geq b$.
Here, $L^{(d, n)}_x$ for $x\in M_{d\times (n-d)}(\mathfrak{o}/\pi^{b'}\mathfrak{o})$ stands for a lattice associated to any representative of $x$ in $M_{d\times (n-d)}(\mathfrak{o})$. It depends on such representative but  
$\mathcal{S}_{(L^{(d, n)}_x, a^{\pm}, b)}\oplus \underbrace{\mathfrak{o} e_{n-d+1}\oplus \cdots \oplus \mathfrak{o} e_n}_d$ does not.

In order to compute the cardinality of the right hand side of Equation (\ref{eqsla2}),
we compare it with 
$$\sum_{x\in M_{d\times (n-d)}(\mathfrak{o}/\pi^{b'}\mathfrak{o})}\#\left( \mathcal{S}_{(L^{(d, n)}_x, a^{\pm}, b)}\oplus \underbrace{\mathfrak{o} e_{n-d+1}\oplus \cdots \oplus \mathfrak{o} e_n}_d\right).$$
In the following, we will see how many times a given lattice is counted in this sum.

We choose a lattice $L_x^{\dag}$ in $\mathcal{S}_{(L^{(d, n)}_x, a^{\pm}, b)}\oplus \underbrace{\mathfrak{o} e_{n-d+1}\oplus \cdots \oplus \mathfrak{o} e_n}_d$.
The cardinality of the set
$\{y\in M_{d\times (n-d)}(\mathfrak{o}/\pi^{b'}\mathfrak{o}) |  L_x^{\dag}\in \mathcal{S}_{(L_y^{(d, n)}, a^{\pm}, b)}\oplus \underbrace{\mathfrak{o} e_{n-d+1}\oplus \cdots \oplus \mathfrak{o} e_n}_d\}$
is then
$$\underbrace{\#(\mathfrak{o}/\pi^{b'-b}\mathfrak{o})\cdot \cdots \cdot \#(\mathfrak{o}/\pi^{b'-b}\mathfrak{o})}_{d}\cdot \underbrace{\#(\mathfrak{o}/\pi^{b'}\mathfrak{o})\cdot \cdots \cdot \#(\mathfrak{o}/\pi^{b'}\mathfrak{o})}_{d(n-d-1)}
=f^{d(n-d)b'-db}.$$
Note that the above number is independent of the choice of $x$.
Since 
$$\#\left(\mathcal{S}_{(L^{(d, n)}_x, a^{\pm}, b)}\oplus \underbrace{\mathfrak{o} e_{n-d+1}\oplus \cdots \oplus \mathfrak{o} e_n}_d\right)=\#\left(\mathcal{S}_{(L^{(d, n)}_x, a^{\pm}, b)}\right),$$
 we have the following equation:

\[\# \left(\mathcal{S}_{(L, a^{\pm}, b)} \setminus \bigcup_{L'\in \mathcal{G}_{L, d, 1}} \mathcal{S}_{(L', a^{\pm}, b-1)}  \right)=
f^{db-d(n-d)b'}\cdot\sum_{x\in M_{d\times (n-d)}(\mathfrak{o}/\pi^{b'}\mathfrak{o})}\# \mathcal{S}_{(L^{(d, n)}_x, a^{\pm}, b)}.\]

Thus to complete the proof, it suffices to show that
\begin{equation}\label{eqsla3}
\#\left(\bigcup_{L'\in \mathcal{G}_{L, d, 1}} \mathcal{S}_{(L', a^{\pm}, b-1)}\right)=
\sum_{m=1}^{d} \left(c_m\cdot \sum_{L'\in \mathcal{G}_{L, d, m}}\# \mathcal{S}_{(L', a^{\pm}, b-m)}\right),
\end{equation}
where $c_m=-\left(\binom{m}{1}_f\cdot c_1+\binom{m}{2}_f\cdot c_2+\cdots +\binom{m}{m-1}_f\cdot c_{m-1} \right)+1$ if $m>1$ and $c_1=1$.

This follows from inclusion-exclusion principle using the counting argument of the Grassmannian given at the beginning of this section.

For the proof\footnote{This proof  was informed by the referee.} of $c_m=(-1)^{m-1}f^{m(m-1)/2}$, we consider the $q$-binomial formula
\[
\prod^{m-1}_{k=0}(1+f^kt)=\sum_{k=0}^{m}f^{k(k-1)/2}\binom{m}{k}_ft^k.
\]
Putting $t=-1$, we have 
\[
\sum_{k=0}^{m}(-1)^kf^{k(k-1)/2}\binom{m}{k}_f=0.
\]
Then by induction, we can easily see that $c_m=(-1)^{m-1}f^{m(m-1)/2}$.
\end{proof}




We now state our main theorem of this section,   an  inductive formula of the Siegel series 
$\mathcal{F}_L(X)$.
\begin{Thm}\label{mainthm}
Assume that Conjecture \ref{conj4} is true. 
Then we have the following inductive formula, with respect to the Gross-Keating invariant, of the Siegel series $\mathcal{F}_L(X)$:


\begin{multline*}
\mathcal{F}_L(X)=\sum_{m=1}^{d} \left( c_m \cdot f^{\left(n+1\right)m}\cdot X^{2m}\cdot\sum_{L'\in \mathcal{G}_{L, d, m}}\mathcal{F}_{L'}(X)\right)+\\
\displaystyle (1-X)(1-f^{d}X)^{-1}\cdot \left(\prod_{i=1}^{d}(1-f^{2i}X^2)\right)\cdot \mathcal{F}_{L_0^{(d, n)}}(f^{d}X),
\end{multline*}
where $c_m=(-1)^{m-1}f^{m(m-1)/2}$.
Here, for $L'\in \mathcal{G}_{L, d, m}$, 
\[
\left\{
  \begin{array}{l}
 \mathrm{GK}(L) \succ \mathrm{GK}(L');\\
|\mathrm{GK}(L')|=|\mathrm{GK}(L)|-2m;\\
 \mathrm{GK}(L_0^{(d, n)})=\mathrm{GK}(L)^{(n-d)}.
    \end{array} \right.
\]

Note that notion of 
$L^{(d, n)}$, $\mathcal{G}_{L, d, m}$, and $\binom{m}{k}_f$ 
can be found at the beginning of this section.
Notion of  $L_0^{(d, n)}$  can be found at Remark \ref{rlx}.(1).
\end{Thm}

\begin{proof}


If we plug  the  formula of Proposition \ref{prop3} into the formula of  Theorem \ref{eqldf},
then we obtain
\begin{multline*}
\alpha(L, H_k)=
\sum_{m=1}^{d} \left(c_m \cdot f^{\left(n-2k+1\right)m}\sum_{L'\in \mathcal{G}_{L, d, m}}\alpha(L', H_k)\right)+\\
\displaystyle f^{-d(n-d)b'}(1-f^{-k})(1-f^{-(k-d)})^{-1}\cdot \left(\prod_{i=1}^{d}(1-f^{2i-2k})\right)\cdot
\sum_{x\in M_{d\times (n-d)}(\mathfrak{o}/\pi^{b'}\mathfrak{o})}
\alpha(L^{(d, n)}_x, H_{k-d}),
\end{multline*}
where $c_m=-\left(\binom{m}{1}_f\cdot c_1+\binom{m}{2}_f\cdot c_2+\cdots +\binom{m}{m-1}_f\cdot c_{m-1} \right)+1$ if $m>1$ and $c_1=1$.

On the other hand, 
as mentioned at Remark \ref{rlx}.(5),
we have that
$$\alpha(L^{(d, n)}_x, H_{k-d})=\alpha(L^{(d, n)}_{0}, H_{k-d})$$
for any $x \in M_{d\times (n-d)}(\mathfrak{o}/\pi^{b'}\mathfrak{o})$.
This completes the proof of the inductive formula.
The rest follows from Theorems 0.1 and 0.3 of \cite{IK1}.
\end{proof}


\begin{Cor}\label{cor411}
Assume that Conjecture \ref{conj4} is true. 
If $a_{n-1}<a_n$ for $\mathrm{GK}(L)=\left(a_1, \cdots, a_n \right)$ so that $d=1$, then the above inductive formula turns to be
\begin{equation*}
\mathcal{F}_L(X)=f^{n+1}\cdot X^2\cdot \mathcal{F}_{L^{(n)}}(X)+
\displaystyle (1-X)(1+fX)\cdot \mathcal{F}_{L_0^{(n)}}(fX).
\end{equation*}
Note that notion of $L^{(n)}$
can be found at the beginning of this section and that 
notion of  $L_0^{(n)}$  can be found at Remark \ref{rlx}.(1).

\end{Cor}

\begin{Rmk}
We note that Katsurada described several interesting inductive formulas in   Theorem 2.6 of \cite{Kat}.
Theorem 2.6.(1) in loc. cit. is the same as 
Corollary \ref{cor411} when $p$ is odd. More generally, our inductive formula in Theorem \ref{mainthm} is  directly followed from   Proposition 2.2 combined with Proposition 2.5.(1) in loc. cit., when $p$ is odd. 
On the other hand, when $p=2$, the inductive formula of Theorem \ref{mainthm} is sharper than that of Theorem 2.6 in \cite{Kat}, due to restriction in Proposition 2.5.(1) of loc. cit. 
Thus our result is a more refined version of Katsurada's formula stated in Theorem 2.6 of \cite{Kat}.
\end{Rmk}

The following lemma  is used in the main theorems \ref{indfaniso}-\ref{indfaniso3} of the next section. 

\begin{Lem}\label{lemro}
 Let $B=\begin{pmatrix} B'&C\\{}^tC&d  \end{pmatrix}$ be a reduced form with   $\mathrm{GK}(B)=(a_1, \cdots, a_{n-1}, a_n)$.
Here, $B'$ is of size $(n-1)\times (n-1)$ and $d\in \mathfrak{o}$.
We assume
that $B'$ is also a reduced form with $\mathrm{GK}(B')=(a_1, \cdots, a_{n-1})$
  and
 that $B$ satisfies  one of the followings:
\[
\left\{
  \begin{array}{l }
 \textit{$a_n>a_{n-1}$};\\
 \textit{$a_n=a_{n-1}>a_{n-2}$, $\sigma(n-1)=n$, $\mathrm{ord}(d)=a_n$, and $\mathfrak{o}=\mathbb{Z}_2$}.
    \end{array} \right.
\]

Let $B_x={}^t\begin{pmatrix} id_{n-1}\\x \end{pmatrix}\cdot B \cdot \begin{pmatrix} id_{n-1}\\x \end{pmatrix}$, where 
 $x\in \mathrm{M}_{1\times (n-1)}(\mathfrak{o})$.
Let $\mathcal{F}_{B_x}(X)$ be the Siegel series of the quadratic lattice associated to the symmetric matrix $B_x$.
Then we have 
\[
\mathcal{F}_{B_0}(X)=\mathcal{F}_{B_x}(X)
\]
for any $x$.

\end{Lem}

\begin{proof}
We write 
$$
B_x=\begin{pmatrix} id_{n-1}&{}^tx  \end{pmatrix}\cdot \begin{pmatrix} B'&C\\{}^tC&d  \end{pmatrix}\cdot \begin{pmatrix} id_{n-1}\\x  \end{pmatrix}
=B'+\left( Cx+{}^t(Cx)+d\cdot {}^txx \right).
$$
The $(i, j)$-th entry of $\left( Cx+{}^t(Cx)+d\cdot {}^txx \right)$ is $c_ix_j+c_jx_i+dx_ix_j$.
Here, we write ${}^tC=(c_1, \cdots, c_{n-1})$.
Then we have the following:\[
 \left\{
  \begin{array}{l }
\mathrm{ord}(c_ix_i+c_ix_i+dx_ix_i)>a_i;\\
\mathrm{ord}(2(c_ix_j+c_jx_i+dx_ix_j))>\frac{a_i+a_j}{2} \textit{ if $i<j$}.
    \end{array} \right.
\]
This, combined with  \footnote{Indeed Ikeda and Katsurada assume that $p=2$ in Theorem 3.3, loc. cit. 
But this theorem also holds for $p>2$ and it was explained in the initial version of their paper posted on arXiv.}Theorems 3.3 and 1.1 of \cite{IK2}, completes the proof.
\end{proof}

\section{On the Siegel series of anisotropic quadratic lattices}\label{sectionssaql}
In this section, we will explain a more refined inductive formula of the Siegel series of  anisotropic quadratic lattices defined over $\mathbb{Z}_p$, so as to compare it with an inductive formula of local intersection multiplicities  of \cite{GK1} in Sections \ref{reGK}-\ref{sectionapp}.
Thus throughout this section, let $(L, q_L)$ be an anisotropic quadratic lattice over $\mathbb{Z}_p$.
Here, we say  that  $(L, q_L)$  is anisotropic if $q_L(x)\neq 0$ for any nonzero element $x\in L$.
Then the rank of $L$ is at most $4$.
We first list a few necessary facts about $(L, q_L)$.

\begin{Rmk}\label{rmk5.1}

\begin{enumerate}
\item Let $D$ be the unique quaternion division algebra over $\mathbb{Q}_p$.
Let  $q_D$ be   the associated quadratic form,
which is defined to be the reduced norm on $D$.
There is a unique maximal order of $D$, denoted by $O_D$.
The maximal order $O_D$ is characterized as follows:
\[O_D=\{v\in D | q_D(v)\in \mathbb{Z}_p\}.\]
\item
The pair $(O_D, q_D)$ is an anisotropic quadratic lattice of rank $4$ over $\mathbb{Z}_p$.
By Exercise 3 of Section 5.2 in \cite{Kit}, there is a suitable basis of $O_D$
such that with respect to this basis the symmetric matrix of $O_D$ is described as follows:
\[\left\{
  \begin{array}{l l}
  (1)\perp (-\delta)\perp (p)\perp (-\delta p)   & \quad    \text{if $p\neq 2$};\\
   \begin{pmatrix} 1&1/2 \\ 1/2 & 1\end{pmatrix} \perp \begin{pmatrix} 2&1 \\ 1 & 2\end{pmatrix}  & \quad    \text{if $p=2$}.
    \end{array} \right.\]
Here, $\delta$ is a unit in $\mathbb{Z}_p$ (with $p\neq 2$) such that $\delta$ modulo $p$ is not a square.
These symmetric matrices are reduced forms. Thus we can see that 
\[
\mathrm{GK}(O_D)=(0,0,1,1).
\]

\item Any anisotropic quadratic lattice over $\mathbb{Z}_p$ of rank $n \left(\leq 4 \right)$ is always embedded into $(O_D, q_D)$, which can be shown by using Theorem 3.5.1 and Corollary 3.5.4 of \cite{Kit}.
Thus we may regard $(L, q_L)$ as a sublattice of $(O_D, q_D)$.
\end{enumerate}
\end{Rmk}

The Gross-Keating invariants of anisotropic quadratic lattices over $\mathbb{Z}_p$
with rank $\leq 3$ are well explained in \cite{B}.
In the next subsection, we will explain the Gross-Keating invariants and some properties of anisotropic quadratic lattices of rank $4$.

\subsection{On an anisotropic quadratic lattice of rank $4$}\label{subsectionlgkaql}

\begin{Lem}\label{lemtype2}
Consider an anisotropic quadratic lattice of rank $2$ over $\mathbb{Z}_2$.
Choose a basis $(e_1, e_2)$ such that
with respect to this basis the half-integral symmetric matrix is $X=\begin{pmatrix} a&b \\ b&c \end{pmatrix}$.
Assume that $\mathrm{GK}(X)=(a_1, a_1+2t)$ for $t\geq 0$.
If $X$ is an  optimal form, then we have the following: 
\[\mathrm{ord}(a)=a_1, \mathrm{ord}(2b)=a_1+t, \textit{ and } \mathrm{ord}(c)=a_1+2t.\]
Thus any optimal form of an anisotropic quadratic lattice of rank $2$ over $\mathbb{Z}_2$ is a reduced form.
\end{Lem}

\begin{proof}
Assume that $X$ with $\mathrm{GK}(X)=(a_1, a_2)$ is diagonalizable.
We claim that $a_1<a_2$. 
 Let $X'=\begin{pmatrix}
x&0\\0&y
\end{pmatrix}$, where $\mathrm{ord}(x)\leq \mathrm{ord}(y)$, be a diagonal matrix isometric to $X$.
The norm of $X$ is $(2^{\mathrm{ord}(x)})$ so that $a_1=\mathrm{ord}(x)$ by Remark \ref{rgk}.(4).
Assume that  $\mathrm{ord}(y)>\mathrm{ord}(x)$.
Since 
$\mathrm{GK}(X)=(a_1, a_2)\succeq (\mathrm{ord}(x), \mathrm{ord}(y))$,
we have that 
$a_2 \geq \mathrm{ord}(y)>\mathrm{ord}(x)=a_1$ and thus
 $a_2>a_1$.
Assume that $\mathrm{ord}(x)=\mathrm{ord}(y)$.
Consider the following matrix equation:
\[
\begin{pmatrix}
1&0\\1&1
\end{pmatrix}\begin{pmatrix}
x&0\\0&y
\end{pmatrix}\begin{pmatrix}
1&1\\0&1
\end{pmatrix}=\begin{pmatrix}
x&x\\x&x+y
\end{pmatrix}.
\] 
Since $\mathrm{ord}(x+y)$ is at least $a_1+1$, $\mathrm{GK}(X)\succeq (a_1, a_1+1)$.
Thus we have that  $a_1<a_2$.

We first assume that  $t=0$.
Since  $\mathrm{GK}(X)=(a_1, a_1)$, 
  $X$ is not diagonalizable by  the above claim.
 The lemma then   follows from Lemma 3.2.(b) of  \cite{B}.

For $t>0$,
it is easy to see that the symmetric matrix with respect to the basis $(e_1, \frac{1}{2^t}\cdot e_2)$ is an optimal form with $\mathrm{GK}=(a_1, a_1)$ by using Remark \ref{rgk}.(4)-(5).
Using the result of  the case $t=0$ completes the proof.
\end{proof}

If we write $(a_1, a_2, a_3)$ to be the Gross-Keating invariant of an anisotropic quadratic lattice  of rank $3$ over $\mathbb{Z}_p$ for any $p$,
then only two of $a_i$'s have the same parity by Lemma 5.3 of \cite{B}.
In the following proposition, we prove the same statement in the case of  $n=4$.

\begin{Prop}\label{parity}
Let $\mathrm{GK}(L)=(a_1, a_2, a_3, a_4)$.
Then only two of $a_i$'s have the same parity.
This holds for any prime $p$.
\end{Prop}

\begin{proof}
Since $L$ is a sublattice of $O_D$ with the same rank $4$,
the parity of $|\mathrm{GK}(L)|$ is the same as that of $|\mathrm{GK}(O_D)|$ by Remark \ref{rgk}.(5).
Since  $\mathrm{GK}(O_D)=(0, 0, 1, 1)$ and $|\mathrm{GK}(O_D)|=2$,
either   only two of $a_i$'s have the same parity or all  of $a_i$'s have the same parity.
Assume that all of  $a_i$'s have the same parity.

If $p\neq 2$, then we choose a  basis $(e_1, e_2, e_3, e_4)$ for $L$, whose
associated symmetric matrix is a diagonal matrix having 
$u_ip^{a_i}$ as the $i$-th diagonal entry, where $u_i$ is a unit in $\mathbb{Z}_p$, by Remark \ref{rgk}.(1).
Thus   the Gross-Keating invariant of the sublattice $L'$ spanned by $(e_1, e_2, e_3)$ is $(a_1, a_2, a_3)$.
Note that $L'$ is anisotropic with $\mathrm{GK}(L')=(a_1, a_2, a_3)$.
This is a contradiction since $a_1, a_2, a_3$ have the same parity.

If $p=2$, then we choose a reduced basis $(e_1, e_2, e_3, e_4)$ of $(L, q_L)$. 
Since $a_i$'s have the same parity, we can choose an involution $\sigma$ attached to $(a_1, a_2, a_3, a_4)$ such that 
$\sigma(1)=2, \sigma(3)=4$ (cf. Definition \ref{def3.1}).
Note that $\sigma$ is unique if $a_2<a_3$. 
Let $B=\begin{pmatrix}
b_{ij}
\end{pmatrix}$ be the associated symmetric matrix. 
The definition of a reduced form (cf. Defintion \ref{def3.2}) then tells us that
two matrices of size $2$ associated to $(e_1, e_2)$ and $(e_3, e_4)$ are also reduced forms. 
  Lemma \ref{lemtype2} yields that $\mathrm{ord}(b_{ii})=a_i$.
Thus  the symmetric matrix associated to a basis $(e_1, e_2, e_3)$ is also a reduced form whose Gross-Keating invariant is $(a_1, a_2, a_3)$, where the associated involution permutes $1$ and $2$.
This is a contradiction to the fact that the lattice spanned by $(e_1, e_2, e_3)$ is anisotropic, since $a_1, a_2, a_3$ have the same parity.
\end{proof}

\begin{Prop}\label{propaniso}
Let $B$ be a reduced form of an anisotropic quadratic lattice $L$ of rank $4$ with $\mathrm{GK}(L)=(a_1, a_2, a_3, a_4)$.
Let $(e_1, e_2, e_3, e_4)$ be a basis of a reduced form $B$.
If $p$ is odd, then we consider $B$ as a diagonal matrix.

Then any $(3\times 3)$-submatrix of the matrix $B$, with respect to a basis $(e_i, e_j, e_k)$ among $(e_1, e_2, e_3, e_4)$, is a reduced form whose Gross-Keating invariant is $(a_i, a_j, a_k)$, where $i<j<k$.

Similarly, any $(2\times 2)$-submatrix of the matrix $B$, with respect to a basis $(e_i, e_j)$ among $(e_1, e_2, e_3, e_4)$, is a reduced form whose Gross-Keating invariant is $(a_i, a_j)$, where $i<j$.
\end{Prop}

\begin{proof}
If $p$ is odd, then it is clear.
When $p=2$, the proof is similar to the argument used in the proof of the above proposition.
Thus we skip it.
\end{proof}

\begin{Rmk}\label{redrmk}
Let 
$\mathrm{GK}(L)=(a_1, a_2, a_3, a_4)$ with $a_4\geq 2$ and let $p=2$.
We choose  a reduced basis  $(e_1, e_2, e_3, e_4)$ with an involution $\sigma$.
Let $\sigma(4)=i$ and $\sigma(j)=k$ so that $j, k$ are odd and $\{1,2,3,4\}=\{4,i,j,k\}$. 
To ease notation, we say $\sigma(a_4)=a_i$ and $\sigma(a_j)=a_k$ if there is no confusion.
Then by using a similar argument used in the proof of  Proposition \ref{parity}, we can easily see that
$(e_1, e_2, e_3, 1/p\cdot e_4)$ is a reduced basis of $L^{(4)}$, up to a permutation,  such that $\mathrm{GK}(\mathrm{L^{(4)}})=(a_1, a_2, a_3)\cup (a_4-2)$.
A key point of the proof is to use the fact that $\mathrm{ord}(q_L(e_i))=a_i$.
The associated involution to $L^{(4)}$ is compatible with $\sigma$
so as to exchange $a_4-2$ and $a_i$, and exchange $a_j$ and $a_k$.
\end{Rmk}

If $p$ is odd, then we can choose a diagonal matrix as a reduced form of $L$. Then the Gross-Keating invariant consists of  orders of each diagonal entries (cf. Remark \ref{rgk}.(1)).

We assume that $p=2$.
Let $B$ be an half-integral symmetric matrix associated to $(L, q_L)$ of rank $4$.
By using Theorem 2.4 of \cite{C1} and Lemma 3.2.(c) of \cite{B},
there are three types of $B$, up to equivalence, as follows:
\[
\left\{
  \begin{array}{l}
  \textit{Case (I) : }B=2^i \begin{pmatrix}1&1/2 \\ 1/2&1  \end{pmatrix}\perp 2^j\begin{pmatrix}1&1/2 \\ 1/2&1  \end{pmatrix};\\
  \textit{Case (II) : }B=2^i\begin{pmatrix}1&1/2 \\ 1/2&1  \end{pmatrix}\perp (u_12^{\mu_1})\perp (u_22^{\mu_2});\\
  \textit{Case (III) : }B=(u_12^{\mu_1})\perp (u_22^{\mu_2})\perp (u_32^{\mu_3})\perp (u_42^{\mu_4}).
      \end{array} \right.
\]
Here,
\[
\left\{
  \begin{array}{l}
  \textit{in Case (I), $i\leq  j$};\\
  \textit{in Case (II), $u_i \equiv 1$ mod $2$ and $\mu_1 \leq \mu_2$};\\
  \textit{in Case (III), $u_i \equiv 1$ mod $2$, $\mu_i \leq \mu_j$ if $i<j$, $\mu_1<\mu_3$, and $\mu_2<\mu_4$}.
      \end{array} \right.
\]

In the following theorem, we explain the Gross-Keating invariant of each case.

\begin{Thm}\label{classification}
The determinant of $B$ is a square. The Gross-Keating invariant of $B$ is described as follows:
\begin{enumerate}
\item In Case (I), we have
$$\mathrm{GK}(B)=(i, i, j, j),$$
where  $i$ and $j$  have different parities.

\item In case (II), $\mu_1$ and $\mu_2$ have the same parity,   $u_1+u_2\equiv 0$ mod $4$,
and
$$\mathrm{GK}(B)=(i, i)\cup (\mu_1, \mu_2+2).$$
Here,  $i$ and $\mu_1$ have different parities.

\item In case (III),
\begin{enumerate}
\item Assume $\mu_1\not\equiv \mu_2$ mod $2$.
Then $\mu_3$ and $\mu_4$ have different parities
and
  \[GK(B)=(\mu_1, \mu_2, \mu_3+2, \mu_4+2).\]

\item Assume that $\mu_1\equiv \mu_2$ mod 2 and that
  $u_1+u_2\equiv 2$ mod $4$ or $\mu_2=\mu_3$.
Then  $\mu_3$ and $\mu_4$ have the same parity
and
  \[GK(B)=(\mu_1, \mu_2+1, \mu_3+1, \mu_4+2).\]

\item Assume that $\mu_1\equiv \mu_2$ mod 2, $u_1+u_2\equiv 0$ mod 4,
and $\mu_3 \geq \mu_2+1$.

Then   \[GK(B)=(\mu_1, \mu_2+2) \cup (\mu_3, \mu_4+2).\]
Here, $\mu_3$ and $\mu_4$ have the same parity, $\mu_1$ and $\mu_3$  have different parities, and  $u_3+u_4\equiv 0$ mod $4$.
\end{enumerate}
\end{enumerate}

\end{Thm}

\begin{proof}
The determinant of $B$ is a square since the determinant of the matrix associated to $O_D$ given at the beginning of this section is a square.
Using this, the proof easily follows from Lemma 2.8 and Theorem 3.1 of \cite{CIKY2} (or Theorems 3.5-3.8 of \cite{CIKY1}) and Proposition \ref{parity}.
\end{proof}

In the next proposition, we compute the local density $\alpha(L, O_D)$ for an anisotropic quadratic lattice $(L, q_L)$ of rank $4$.
The definition (and normalization) of the local density follows from Section 5 of \cite{IK2}.

\begin{Prop}\label{propld}
The local density $\alpha(L, O_D)$ is
\[
\alpha(L, O_D)=[O_D:L]\cdot  \cdot p^{3}\cdot p^{-2}(2(p+1))^2.
\]

Here, $[O_D:L]$  equals to $(\mathrm{GK}(L)-2)/2$.
\end{Prop}

\begin{proof}
By Hilfssatz 17 of \cite{Sie}, we have
\[\alpha(L, O_D)= [O_D:L]\cdot \alpha(O_D, O_D),  \]
where $[O_D:L]$ is the index of $L$ in $O_D$.
The local density $\alpha(O_D, O_D)$ for a single quadratic lattice $O_D$ is fully studied in \cite{C1} $(p=2)$ and \cite{GY} $(p\neq 2)$.

By using the matrix description of the quadratic lattice $(O_D, q_D)$ given at Remark \ref{rmk5.1}.(2),
 the local density $\alpha(O_D, O_D)$ (cf. Proposition 6.2.3 and Theorem 7.3 of \cite{GY} when $p\neq 2$, Theorems 4.12 and 5.2 of \cite{C1} when $p=2$) is
\[
\alpha(O_D, O_D)=  p^{3}\cdot (p^{-1}2(p+1))^2.
\] 
This completes the proof.
\end{proof}

\begin{Rmk}
In the above proof, indeed Theorems 4.12 and 5.2 of \cite{C1} when $p=2$ yield that 
\[
\alpha(O_D, O_D)= 1/2\cdot p^{-3} \cdot p^{-6}\cdot p^4(2(p+1))^2.
\]
But, using the normalization of the local density explained in Section 5 of \cite{IK2}, we need to multiply $p^6$ and to ignore $1/2$ in the above formula. 
\end{Rmk}

\subsection{The Siegel series of anisotropic quadratic lattice}\label{subsectionssaqlr4}

In this subsection, we will explain a more refined inductive formula of the Siegel series of anisotropic quadratic lattices over $\mathbb{Z}_p$. 
 Let $(L, q_L)$ be an anisotropic quadratic lattice over $\mathbb{Z}_p$ of rank $4$.
By Remark \ref{rmk5.1}.(3), we may consider it as a sublattice of $(O_D, q_D)$.

We will work with an exclusively chosen basis $(e_1, e_2, e_3, e_4)$ of $L$  as follows until the end of this subsection:
\[
\textit{$(e_1, \cdots, e_n)$ is } \left\{
  \begin{array}{l l}
 \textit{diagonal and optimal} & \quad  \textit{if $p$ is odd};\\
 \textit{reduced} & \quad  \textit{if $p=2$}.
    \end{array} \right.
\]

Let $L^{(4)}$ be the lattice spanned by $(e_1, e_2, e_{3}, 1/p\cdot e_4)$
and let $L_0^{(4)}$ be the lattice spanned by $(e_1, e_2, e_{3})$.
If $a_3<a_4$, then these notions are the same as those introduced at the beginning of Section \ref{sectionifss} and in Remark \ref{rlx}.(1).
We keep using them even in the case of $a_3=a_4$.

\begin{Lem}\label{nonmax}
Assume that $(L, q_L)$ is not maximal.
Then $L^{(4)}$ is contained in   $(O_D, q_D)$.
\end{Lem}

\begin{proof}
It suffices to show that $L^{(4)}$ is a quadratic lattice. 
Let $\mathrm{GK}(L)=(a_1, a_2, a_3, a_4)$.
Since  $(L, q_L)$ is not maximal,   $\mathrm{GK}(O_D)=(0,0,1,1)$, and only two of $a_i$'s have the same parity,
we have that $a_4\geq 2$.
Using  Remark \ref{redrmk} when $p=2$, we can see that  $\mathrm{GK}(L^{(4)})=(a_1, a_2, a_3)\cup (a_4-2)\succeq (0,0,0,0)$.
This holds when $p\neq 2$, which 
 completes the proof by Lemma \ref{lemnorm} and Corollary \ref{primex}. 
\end{proof}

Recall that $\mathcal{F}_L(X)$ is the Siegel series associated to the quadratic lattice $(L, q_L)$ such that 
$\mathcal{F}_L(f^{-k})=\alpha(L, H_k)$ (cf. Definition \ref{defse}).
In the following theorem, we will explain an inductive formula of the Siegel series $\mathcal{F}_L(X)$ for an anisotropic quadratic lattice of rank $4$ over $\mathbb{Z}_2$. 
This formula is much simpler than  that of Theorem \ref{mainthm}.
 
\begin{Thm}\label{indfaniso}
Assume that $(L, q_L)$ is not maximal.
Let $\mathrm{GK}(L)=(a_1, a_2, a_3, a_4)$.
Then we have the following inductive formula:
\[\mathcal{F}_L(X)=p^{5}\cdot X^2\cdot \mathcal{F}_{L^{(4)}}(X)+(1-X)(1+pX)\cdot \mathcal{F}_{L^{(4)}_0}(pX).\]
Here, 
\[
\left\{
  \begin{array}{l}
 \textit{$\mathrm{GK}(L^{(4)})=(a_1, a_2, a_3)\cup (a_4-2)$};\\
 \textit{$\mathrm{GK}(L^{(4)}_0)=(a_1, a_2, a_3)$}.
    \end{array} \right.
\]
\end{Thm}

\begin{proof}
If $a_3<a_4$, then the formula follows from Theorem \ref{mainthm}.
Assume that $a_3=a_4$.
Since $a_2$ and $a_3$ should have different parities, we have that $a_2<a_3$.

From our choice of $(e_1, e_2, e_3, e_4)$,
the symmetric matrix with a basis  
$(e_1, e_2, e_3, \frac{1}{p}\cdot e_4)$ (resp. $(e_1, e_2, e_3)$) is a reduced form whose associated Gross-Keating invariant is $(a_1, a_2, a_3)\cup (a_4-2)$ (resp. $(a_1, a_2, a_3)$) (cf. Proposition \ref{propaniso} and Remark \ref{redrmk}).
Using the argument explained in  the proofs of Lemma \ref{l5} and Proposition \ref{prop3},
we have the following formula:
\[
\# \mathcal{S}_{(L, a^{\pm}, b)}=\# \mathcal{S}_{(L^{(4)}, a^{\pm}, b-1)}+
f^{b-3b'}\cdot\sum_{x\in M_{1\times 3}(\mathfrak{o}/\pi^{b'}\mathfrak{o})}\# \mathcal{S}_{(L^{(4)}_x, a^{\pm}, b)}.
\]
Here $b'$ is any integer with $b'\geq b$
and $L^{(4)}_x$ is a lattice of rank $3$ constructed by using the same argument (with $n=4$ and $d=1$) given in Remark \ref{rlx}.(1)
and Proposition \ref{prop3}.

Assume that  $p=2$.
We consider a reduced form $B$ with respect to a basis $(e_1, e_2, e_3, e_4)$  of $(L, q_L)$.
Then $B$ satisfies the assumption of Lemma \ref{lemro}, by Proposition \ref{propaniso} and Lemma \ref{lemtype2}.
Thus, if we plug  the above formula into  the formula of  Theorem \ref{eqldf}  using the result of Lemma \ref{lemro}, then we obtain the desired inductive formula.

We now assume that $p\neq 2$.
We consider a diagonal matrix $B=(u_1p^{a_1})\perp (u_2p^{a_2})\perp (u_3p^{a_3})\perp (u_4p^{a_4})$ with respect to a basis $(e_1, e_2, e_3, e_4)$, where $u_i\in\mathbb{Z}_p$ is a unit.
Then  for $x=(x_1, x_2, x_3)\in \mathbb{Z}_p^3$, we have

\[B_x=\begin{pmatrix} id_{3}&{}^tx  \end{pmatrix}\cdot B\cdot \begin{pmatrix} id_{3}\\x  \end{pmatrix}=
\begin{pmatrix}
u_1p^{a_1} &0&0 \\
 0&u_2p^{a_2} &0 \\
0&0&u_3p^{a_3} 
\end{pmatrix}
+u_4p^{a_4}\begin{pmatrix} x_1x_1&x_1x_2&x_1x_3\\
x_2x_1&x_2x_2&x_2x_3\\
x_3x_1&x_3x_2&x_3x_3
  \end{pmatrix}.
\]

If $\mathrm{ord}((u_3+u_4x_3^2)p^{a_3})=a_3$, 
where $(u_3+u_4x_3^2)p^{a_3}$ is the  $(3, 3)$-th entry of $B_x$,  then the matrix $B_x$ is a reduced form with $\mathrm{GK}(B_x)=(a_1, a_2, a_3)$ by using Theorem 3.3 of \cite{IK2}
for the $2\times 2$-minor involving $u_1p^{a_1}$ and $u_2p^{a_2}$ and the definition of a reduced form given in Definition \ref{def3.2}. 

As explained in Remark \ref{rlx}.(4), the Siegel series is completely determined by the Extended Gross-Keating datum. 
In our case of anisotropic quadratic lattices, the extended Gross-Keating datum is the same as the Gross-Keating invariant.
Thus if $\mathrm{GK}(B_x)=(a_1, a_2, a_3)$ for any $x$, then the associated Siegel series's are all equal.  
Using a similar argument used in the case $p=2$,
we have the desired formula.

Thus, it suffices to prove that $\mathrm{ord}((u_3+u_4x_3^2)p^{a_3})=a_3$, equivalently that $u_3+u_4x_3^2$ is a unit.
For $a\in \mathbb{Z}_p$, let $\begin{pmatrix} \frac{a}{p} \end{pmatrix}$ be the Legendre symbol.
If $u_3+u_4x_3^2$ is not a unit, then $\begin{pmatrix} \frac{-u_3}{p} \end{pmatrix}=\begin{pmatrix} \frac{u_4x_3^2}{p} \end{pmatrix}$
so that $\begin{pmatrix} \frac{-u_3u_4}{p} \end{pmatrix}=\begin{pmatrix} \frac{u_4x_3}{p} \end{pmatrix}^2=1$.
Since the lattice spanned by $e_3$ and $e_4$ is anisotropic, we have that  $\begin{pmatrix} \frac{-u_3u_4}{p} \end{pmatrix}=-1$
 by Lemma 2.8 of \cite{B}.
This is a contradiction.
Thus we conclude that $u_3+u_4x_3^2$ is  a unit.
\end{proof}

The proof of the above theorem also holds for any anisotropic quadratic lattice of rank $n\leq 3$.
We  state it as the following theorem:
\begin{Thm}\label{indfaniso3}
In this theorem, let $(L, q_L)$ be an anisotropic quadratic lattice over $\mathbb{Z}_p$ of rank $n$.
Let $\mathrm{GK}(L)=(a_1, \cdots, a_n)$.
Assume that $(L, q_L)$ is not maximal.
Then we have the following inductive formula for $\mathcal{F}_L(X)$:
\[
\mathcal{F}_L(X)=\left\{
  \begin{array}{l l}
  p^{n+1}\cdot X^2\cdot \mathcal{F}_{L^{(n)}}(X)+(1-X)(1+pX)\cdot \mathcal{F}_{L^{(n)}_0}(pX)   & \quad    \text{if $2\leq n \leq 4$};\\
  p^{2}\cdot X^2\cdot \mathcal{F}_{L^{(1)}}(X)+(1-X)(1+pX)  & \quad    \text{if $n=1$}.
    \end{array} \right.\]

Here, $L^{(n)}$ is spanned by $(e_1, \cdots, e_{n-1}, 1/p\cdot e_n)$ 
and $L^{(n)}_0$ is spanned by $(e_1, \cdots, e_{n-1})$ so that 

\[
\left\{
  \begin{array}{l}
 \textit{$\mathrm{GK}(L^{(n)})=(a_1, \cdots, a_{n-1})\cup (a_n-2)$};\\
 \textit{$\mathrm{GK}(L^{(n)}_0)=(a_1, \cdots, a_{n-1})$}.
    \end{array} \right.
\]

\end{Thm}

\begin{Rmk}\label{rmk512}
As mentioned in the  proof of Theorem \ref{indfaniso}, for an anisotropic quadratic $\mathbb{Z}_p$-lattice, 
it is easy to see from Definition 6.3 of \cite{IK1} that the extended Gross-Keating datum is the same as the Gross-Keating invariant.
Thus, the Siegel series for an anisotropic quadratic $\mathbb{Z}_p$-lattice is completely determined by the Gross-Keating invariant by 
Remark \ref{rlx}.(4).
\end{Rmk}

\begin{Rmk}\label{01max}
Let $\mathrm{GK}(L)=(a_1, \cdots, a_n)$. 
Lemma \ref{nonmax} and its proof imply that $L$ of rank $4$ is maximal if and only if $a_4\leq 1$. 
We claim that this statement holds for any anisotropic quadratic lattice of any rank over $\mathbb{Z}_p$.
This implies that
if $(L, q_L)$ is not maximal, then $L^{(n)}$ is integral in Theorem \ref{indfaniso3}.

If the rank is $1$, then it is obvious by Remark \ref{rgk}.(4)-(5).
If the rank is $3$, then it follows from the fact that only two integers among $a_i$'s have the same parity.
If the rank is $2$, then it suffices to show that the quadratic lattice $M$ with $\mathrm{GK}(M)=(1,1)$ is maximal.

Let $(e_1, e_2, e_3, e_4)$ be a basis of $(O_D, q_D)$ given in  Remark \ref{rmk5.1}.(2).
Let $M'$ be a sublattice of $O_D$ spanned by $(e_3, e_4)$ so that $\mathrm{GK}(M')=(1,1)$.
If $M'$ is not maximal, then we should be able to find a quadratic lattice $\tilde{M}$ of rank $2$ containing $M'$ so that $\mathrm{GK}(\tilde{M})=(0,0)$. 
Thus, the lattice spanned by $(e_1,e_2)$ and $\tilde{M}$ is also integral whose Gross-Keating invariant is $(0,0,0,0)$, which is a contraction to the fact that $(O_D, q_D)$ is maximal.
Thus $M'$ is maximal.

Let's go back to the case $M$ with $\mathrm{GK}(M)=(1,1)$. 
Thus the Siegel series $\mathcal{F}_M(X)$ should be the same as $\mathcal{F}_{M'}(X)$ by Remark \ref{rmk512}.
Assume that $M$ is not maximal. 
Since $M'$ is maximal, the local density associated to $\mathcal{F}_{M'}(X)$ is the same as the primitive local density for $M'$, whereas the local density associated to $\mathcal{F}_M(X)$ should be the sum of the primitive local density for $M$ and those for  quadratic lattices including $M$.
But, the prmitive local densities for $M$ and $M'$ are equal  and the primitive local density for a bigger lattice including $M$ is not trivial (cf. Theorem \ref{tpld} and Equation (\ref{eq37})). 
Thus, $\mathcal{F}_M(X)$ cannot be the same as $\mathcal{F}_{M'}(X)$, which is a contradiction.
\end{Rmk}

\section{Comparison: the Siegel series and the local intersection multiplicities}\label{reGK}

Gross and Keating computed the   local intersection multiplicities in \cite{GK1} and Kudla confirmed that it is the same as the derivative of the Siegel series of an anisotropic quadratic lattice of rank $3$ at $p^{-2}$ (cf.\cite{A}).  
 The method used to show the equality between these two objects is to compute both sides independently, and then to compare them directly. 
The calculation of the local intersection multiplicities in \cite{GK1} is based on an inductive formula given in Lemma 5.6 in loc. cit.

In this section, we will compare the inductive formula of \cite{GK1} with 
our inductive formula of Theorem \ref{indfaniso3}.
Then we will show that these two are essentially equal, beyond matching values.
In addition to that, we will explain a newly discovered equality between the local intersection multiplicity on the special fiber and the derivative of another Siegel series in Theorem \ref{thmaniso1}. 
This observation had been missed in both of Siegel series and intersection numbers.


Let us restrict the following situation exclusively in this section:
\[
\left\{
  \begin{array}{l}
 \textit{$L$ : anisotropic quadratic lattice over $\mathbb{Z}_p$ of rank $3$};\\
 \textit{$M$ : anisotropic quadratic lattice over $\mathbb{Z}_p$ of rank $2$};\\
 \textit{$N$ : anisotropic quadratic lattice over $\mathbb{Z}_p$ of rank $1$}.
    \end{array} \right.
\]

Since $\mathcal{F}_L(X)$ is determined by $\mathrm{GK}(L)=(a_1, a_2, a_3)$ (cf. Remark \ref{rmk512}), 
we can write $\mathcal{F}_L(X)=\mathcal{F}_{(a_1, a_2, a_3)}(X)$.
Similarly, we can write $\mathcal{F}_M(X)=\mathcal{F}_{(a_1, a_2)}(X)$ with $\mathrm{GK}(M)=(a_1, a_2)$
and $\mathcal{F}_N(X)=\mathcal{F}_{(a_1)}(X)$ with $\mathrm{GK}(N)=(a_1)$.


As in Section \ref{subsectionssaqlr4}, a basis of each lattice, consisting of  $e_i$'s, is chosen to be
\[
 \left\{
  \begin{array}{l l}
 \textit{diagonal and optimal} & \quad  \textit{if $p$ is odd};\\
 \textit{reduced} & \quad  \textit{if $p=2$}.
    \end{array} \right.
\]

In the following two lemmas, we list the initial values of the Siegel series and its derivative, in order to compare both inductive formulas.

\begin{Lem}\label{lemaniso1}
We have
\[\mathcal{F}_L(1/p^2)=\mathcal{F}_{M}(1/p)=\mathcal{F}_{N}(1)=0.\]
\end{Lem}

\begin{proof}
These directly follow from 
Corollary \ref{rmkse}.
\end{proof}

\begin{Lem}\label{lemaniso2}
Special values of the derivative of the Siegel series are given as follows:
\[
\mathcal{F}_{M}'(1/p)=\left\{
  \begin{array}{l l}
 -(p-1)  & \quad  \textit{if $(a_1, a_2)=(0, 0)$};\\
 -2(p-1)(p+1)  & \quad  \textit{if $(a_1, a_2)=(1, 1)$};\\
 -2(p-1)  & \quad  \textit{if $(a_1, a_2)=(0, 1)$},
    \end{array} \right.
    \]
\[
\mathcal{F}_{N}'(1)=\left\{
  \begin{array}{l l}
 -1  & \quad  \textit{if $(a_1)=(0)$};\\
 -(p+1)  & \quad  \textit{if $(a_1)=(1)$}.
    \end{array} \right.
\]
\end{Lem}
\begin{proof}
We write
 $\mathcal{F}_{(a_1, a_2)}(X)$ for $\mathcal{F}_M(X)$  with $\mathrm{GK}(M)=(a_1, a_2)$.
In our situation, $a_i$ is either $0$ or $1$. Thus $M$ is a maximal quadratic lattice (cf. Remark \ref{01max}). 
Using Corollary \ref{rmkse}, the Siegel series for a maximal quadratic lattice is described explicitly as follows:
\[
\left\{
  \begin{array}{l}
\mathcal{F}_{(0, 0)}(X)=(1-X)(1-pX);\\
\mathcal{F}_{(1, 1)}(X)=(1-X)(1+p^2X)(1-p^2X^2);\\
\mathcal{F}_{(0, 1)}(X)=(1-X)(1-p^2X^2).
    \end{array} \right.
    \]

Thus 
\[
\left\{
  \begin{array}{l}
\mathcal{F}'_{(0, 0)}(1/p)=-(p-1);\\
\mathcal{F}'_{(1,1)}(1/p)=-2(p-1)(p+1);\\
\mathcal{F}'_{(0, 1)}(1/p)=-2(p-1).
    \end{array} \right.
    \]

Similarly, 
we write 
$\mathcal{F}_{(a_1)}(X)$ for $\mathcal{F}_N(X)$  with $\mathrm{GK}(N)=(a_1)$.
Then 
\[
\left\{
  \begin{array}{l}
\mathcal{F}_{(0)}(X)=1-X;\\
\mathcal{F}_{(1)}(X)=(1-X)(1+pX).
    \end{array} \right.
    \]
Thus
\[
\left\{
  \begin{array}{l}
\mathcal{F}'_{(0)}(1)=-1;\\
\mathcal{F}'_{(1)}(1)=-(p+1).
    \end{array} \right.
    \]
\end{proof}

Let $\mathrm{GK}(L)=(a_1, a_2, a_3)$ and let   $(e_1, e_2, e_3)$ be a chosen basis of $L$ described at the beginning of this section.
Let 
 $M$ be $L^{(3)}_0$  spanned by $(e_1, e_2)$ so that $\mathrm{GK}(M)=(a_1, a_2)$.
We have the following inductive formula:

\begin{Prop}\label{propaniso1}
Let $\alpha(L, O_D)$ be the local density of the pair of quadratic lattices $(L, q_L)$ and $(O_D, q_D)$.
Assume that $L$ is not maximal.
Then we have the following inductive formula:
\[
\frac{\mathcal{F}'_L(1/p^2)}{\alpha(L, O_D)}=\frac{\mathcal{F}'_{L^{(3)}}(1/p^2)}{\alpha(L^{(3)}, O_D)}+\frac{p-1}{2p}\cdot \mathcal{F}'_M(1/p).
\]
Here, $L^{(3)}$ is a lattice spanned by $(e_1, e_2, \frac{1}{p}\cdot e_3)$.
\end{Prop}

\begin{proof}
By differentiating the formula of Theorem \ref{indfaniso3} at $p^{-2}$ using Lemma \ref{lemaniso1}, we have
\[
\mathcal{F}'_L(1/p^2)=\mathcal{F}'_{L^{(3)}}(1/p^2)+(1-\frac{1}{p^2})(p+1)\cdot\mathcal{F}'_M(1/p).
\]

On the other hand, for any anisotropic quadratic lattice $L$ of rank $3$, we have
\[\alpha(L, O_D)=2(p+1)^2p^{-1}\]
 by Theorem 1.1 of \cite{Wed1}.
The desired inductive formula follows from these two.
\end{proof}


Let $N$ be $M^{(2)}_0$  spanned by $\left(e_1\right)$ so that $\mathrm{GK}(N)=(a_1)$.

\begin{Prop}\label{propaniso2}
Assume that $M$ is not maximal.
Then we have the following inductive formulas:
\[
\left\{
  \begin{array}{l}
 \mathcal{F}'_M(1/p)=p\cdot \mathcal{F}'_{M^{(2)}}(1/p)+
2(p-1)\cdot \mathcal{F}'_N(1);\\
\mathcal{F}'_N(1)=p^2\cdot\mathcal{F}'_{N^{(1)}}(1)-(p+1).
    \end{array} \right.
\]
 Here, $M^{(2)}$ is a lattice spanned by $(e_1,  \frac{1}{p}\cdot e_2)$ and 
$N^{(1)}$ is a lattice spanned by $(\frac{1}{p}\cdot e_1)$.
In the second equation, we assume that $N$ is not maximal.
\end{Prop}

\begin{proof}
The formulas directly follow  by
 differentiating the formulas of Theorem \ref{indfaniso3} using Lemma \ref{lemaniso1}.
 \end{proof}

We write $\mathcal{F}'_{(a_1)}(1)=\mathcal{F}'_N(1)$ with $\mathrm{GK}(N)=(a_1)$.
The above inductive formula yields the explicit value of $\mathcal{F}'_{(a_1)}(1)$ as follows. 

\begin{Lem}\label{ll1}
We have that
\[\mathcal{F}'_{(a_1)}(1)=-(1+p+p^2+\cdots +p^{a_1}).\]
    
\end{Lem}
\begin{proof}
Since $\mathrm{GK}(N^{(1)})=(a_1-2)$, 
by Proposition \ref{propaniso2} and Lemma \ref{lemaniso2},
we have that
\[
\mathcal{F}'_{(a_1)}(1)=\left\{
  \begin{array}{l l}
  -p^{a_1}-(p+1)(1+p^2+\cdots +p^{a_1-2})   & \quad    \text{if $a_1$ is even};\\
  -(p+1)p^{a_1-1}-(p+1)(1+p^2+\cdots +p^{a_1-3})   & \quad    \text{if $a_1$ is odd}.
    \end{array} \right.
    \]
    This completes the proof.
\end{proof}

For an anisotropic lattice $M$ with $\mathrm{GK}(M)=(a_1, a_2)$, we define 
\begin{equation}\label{eqta}
\displaystyle\mathcal{T}_{a_1, a_2}=\sum_{x=0}^{a_1}\sum_{y=0}^{a_2}p^{\mathrm{min}\{a_1-x+y, a_2-y+x\}}.
\end{equation}

The number $\displaystyle\mathcal{T}_{a_1, a_2}$ is indeed the local intersection multiplicity on the special fiber, defined by  Equations (5.3) and (5.16) and Lemma 5.6 of \cite{GK1}.
In the following proposition, we will explain an inductive formula of $\displaystyle\mathcal{T}_{a_1, a_2}$, motivated by an inductive formula of $ \mathcal{F}'_M(1/p)$ in Proposition \ref{propaniso2} as they are supposed to match each other.
Later in Theorem \ref{thmaniso1},
$\displaystyle\mathcal{T}_{a_1, a_2}$ will be compared with the derivative of the Siegel series associated to the quadratic lattice $M$.

\begin{Prop}\label{propaniso3}
If $a_2\geq 2$, then 
\[\mathcal{T}_{a_1, a_2}=p\cdot \mathcal{T}_{a_1, a_2-2}-2\mathcal{F}'_{(a_1)}(1).\]
\end{Prop}

\begin{proof}
We write 
\[
\left\{
  \begin{array}{l}
\displaystyle\mathcal{T}_{a_1, a_2}=\sum_{x=0}^{a_1}\sum_{y=0}^{a_2}p^{\mathrm{min}\{a_1-x+y, a_2-y+x\}};\\
p\cdot \displaystyle\mathcal{T}_{a_1, a_2-2}=\sum_{x=0}^{a_1}\sum_{y=0}^{a_2-2}p^{\mathrm{min}\{a_1-x+(y+1), a_2-(y+1)+x\}}.
    \end{array} \right.
\]
We rewrite the above sums as follows:
\[
\left\{
  \begin{array}{l}
\displaystyle \mathcal{T}_{a_1, a_2}=
\sum_{x=0}^{a_1}\sum_{y=0}^{x+\frac{a_2-a_1}{2}}p^{a_1-x+y}+\sum_{x=0}^{a_1}\sum_{x+\frac{a_2-a_1}{2}<y}^{a_2}p^{a_2-y+x};\\
\displaystyle p\cdot \displaystyle\mathcal{T}_{a_1, a_2-2}=
\sum_{x=0}^{a_1}\sum_{y=0}^{x+\frac{a_2-a_1}{2}-1}p^{a_1-x+(y+1)}+\sum_{x=0}^{a_1}\sum_{x+\frac{a_2-a_1}{2}-1<y}^{a_2-2}p^{a_2-(y+1)+x}.
    \end{array} \right.
\]
Then we have 
\[
\left\{
  \begin{array}{l}
\displaystyle 
\sum_{x=0}^{a_1}\sum_{y=0}^{x+\frac{a_2-a_1}{2}}p^{a_1-x+y}-\sum_{x=0}^{a_1}\sum_{y=0}^{x+\frac{a_2-a_1}{2}-1}p^{a_1-x+(y+1)}
=\sum_{x=0}^{a_1}p^{a_1-x};\\
\displaystyle 
\sum_{x=0}^{a_1}\sum_{x+\frac{a_2-a_1}{2}<y}^{a_2}p^{a_2-y+x}-
\sum_{x=0}^{a_1}\sum_{x+\frac{a_2-a_1}{2}-1<y}^{a_2-2}p^{a_2-(y+1)+x}
=\sum_{x=0}^{a_1}p^{a_2-a_2+x}.
    \end{array} \right.
\]
Thus we have 
\[\mathcal{T}_{a_1, a_2}-p\cdot \mathcal{T}_{a_1, a_2-2}
=2\sum_{x=0}^{a_1}p^x=-2\mathcal{F}'_{(a_1)}(1).\]
\end{proof}

We now compare the local intersection multiplicity $\displaystyle\mathcal{T}_{a_1, a_2}$ on the special fiber, with the derivative of the Siegel series associated to the lattice $M$ in the following theorem.

\begin{Thm}\label{thmaniso1}
We have the following equality:
\[
\mathcal{T}_{a_1, a_2}=\frac{-1}{p-1}\cdot \mathcal{F}'_M(1/p).
\]
In addition, both sides satisfy the same inductive formula.
\end{Thm}

\begin{proof}
By Propositions \ref{propaniso2} and \ref{propaniso3},
it suffices to prove that both sides have the same initial values.

The initial values of $\mathcal{T}_{a_1, a_2}$ can be computed directly from its definition as follows:
\[
\mathcal{T}_{a_1, a_2}=\left\{
  \begin{array}{l l}
 1  & \quad  \textit{if $(a_1, a_2)=(0, 0)$};\\
 2(p+1)  & \quad  \textit{if $(a_1, a_2)=(1, 1)$};\\
 2  & \quad  \textit{if $(a_1, a_2)=(0, 1)$},
    \end{array} \right.
    \]

On the other hand, by Lemma \ref{lemaniso2}, we have
\[
\mathcal{F}'_M(1/p)=\left\{
  \begin{array}{l l}
 -(p-1)  & \quad  \textit{if $(a_1, a_2)=(0, 0)$};\\
 -2(p-1)(p+1)  & \quad  \textit{if $(a_1, a_2)=(1, 1)$};\\
 -2(p-1)  & \quad  \textit{if $(a_1, a_2)=(0, 1)$},
    \end{array} \right.
    \]

Therefore, both $\mathcal{T}_{a_1, a_2}$ and $\frac{-1}{p-1}\cdot \mathcal{F}'_M(1/p)$ have the same initial values.
\end{proof}

For an anisotropic quadratic $\mathbb{Z}_p$-lattice $L$ of rank 3 with ${\rm GK}(L)=(a_1,a_2,a_3)$, put $\alpha_p(a_1,a_2,a_3):=\alpha_p(L)$ which is the local intersection multiplicity defined in Equation 
(5.3) of \cite{GK1}. 
We finally compare it with the derivative of the Siegel series associated to the quadratic lattice $L$ in the following theorem.

\begin{Thm}\label{thmaniso2}
Let $\mathrm{GK}(L)=(a_1, a_2, a_3)$.
Then we have the following equality:
\[
\alpha_p(a_1,a_2,a_3)=\frac{-2p}{(p-1)^2}\cdot\frac{\mathcal{F}'_L(1/p^2)}{\alpha(L, O_D)}.
\]
Moreover, both sides satisfy the same inductive formula.
\end{Thm}

\begin{proof}
Let $\widetilde{\mathcal{T}}_{a_1, a_2, a_3}=\frac{-2p}{(p-1)^2}\cdot\frac{\mathcal{F}'_L(1/p^2)}{\alpha(L, O_D)}$.
Then by Theorem \ref{thmaniso1}, the formula of Proposition \ref{propaniso1} turns to be
\[
\widetilde{\mathcal{T}}_{a_1, a_2, a_3}=\widetilde{\mathcal{T}}_{a_1, a_2, a_3-2}+\mathcal{T}_{a_1, a_2}.
\]
On the other hand, by Lemma 5.6 and Equations (5.16) and (5.18) of \cite{GK1},
the local intersection multiplicity $\alpha_p(a_1,a_2,a_3)$ satisfies the following inductive formula:
\[
\alpha_p(a_1,a_2,a_3)=\alpha_p(a_1,a_2,a_3-2)+\mathcal{T}_{a_1, a_2}.
\]

Therefore, it suffices to show that both $\widetilde{\mathcal{T}}_{a_1, a_2, a_3}$ and $\alpha_p(a_1,a_2,a_3)$ have the same initial values.

We write   $\mathcal{F}_{(a_1, a_2, a_3)}(X)$ for $\mathcal{F}_L(X)$.
If $a_i$ consists of either $0$ or $1$, then the associated quadratic lattice $L$  is maximal by Remark \ref{01max}.
Using Corollary \ref{rmkse}, the Siegel series for a maximal quadratic lattice is described explicitly as follows:
\[
\left\{
  \begin{array}{l}
\mathcal{F}_{(0, 0, 1)}(X)=(1-X)(1-p^2X^2)(1-p^2X);\\
\mathcal{F}_{(0, 1, 1)}(X)=(1-X)(1-p^2X^2)(1-p^4X^2).
    \end{array} \right.
\]

Thus

\[
\left\{
  \begin{array}{l}
\mathcal{F}'_{(0, 0, 1)}(\frac{1}{p^2})=-p^2(1-\frac{1}{p^2})^2,
~~~~~ \frac{-2p}{(p-1)^2}\cdot\frac{\mathcal{F}'_{(0, 0, 1)}(\frac{1}{p^2})}{\alpha(L, O_D)}=1 
  ;\\
 \mathcal{F}'_{(0, 1,1)}(\frac{1}{p^2})=-2p^2(1-\frac{1}{p^2})^2, ~~~~~ \frac{-2p}{(p-1)^2}\cdot\frac{\mathcal{F}'_{(0, 1, 1)}(\frac{1}{p^2})}{\alpha(L, O_D)}=2.
    \end{array} \right.
\]

Therefore, both sides have the same initial values by Proposition 1.6 of \cite{Rap1}.
\end{proof}

\section{Application 1: The intersection number over a finite field}\label{sectionapp}
In this section we revisit the results of Gross-Keating \cite{GK1}  and give a new identity between certain intersection numbers
of cycles over a finite field and the sum of the Fourier coefficients of the Siegel-Eisenstein series for $\mathrm{Sp}_{4}/\Q$ of weight $2$. 
We follow the notation and results in Chapters 3-5 of \cite{A} (cf. \cite{G1}, \cite{G2}, \cite{Wed2}). 

\subsection{Main results}
For a positive integer $m$, we denote by $T_m$ the modular correspondence of degree $m$ defined in \cite{G2}.  
It can be regarded as a flat scheme $T_m$ over $\Z$ which is explicitly given by $T_m={\rm Spec}~ \Z[x,y]/(\varphi_m)
\subset {\rm Spec}~ \Z[x,y]=:S$, where we think the latter scheme $S$ as the product of two copies of the 
coarse moduli space $Y_0(1)$ of elliptic curves.
Here $\varphi_m$ is the modular polynomial of degree $m$ (see \cite{V1} and \cite{G2}). 
We consider $T_{m,p}:=T_m\otimes {\rm Spec}~ \bF_p$
and $T_{m, \mathbb{C}}:=T_m\otimes {\rm Spec}~ \mathbb{C}$. 
Two cycles $T_{m_1,\C}$ and $T_{m_2,\C}$ intersect properly if and only if $m_1m_2$ is not a  square (cf. Proposition 2.4 of \cite{GK1}).
The associated 
intersection number over $\C$ is defined by 
\begin{equation}\label{intC}
(T_{m_1,\C},T_{m_2,\C}):={\rm length}_{\C}\C[x,y]/(\varphi_{m_1},\varphi_{m_2}).  
\end{equation}

We first state the following proposition to explain exactly when two cycles over a finite field intersect properly. 
It turns out to be 
the same as the situation over $\mathbb{C}$.

\begin{Prop}\label{ip}
For given two positive integers $m_1$ and $m_2$,
the cycles $T_{m_1,p}$ and $T_{m_2,p}$ intersect properly in $S\otimes {\rm Spec}~ \bF_p$ if and only if 
$m_1m_2$ is not a square (equivalently, $T_{m_1,\mathbb{C}}$ and $T_{m_2,\mathbb{C}}$ intersect properly inside 
$S\otimes {\rm Spec}~ \mathbb{C}$ by Proposition 2.4 of \cite{GK1}).
\end{Prop}

\begin{proof}Assume that $T_{m_1,p}$ and $T_{m_2,p}$ intersect properly. 
If $m_1m_2$ is square, then we can write $m_i=d^2_i n$ for some positive integers $d_1,d_2,n$. By the description of 
the modular polynomial in p.2 of \cite{V1} we see that the polynomial $\psi_{n}$ (see loc.cit. for the notation of $\psi_n$) 
is a common divisor of $\varphi_{m_1}$ and $\varphi_{m_2}$ as an element of $\Z[X,Y]$. 
Since $p\nmid \psi(X,Y)$, it defines a curve defined over $\F_p$. This gives a contradiction with the assumption. 

Next assume that $m_1m_2$ is not a square. 
If $T_{m_1,p}$ and $T_{m_2,p}$ do not intersect properly, then they share an irreducible common divisor $D$. 
Let $K$ be the function field of $D$ which is of transcendental degree one over $\bF_p$. 
It follows that there exists a geometric point $x={\rm Spec}\overline{K}=((E,E'),f_1,f_2)$ with two endomorphisms $f_1,f_2$ of degree $m_1,m_2$ between elliptic curves $E,E'$ over $K$. By the proof of Proposition \ref{decom} below $j(E),j(E')$ belong to $K\setminus \bF_p$. 
Further it holds that ${\rm Hom}(E,E')\otimes_\Z \Q\simeq {\rm End}(E)\otimes_\Z \Q=\Q$ by Theorem 12 of \cite{CL}. Combining this with 
the $\Z$-freeness of ${\rm Hom}(E,E')$ (cf. Chapter III, Proposition 4.2-(b) in p.68 of \cite{Sil}) we see that ${\rm Hom}(E,E')=\Z$. 
Since $f_1,f_2$ belong to ${\rm Hom}(E,E')=\Z$ there exist two (non-zero) integers $n_1,n_2$ such that $n_1f_1=n_2f_2$. 
It implies that $n^2_1m_1=n^2_2m_2$ and also $m_1m_2$ has to be square. However this yields a contradiction.       
\end{proof}

\begin{Cor}\label{reg}
If $m_1m_2$ is not a square, then 
$\varphi_{m_1}$ and $\varphi_{m_2}$ make up a regular sequence in $\F_p[x,y]$ and 
also in $\bar{\F}_p[[x,y]]$.  
\end{Cor}
\begin{proof}It suffices to prove that $\varphi_{m_1}$ is not a zero divisor in $A:=\F_p[x,y]/(\varphi_{m_2})$. 
Suppose the contrary. 
Then there exists $\alpha \in \F_p[x,y]$ which is not divided by $\varphi_{m_2}$ such that 
$\varphi_{m_2}$ divides $\alpha \varphi_{m_1}$. Since $\F_p[x,y]$ is UFD, there exists an irreducible common factor $h$ of $\varphi_{m_2}$ and 
$\varphi_{m_1}$. 
 Proposition \ref{ip} implies that $\F_p[x,y]/(\varphi_{m_1}, \varphi_{m_2})$ is Artinian but this gives a contradiction with the existence of $h$.  The same argument works for $\bar{\F}_p[[x,y]]$. 
\end{proof}

Let us take two positive integers $m_1$ and $m_2$ such that $m_1m_2$ is not a square.
We write $(T_{m_1,p},T_{m_2,p})$, the intersection number over a finite field, which is explicitly defined as follows: 
\begin{equation}\label{int}
(T_{m_1,p},T_{m_2,p}):={\rm length}_{\F_p}\F_p[x,y]/(\varphi_{m_1},\varphi_{m_2}).  
\end{equation}

Our goal is to compute $(T_{m_1,p},T_{m_2,p})$ explicitly.


\begin{Thm}\label{newid}
Assume that $p$ is odd.
Then for any two positive integers $m_1$ and $m_2$ such that $m_1m_2$ is not a square, the intersection number $(T_{m_1,p},T_{m_2,p})$ is independent of the choice of a prime number $p$ and its explicit value is given as follows:
$$(T_{m_1,p},T_{m_2,p})=\frac{1}{288}
 \sum_{T\in {\rm Sym}_{2}(\Z)_{> 0} \atop {\rm diag}(T)=(m_1, m_2)}c(T)=
 (T_{m_1,\C},T_{m_2,\C}).$$
Here, $c(T)$ is the Fourier coefficient of the Siegel-Eisenstein series for ${\rm Sp}_4(\Z)$ of weight 2 with respect to the $(2\times 2)$- half-integral symmetric matrix $T$ (cf. \cite{nagaoka}). 
\end{Thm}

If we reinterpret Theorem \ref{newid} in terms of modular polynomials in Equation (\ref{int}), then the $\mathbb{Z}[\frac{1}{2}]$-module $\mathbb{Z}[\frac{1}{2}][x,y]/(\varphi_{m_1}, \varphi_{m_2})$ satisfies the following interesting properties.

\begin{Thm}\label{flat}
We have the following interpretation about $\mathbb{Z}[\frac{1}{2}][x,y]/(\varphi_{m_1}, \varphi_{m_2})$:
\begin{enumerate}

\item $\mathbb{Z}[\frac{1}{2}][x,y]/(\varphi_{m_1}, \varphi_{m_2})$ is a free $\mathbb{Z}[\frac{1}{2}]$-module. 
\item 
The rank of $\mathbb{Z}[\frac{1}{2}][x,y]/(\varphi_{m_1}, \varphi_{m_2})$, as a $\mathbb{Z}[\frac{1}{2}]$-module,  is equal to 
\[
\frac{1}{288}
 \sum_{T\in {\rm Sym}_{2}(\Z)_{> 0} \atop {\rm diag}(T)=(m_1, m_2)}c(T).
\]
Here, $c(T)$ is as described in the above theorem.
\end{enumerate}
\end{Thm}

\subsection{Decomposition of the intersection number over a finite field}
In what follows let us go into the proof of Theorem \ref{newid}. Let $m_1,m_2$ be positive integers such that $m_1m_2$ is not a square. 
We denote by ${\rm CLN}_{\bF_p}$ (resp. ${\rm CLN}_{W(\bF_p)}$) the category of complete local Noetherian $\bF_p$ 
(resp. $W(\bF_p)$)-algebras with 
the residue field $\bF_p$. The local deformation functor for a pair of elliptic curves $x:=(E,E')$ over $\bF_p$ on ${\rm CLN}_{\bF_p}$ 
is pro-represented by $R_p:=\bF_p[[t,t']]$, which is called the universal deformation ring of $x$ on ${\rm CLN}_{\bF_p}$. 
Similarly we have the universal deformation ring $R:=W(\bF_p)[[t,t']]$ of $x$ on ${\rm CLN}_{W(\bF_p)}$. 
Let us clarify the relation between the completion $\widehat{\mathcal{O}}_{S_{\bF_p},x}$ of the structure sheaf $\mathcal{O}_{S_{\bF_p}}$ at $x$ and $R_p$. Note that $\widehat{\mathcal{O}}_{S_{\bF_p},x}\simeq \bF_p[[j-j(E),j'-j(E')]]$. 
Here, $j(E)$ and $j(E')$ are the $j$-invariants of $E$ and $E'$, respectively. 
Since ${\rm Aut}(x)={\rm Aut}(E)\times {\rm Aut}(E')$ acts naturally on the deformation datum 
(cf. (8.2) of \cite{KM}) we have that $\widehat{\mathcal{O}}_{S_{\bF_p},x}\simeq R^{{\rm Aut}(x)}_p\subset R_p$ and that $R_p$ is a free 
$\widehat{\mathcal{O}}_{S_{\bF_p},x}$-module of rank $\frac{\sharp {\rm Aut}(x)}{4}$ (cf. page 21 of \cite{G2}). 
It is better to work on $R_p$ instead of $\widehat{\mathcal{O}}_{S_{\bF_p},x}$ because the moduli space of 
elliptic curves is not a fine moduli space.
The difference between these two objects in the computation of local intersection 
multiplicity is understood as below. We first consider the decomposition of the modular polynomial $\varphi_m$ over $R_p$ in terms of the local deformation theory. 

\begin{Prop}\label{decom} For a positive integer $m$, let $(\varphi_m)$ be the ideal of $R_p$ generated by 
the modular  polynomial $\varphi_m$. 
Then 
$$(\varphi_m)=\prod_{f:E\lra E'\ {isogeny\ of\ }\atop 
{\rm degree}\ m,\ {\rm mod} \pm1}I_{m,f,p},$$
where $I_{m,f,p}=(\varphi_{m,f,p})$ with $\varphi_{m,f,p}\in R_p$ is the minimal ideal of $R_p$ such that 
$f$ lifts to an isogeny over $R_p/I_{m,f,p}$. 
\end{Prop}
\begin{proof}
We imitate the proof of Lemma 4.1 of \cite{G2} for $R_p$. 
Let ${\rm Def}_{f,\ast}$ for $\ast\in \{\bF_p,W(\bF_p)\}$ be the deformation functor on ${\rm CLN}_{\ast}$ for $f$.
As is similar to the proof of Lemma 4.1 of \cite{G2}, the deformation functor ${\rm Def}_{f,\bF_p}$ is represented by the divisor ${\rm div}(\varphi_{f,p})$ in 
${\rm Spf}\hspace{0.5mm} R_p$ for some $\varphi_{f,p}\in R_p$.  
For two isogenies $f,g:E\lra E'$ of degree $m$,  $\varphi_{f,p}$ and $\varphi_{g,p}$ are coprime unless $f=\pm g$. 
This is proved as follows. Assume that $\varphi_{f,p}$ and $\varphi_{g,p}$ has a common irreducible factor $h$. 
Let $\mathbb{E},\mathbb{E}'$ be the universal deformations of $E,E'$ over $R_p$. 
Since $R_p/(h)$ is of dimension one, the $j$-invariants $j(\mathbb{E}\otimes {\rm Spf}R_p/(h)),\ j(\mathbb{E}'\otimes {\rm Spf}R_p/(h))$ are not constant, 
hence they do not belong to $\bF_p$. It follows from Theorem 12 of \cite{CL} that 
${\rm End}(\mathbb{E}\otimes {\rm Spf}R_p/(h))={\rm End}(\mathbb{E}'\otimes {\rm Spf}R_p/(h))=\Z.$
This implies that $${\rm Hom}(\mathbb{E}\otimes {\rm Spf}R_p/(h),\mathbb{E}'\otimes {\rm Spf}R_p/(h))=\Z.$$
Since $f,g$ are liftable to isogenies 
$\tilde{f},\tilde{g}\in {\rm Hom}(\mathbb{E}\otimes {\rm Spf}R_p/(h),\mathbb{E}'\otimes {\rm Spf}R_p/(h))$ of the same degree, 
$\tilde{f}$  has to be $\pm \tilde{g}$. Reduction to $\bF_p$ we have $f=\pm g$. 

By Lemma 4.1 of \cite{G2} which is in the situation with $R$, we also decompose the ideal $\varphi_m R$ of $R$ generated by $\varphi_m$ as follows: 
$$\varphi_m R=\prod_{f:E\lra E'\ {\rm iso.\ of\ }\atop 
{\rm degree}\ m,\ {\rm mod} \pm1}I_{m,f},$$
where $I_{m,f}$ is the minimal ideal of $R$ such that $f$ lifts to an isogeny over $R/I_{m,f}$.
The ideal  $I_{m,f}$ is generated by a single element $\varphi_{m,f}$ in $R$ which cannot be divisible by $p$.  

Let $\varphi_{m,f,p}$ be the image of $\varphi_{m,f}$ under the natural projection $R\lra R_p$.  
Then we have 
$$
{\rm div}(\varphi_{f,p})(S)= {\rm Def}_{f,\bF_p}(S)= {\rm Def}_{f,W(\bF_p)}(S)=
{\rm div}(\varphi_{m,f})(S)={\rm div}(\varphi_{m,f,p})(S)$$
for any formal scheme $S$ over ${\rm Spf}\hspace{0.5mm}R_p$.  
Hence we have $I_{m,f,p}=(\varphi_{f,p})=(\varphi_{m,f,p})$.  
\end{proof}
 
\begin{Lem}\label{difference} For positive integers $m_1,m_2$ with $m_1m_2$ non-square and a pair $(E,E')$ of elliptic curves over a finite field, it holds that 
\begin{eqnarray}
{\rm length}_{\bF_p}\bF_p[[j,j']]_{(j-j(E),j'-j(E'))}/(\varphi_{m_1},\varphi_{m_2})=\nonumber \\
\sum_{f_1}
\sum_{f_2}
\frac{4}{\sharp {\rm Aut}(E) \sharp {\rm Aut}(E')}{\rm length}_{\bF_p}\bF_p[[t,t']]/(\varphi_{m_1,f_1,p},\varphi_{m_2,f_2,p}) \nonumber
\end{eqnarray}
where $\varphi_{m_i,f_i,p}$ is a factor of $\varphi_{m_i}$ given in the previous proposition. 
Here, the sums are over isogenies $f_i:E\rightarrow E'$ of degree $m_i$ up to $\pm 1$.
\end{Lem}
\begin{proof}
 By Corollary \ref{reg},  $\varphi_{m_1}$ and $\varphi_{m_2}$ make up a regular sequence.  
Using a similar argument of Equation (4.1) in page 23 of \cite{G2},
the claim follows from Lemma 4.2 of \cite{G2}, Proposition \ref{decom}, and  the fact that $R_p$ is a free 
$\mathcal{O}_{S_{\bF_p}}$-module of rank $\frac{\sharp {\rm Aut}(E) \sharp {\rm Aut}(E')}{4}$. 
\end{proof}

We denote by $(\mathbb{E}, \mathbb{E}'$) the universal pair of elliptic curves over  $R_p$. 
Let $y=((E,E'),f_1,f_2)$ be a pair of $(E,E')$ and  two isogenies $f_1,f_2:E\lra E'$ with deg$(f_i)=m_i, i=1,2$. 
We define  $I_y$ as the minimal ideal of $R_p$ such that both $f_i$'s lift to  isogenies 
$\mathbb{E}\lra \mathbb{E}'$ of degree $m_i$'s  modulo $I_y$ for $i=1,2$, respectively. 
Put 
$${\rm IM}_{p,y}:={\rm length}_{\bF_p} R_p/I_y,$$
which is exactly the same as the local contribution in the summation of Lemma \ref{difference}.

From now on, for a pair of two elliptic curves $(E,E')$ defined over $\bF_p$, we use the following notation:
\[
\left\{
\begin{array}{l l}
(E,E'):{\rm (ord)} & \quad \textit{if both $E$ and $E'$ are ordinary};\\
(E,E'):{\rm (ss)} & \quad \textit{if both $E$ and $E'$ are supersingular}.
\end{array}
\right.
\]

The intersection number $(T_{m_1,p},T_{m_2,p})$ over a finite field is described as follows:
\begin{Prop}\label{decom2}
Assume that $m_1m_2$ is non-square. 
Then we have
\begin{equation}\label{eqssord}
(T_{m_1,p},T_{m_2,p})=\sum_{y=((E,E'),f_1,f_2) \atop (E,E'):{\rm (ord)}}\frac{{\rm IM}_{p,y}}{\sharp {\rm Aut}(E) \sharp {\rm Aut}(E')}+
\sum_{y=((E,E'),f_1,f_2) \atop (E,E'):{\rm (ss)}}\frac{{\rm IM}_{p,y}}{\sharp {\rm Aut}(E) \sharp {\rm Aut}(E')}.
\end{equation} 
\end{Prop}
\begin{proof} 
 LHS is decomposed as follows:
\begin{eqnarray}
(T_{m_1,p},T_{m_2,p})&=&\sum_{(E,E')}(T_{m_1,p},T_{m_2,p})_{(j(E,)j(E'))} \nonumber \\
&=&
\sum_{(E,E'):{\rm (ord)}}(T_{m_1,p},T_{m_2,p})_{(j(E,)j(E'))}+\sum_{(E,E'):{\rm (ss)}}(T_{m_1,p},T_{m_2,p})_{(j(E,)j(E'))}.
\end{eqnarray}

Using Lemma \ref{difference},  we have that 
\begin{eqnarray}
(T_{m_1,p},T_{m_2,p})_{(j(E,)j(E'))}
&=&\frac{4}{\sharp {\rm Aut}(E) \sharp {\rm Aut}(E')}\sum_{f_i:E\lra E'\ {\rm  iso.\ of\ }\atop 
{\rm degree}\ m_i,\ {\rm mod} \pm1}
{\rm IM}_{p,((E,E'),f_1,f_2)}  \\
&=&
\sum_{f_i:E\lra E'\ {\rm  iso.}\atop 
{\ of\ \rm degree}\ m_i}\frac{{\rm IM}_{p,((E,E'),f_1,f_2)}}{\sharp {\rm Aut}(E) \sharp {\rm Aut}(E')}.
\nonumber  
\end{eqnarray}
This completes the claim.
\end{proof}


\subsection{The local intersection multiplicity over a finite field}
Based on Proposition \ref{decom2}, we  compute
 ${\rm IM}_{p,y}$ by using Serre-Tate theory for the ordinary case  and 
\cite{A} for the supersingular case. 
Let us start with a  series of the following lemmas. 
\begin{Lem}\label{ind} Let $m_1,m_2$ be two integers with $m_1m_2$  non-square and let $E,E'$ be two elliptic curves over a finite field. 
Then 
two isogenies $f_1,f_2:E\lra E'$ with  ${\rm deg}(f_i)=m_i$ are linearly independent in the $\Z$-module ${\rm Hom}(E,E')$ and also in ${\rm Hom}(E,E')\otimes_\Z\Z_p$. 
\end{Lem}
\begin{proof}Assume that $f_1$ and $f_2$ are linearly dependent. Then $n_1f_1=n_2f_2$ for some integers $n_1,n_2$. 
We may assume that $n_1$ and $n_2$ are coprime since ${\rm Hom}(E,E')$ is a free $\Z$-module (III, Proposition 4.2 of \cite{Sil}). 
By comparing the degree, we have $n^2_1m_1=n^2_2m_2$. Observe $(n_1m_1)^2=n^2_2m_1m_2$.
This  implies that $m_1m_2$ is a square, which is   a contradiction. 
For the latter since $\Z_p$ is a $\Z$-flat module, $\Z$-freeness of ${\rm Hom}(E,E')$ implies that 
${\rm Hom}(E,E')\otimes_\Z\Z_p$ is a free $\Z_p$-module. The degree map on ${\rm Hom}(E,E')$ is extended $\Z_p$-linearly on 
${\rm Hom}(E,E')\otimes_\Z\Z_p$. Then the same argument above works and this completes the proof.    
\end{proof}
For a $p$-adic integer $z=\ds\sum_{i=0}^\infty a_i p^i$ with $\ a_i\in \{0,\ldots,p-1\}$ and the indeterminant $t$, we abusively define   
$$(1+t)^{z}:=\sum_{n=0}^\infty 
\Big(\begin{array}{c}
z\\
n
\end{array}\Big)t^n=
\prod_{i=0}^\infty(1+t^{p^i})^{a_i} ~~~\textit{modulo $p$}$$
as an element in $\F_p[[t]]$, rather than $\Z_p[[t]]$. Here $\Big(\begin{array}{c}
z\\
n
\end{array}\Big):=\ds\frac{z(z-1)\cdots(z-n+1)}{n!}$ and set 
$\Big(\begin{array}{c}
z\\
0
\end{array}\Big):=1$. It follows from the argument in p. 288  of \cite{Ro} (see ``THE PROOF OF LEMMA") that $(1+t)^z(1+t)^{z'}=(1+t)^{z+z'}$ and 
$((1+t)^z)^{z'}=(1+t)^{zz'}$ for any $z,z'\in \Z_p$ and these facts yields the above second equality for $(1+t)^z$. 

\begin{Lem}\label{power}For $z,w\in \Z^\times_p$ and $r\in \Z_{\ge 0}$ define an element $H(t,t')=(1+t)^{zp^r}-(1+t')^{w}$ in $\F_p[[t,t']]$. 
Then there exists an element $f(t)$ in $\F_p[[t]]$ such that $H(t,f(t))=0$ and 
$f(t)\equiv zw^{-1}t^{p^r}\ {\rm mod}\ (t^{{p^r}+1})$. 
\end{Lem}
\begin{proof}Apply IV, Lemma 1.2 (Hensel's lemma) of \cite{Sil} with $R=\F_p[[t]]$ (here $R$ is the notation there), $I=(t)$, $F(t')=H(t,t')$, $a=zw^{-1}t^{p^r}$, and $\alpha=-w$. 
Notice that $F(a)\in I^{p^r+1}$ and $F'(a)=-w(1+zw^{-1}t^{p^r})^{w-1}\in R^\times=\F_p[[t]]^\times$.   
\end{proof}

\begin{Lem}\label{plength}Let $e_1,e_2$ be two non-negative integers. For any two elements $f,g\in R_p=\bF_p[[t,t']]$ which are coprime, it holds that 
$${\rm length}_{\bF_p}R_p/(f^{p^{e_1}},g^{p^{e_2}})=p^{e_1+e_2}\cdot {\rm length}_{\bF_p}R_p/(f,g).$$ 
\end{Lem}
\begin{proof}It follows from Lemma 4.2 of \cite{G2}.  
\end{proof}

For an isogeny $f:E\lra E'$ between ordinary elliptic curves over a finite field $k$, by functoriality we can associate a unique element of 
$${\rm Hom}(\widehat{E},\widehat{E}')\times {\rm Hom}_{\Z_p}(T_p(E)\otimes_{\Z_p}\Q_p/\Z_p,T_p(E')\otimes_{\Z_p}\Q_p/\Z_p)=
{\rm End}(\widehat{{\mathbb{G}}}_m/k)\times {\rm End}_{\Z_p}(\Q_p/\Z_p)\simeq \Z^2_p$$
where $\widehat{E},\widehat{E}'$ are formal groups associated to $E,E'$ respectively and $\widehat{{\mathbb{G}}}_m$ is 
the multiplicative formal group over $k$.  
We write $(z(f),w(f))\in \Z^2_p$ for the element corresponding to $f$ via the above identification. 

\begin{Prop}\label{local-im} 
Let $y=((E, E'),f_1,f_2)$ as explained in Lemma \ref{ind}.
Let $T$ be the half-integral symmetric matrix associated to the quadratic lattice spanned by $(f_1,f_2)$ so that ${\rm diag}(T)=(m_1,m_2)$. 
Then 
\begin{enumerate}
\item {\rm (ord)} if $(E,E')$ is a pair of ordinary elliptic curves, then ${\rm ord}_p(\det(2T))$ is an even integer and 
${\rm IM}_{p,y}=p^{r_T}$, where $r_T:=\frac{1}{2}{\rm ord}_p(\det(2T))$;

\item {\rm (ss)} if $(E,E')$ is a pair of supersingular elliptic curves, then ${\rm IM}_{p,y}=\T_{a_1,a_2}$.
Here, $(a_1, a_2)=\mathrm{GK}(T\otimes \Z_p)$.
For the definition of $\T_{a_1,a_2}$, see Equation $\left(\ref{eqta}\right)$.

\end{enumerate}
\end{Prop}
\begin{proof}
The second claim follows from Equation (4.1), p. 151 of \cite{Rap1}. 
Assume  the first case.  
As in Equation (2.1) of \cite{G1} there exists $d\in \Z$ such that $d^2 d(E)=\det(2T)$, where $d(E)$ is the discriminant of ${\rm End}(E)$. Since $E$ is ordinary, we see that
$p\nmid d(E)$ by Theorem 12, p.182 of \cite{Lang}.
We express the matrix $T$ in
terms of $(z_i,w_i):=(z(f_i),w(f_i)),i=1,2$
as follows:
\begin{eqnarray}\label{matrix}
T&=&\left(\begin{array}{cc}
{\rm deg}(f_1) & \frac{1}{2}\left({\rm deg}(f_1+f_2)-{\rm deg}(f_1)-{\rm deg}(f_2)\right) \\
\frac{1}{2}\left({\rm deg}(f_1+f_2)-{\rm deg}(f_1)-{\rm deg}(f_2)\right) & {\rm deg}(f_2)
\end{array}\right) \\
&=&
\left(\begin{array}{cc}
z_1w_1 &  \frac{1}{2}\left(z_1w_2+z_2w_1\right) \\
 \frac{1}{2}\left(z_1w_2+z_2w_1\right) & z_2w_2
\end{array}\right). \nonumber
\end{eqnarray}
Then we have
$${\rm ord}_p(\det(2T))=2{\rm ord}_p(z_1w_2-z_2w_1).$$ 
It follows from Serre-Tate theory, Theorem 2.1-3), p. 148 of \cite{Katz} (see also the observation in p. 140 of \cite{Rap1}) that  the minimal ideal $I_y$ is given by 
$I_y=(H_1,H_2)$, where 
$$H_i=H_i(t,t'):=(1+t)^{z_i}-(1+t')^{w_i},\ i=1,2.$$
By Lemma \ref{plength} and changing $z_i$ and $w_i$ if necessary,    
we may assume that $w_1\in \Z^\times_p$. Remark that  
$H_{1,a}:=(1+t)^{az_1}-(1+t')^{aw_1}\equiv 0$ mod $H_1$ in $R_p$ for any $a\in \Z_p$. 
By replacing $(f_1,f_2)$ with $(f_1,w_1f_2)$ (here we view them as a pair of two homomorphisms between $\widehat{E}$ and $\widehat{E}'$)  
$R_p/(H_1,H_2)\simeq R_p/((1+t)^{z_2w_1}-(1+t')^{w_2w_1},H_1)$.  
Then we have that
$$R_p/(H_1,H_2)\simeq R_p/((1+t)^{z_2w_1}-(1+t')^{w_2w_1},H_1,H_{1,w_2})=R_p/((1+t)^{z_1w_2}-(1+t)^{z_2w_1},H_1).$$
 Lemma \ref{power} yields an existence of  $f(t)\in \F_p[[t]]$ such that $H_1(t,f(t))=0$. 
This implies that
$$R_p/(H_1,H_2)=\bF_p[[t]]/((1+t)^{z_1w_2}-(1+t)^{z_2w_1}).$$
 Write $z_1w_2-z_2w_1=p^{r_T}\alpha$ for $\alpha\in \Z^\times_p$. 
Then $R_p/(H_1,H_2)=\bF_p[[t]]/(t^{p^{r_T}})$. 
Hence we have the claim. 
\end{proof}

\subsection{Comparison in the ordinary locus using quasi-canonical lifts: over $\mathbb{C}$ and over $\bF_p$}\label{subsec7.4}
Let us recall the canonical lift and quasi-canonical lifts of an ordinary elliptic curve $E$ over $\bF_p$. 
By Serre-Tate's theorem (\cite{Katz}), there exists a unique canonical lift $\widetilde{E}^{{\rm can}}$ to $W=W(\bF_p)$ such that 
the reduction map induces an isomorphism ${\rm End}_W(\widetilde{E}^{{\rm can}})\stackrel{\sim}{\lra}{\rm End}(E)$. 
For the canonical lift, it is known that ${\rm End}_{\C}(\widetilde{E}^{{\rm can}})={\rm End}_{W(\bF_p)}(\widetilde{E}^{{\rm can}})={\rm End}_{\bF_p}(E)$ is an 
order $\mathcal{O}_{K,n}$ of the conductor $n$ which is prime to $p$ and that $K:={\rm End}_{\bF_p}(E)\otimes_\Z\Q$ is an imaginary quadratic field for which $p$ is split. Therefore the discriminant $d(\widetilde{E}^{{\rm can}})$ of ${\rm End}_{\C}(\widetilde{E}^{{\rm can}})$ (which  equals that for 
${\rm End}_{\bF_p}(E)$) is given by $nD_K$, where $D_K$ is the discriminant, and importantly we see that 
$p\not| nD_K$ (see Theorem 12, p.182 of \cite{Lang}). 
Furthermore for two ordinary elliptic curves $E,E'$, 
the reduction map induces an isomorphism 
$${\rm Hom}_{\C}(\widetilde{E}^{{\rm can}},\widetilde{E}'^{{\rm can}})\simeq 
{\rm Hom}_{W(\bF_p)}(\widetilde{E}^{{\rm can}},\widetilde{E}'^{{\rm can}})\stackrel{\sim}{\lra} {\rm Hom}(E,E')$$
since the Serre-Tate local coordinates $q(E),q(E')$ for $E,E'$ satisfy $q(E)=q(E')=1$ respectively (cf. Theorem 2.1 and also last a few lines in p.180 of \cite{Katz}).

On the other hand if an elliptic curve $\tilde{E}$ over $\C$ has CM by an order in an imaginary quadratic field $K$, then 
 ${\rm End}_\C(\tilde{E})=\mathcal{O}_{K,np^s}$ for some positive integer $n$ which is coprime to $p$ and $s\in \Z_{\ge 0}$. 
 Assume that $p$ is split in $K$ or equivalently that $\tilde{E}$ has a good $p$-ordinary reduction $E/\bF_p$ (notice that 
one can take a smooth integral model over the ring $\mathcal{O}_L$ of integers of some number field $L$ among the isomorphism class of $\tilde{E}$ since its $j$-invariant is an algebraic integer). 
Then we see   that
\begin{equation}\label{conductor}
\mathcal{O}_{K,np^s}={\rm End}_\C(\tilde{E})={\rm End}_{\mathcal{O}_L}(\tilde{E})\hookrightarrow {\rm End}_{\bF_p}(E)=\mathcal{O}_{K,n}.
\end{equation} 
by Theorem 12, p.182 of \cite{Lang} again. 
The elliptic curve $\tilde{E}$ is also a lift of an ordinary elliptic curve $E$ and it is called a quasi-canonical lift of level $s$ for $E$ (cf. \cite{G}, 
\cite{Yu}).
It is known by Section 6 of \cite{G} or Proposition 3.5 of \cite{Me} (see also p.97 of \cite{Me}) that the number of isomorphism classes of quasi-canonical lifts of level $s$ for an ordinary elliptic curve $E$ is given  
as follows: 
\begin{equation}\label{nq}
\frac{\sharp \mathcal{O}^\times_{K,np^s}}{\sharp \mathcal{O}^\times_{K,n}}\sharp (\Z/p^s\Z)^\times=
\frac{\sharp {\rm Aut}(\tilde{E})}{\sharp {\rm Aut}(E)}(p^s-p^{s-1}).
\end{equation}
Let $t$ be the local parameter for the local deformation of quasi-canonical lift of $p$-divisible group of $E$ and let
$j$ be the parameter of the coarse moduli space $\A_j$ defined by $j$-invariant.  
Then we have the relation $j-j(E)=t^{\frac{\sharp {\rm Aut}(E)}{2}}$. This explains the appearance of 
$\ds\frac{\sharp {\rm Aut}(\tilde{E})}{\sharp {\rm Aut}(E)}$ in (\ref{nq}).

  From now on, we will count the number of lifts to quasi-canonical lifts for given two isogenies $f_1, f_2 : E\rightarrow E'$ where $E$ and $E'$ are ordinary elliptic curves.
 Let $(z_i,w_i)=(z(f_i),w(f_i))$ be an element of $\Z^2_p$ for $f_i$'s with $i=1,2$, as explained in the paragraph following Lemma \ref{plength}. 
  We write $z_i=u_ip^{a_i}$ and $w_i=v_ip^{b_i}$, where $u_i, v_i$ are units in $\mathbb{Z}_p$. 
Let $T$ be the  symmetric matrix associated to $(f_1, f_2)$ as Equation (\ref{matrix}).
Note that Lemma \ref{ind} confirms that $T$ is nonsingular in our situation.
Put $r=r_T$. 
For two isogenies $f_1,f_2$ given as above we denote by 
$N(s,s';f_1,f_2)$ the number of isomorphism classes of quasi-canonical lifts $(\tilde{E}_s,\tilde{E}_{s'})$ of level $s,s'$ respectively ($0\le s,s'\le r$) such that 
$f_1,f_2$ lift to isogenies from $E_s$ to $E_{s'}$.

\begin{Prop}\label{nend}
Keep the notation being as above. Assume that $a_1=\min\{a_1, a_2, b_1, b_2\}$. 

\begin{enumerate}

\item If $b_1< b_2$ and $a_2+b_1\le a_1+b_2$, then  we have that
\[
N(s,s';f_1,f_2)=\left\{
  \begin{array}{l l}
 \frac{\sharp \mathcal{O}^\times_{K,np^s}}{\sharp \mathcal{O}^\times_{K,n}}
\frac{\sharp \mathcal{O}^\times_{K,np^{s'}}}{\sharp \mathcal{O}^\times_{K,n}}\sharp (\Z/p^s\Z)^\times\cdot \sharp (\Z/p^{s'}\Z)^\times  & \quad  \textit{if $0\leq s\leq a_1$ 
and $0\le s'\le b_1$};\\
 \frac{\sharp \mathcal{O}^\times_{K,np^s}}{\sharp \mathcal{O}^\times_{K,n}}
\frac{\sharp \mathcal{O}^\times_{K,np^{s'}}}{\sharp \mathcal{O}^\times_{K,n}}\sharp (\Z/p^s\Z)^\times\cdot p^{b_1}  & \quad  \textit{if $a_1<s\leq a_2$ and 
$s'=s+b_1-a_1$};\\
\frac{\sharp \mathcal{O}^\times_{K,np^s}}{\sharp \mathcal{O}^\times_{K,n}}
\frac{\sharp \mathcal{O}^\times_{K,np^{s'}}}{\sharp \mathcal{O}^\times_{K,n}}\sharp (\Z/p^s\Z)^\times\cdot p^{b_1}  & \quad  
\textit{if $a_2<s\le r-b_1$} \\
& \quad   \textit{\quad   and $s'=s+b_1-a_1=s+b_2-a_2$};\\
0&\quad    \textit{otherwise}.
    \end{array} \right.
    \]
    
\item     If $b_1<  b_2$ and  $a_2+b_1> a_1+b_2$, then $r=a_1+b_2$ and we have that 
\[
N(s,s';f_1,f_2)=\left\{
  \begin{array}{l l}
 \frac{\sharp \mathcal{O}^\times_{K,np^s}}{\sharp \mathcal{O}^\times_{K,n}}
\frac{\sharp \mathcal{O}^\times_{K,np^{s'}}}{\sharp \mathcal{O}^\times_{K,n}}\sharp (\Z/p^s\Z)^\times\cdot \sharp (\Z/p^{s'}\Z)^\times  & \quad  \textit{if $0\leq s\leq a_1$ 
and $0\le s'\le b_1$};\\
 \frac{\sharp \mathcal{O}^\times_{K,np^s}}{\sharp \mathcal{O}^\times_{K,n}}
\frac{\sharp \mathcal{O}^\times_{K,np^{s'}}}{\sharp \mathcal{O}^\times_{K,n}}\sharp (\Z/p^s\Z)^\times\cdot p^{b_1}  & \quad  \textit{if $a_1<s\leq r-b_1$ and 
$s'=s+b_1-a_1$};\\
0&\quad    \textit{otherwise}.
    \end{array} \right.
    \]
    
\item If $b_1\ge b_2$ and $a_2+b_1\le a_1+b_2$, then we have  that 
\[
N(s,s';f_1,f_2)=\left\{
  \begin{array}{l l}
 \frac{\sharp \mathcal{O}^\times_{K,np^s}}{\sharp \mathcal{O}^\times_{K,n}}
\frac{\sharp \mathcal{O}^\times_{K,np^{s'}}}{\sharp \mathcal{O}^\times_{K,n}}\sharp (\Z/p^s\Z)^\times\cdot \sharp (\Z/p^{s'}\Z)^\times  & \quad  \textit{if $0\leq s'\leq b_2$ 
and $0\le s\le a_1$};\\
 \frac{\sharp \mathcal{O}^\times_{K,np^s}}{\sharp \mathcal{O}^\times_{K,n}}
\frac{\sharp \mathcal{O}^\times_{K,np^{s'}}}{\sharp \mathcal{O}^\times_{K,n}}\sharp (\Z/p^s\Z)^\times\cdot p^{a_1}  & \quad  
\textit{if $b_2<s'\leq b_1$ and 
$s=s'-b_2+a_2$};\\
\frac{\sharp \mathcal{O}^\times_{K,np^s}}{\sharp \mathcal{O}^\times_{K,n}}
\frac{\sharp \mathcal{O}^\times_{K,np^{s'}}}{\sharp \mathcal{O}^\times_{K,n}}\sharp (\Z/p^s\Z)^\times\cdot p^{a_1}  & \quad  
\textit{if $b_1<s'\le r-a_1$} \\
& \quad   \textit{\quad   and $s=s'-b_1+a_1=s'-b_2+a_2$};\\
0&\quad    \textit{otherwise}.
    \end{array} \right.
    \]
    
\item     If $b_1\ge  b_2$ and  $a_2+b_1> a_1+b_2$, then $r=a_1+b_2$ and we have that 
\[
N(s,s';f_1,f_2)=\left\{
  \begin{array}{l l}
 \frac{\sharp \mathcal{O}^\times_{K,np^s}}{\sharp \mathcal{O}^\times_{K,n}}
\frac{\sharp \mathcal{O}^\times_{K,np^{s'}}}{\sharp \mathcal{O}^\times_{K,n}}\sharp (\Z/p^s\Z)^\times\cdot \sharp (\Z/p^{s'}\Z)^\times  & \quad  \textit{if $0\leq s'\leq b_2$ 
and $0\le s\le a_1$};\\
 \frac{\sharp \mathcal{O}^\times_{K,np^s}}{\sharp \mathcal{O}^\times_{K,n}}
\frac{\sharp \mathcal{O}^\times_{K,np^{s'}}}{\sharp \mathcal{O}^\times_{K,n}}\sharp (\Z/p^s\Z)^\times\cdot p^{a_1}  & \quad  
\textit{if $b_2<s'\leq r-a_1$ and 
$s=s'-b_2+a_2$};\\
0&\quad    \textit{otherwise}.
    \end{array} \right.
    \]
    \end{enumerate}
    
\end{Prop}
\begin{proof}
 We choose quasi-canonical lifts $\tilde{E}_s$ and $\tilde{E}'_{s'}$ of level $s$ and $s'$ for $E$ and $E'$ respectively. 
Let $q_s:=q(\tilde{E}_s)$ and $\ q_{s'}:=(\tilde{E}'_{s'})$ be the Serre-Tate coordinates for $\tilde{E}_s$ and $\tilde{E}'_{s'}$ respectively.
 Then $q_s$ (resp. $q_{s'}$) is a primitive $p^s$ (resp. $p^{s'}$)-th root of unity in $\bQ_p$ by Proposition 3.5-(3) of \cite{Me}.
The condition  to lift $f_1$ and $f_2$ as an element of ${\rm Hom}(\tilde{E}_s,\tilde{E}'_{s'})$ is given by the following two equations:
\begin{equation}\label{eqdefor}
q^{z_1}_s=q^{w_1}_{s'},\ q^{z_2}_s=q^{w_2}_{s'}.
\end{equation}
Here, we follow notations of page 140 of \cite{Rap1}. 

Firstly we consider the case when $b_1\le b_2$ and $a_2+b_1\le a_1+b_2$. 
If $s\leq a_1$, then $q^{z_1}_s=q^{z_2}_{s}=1$ since $a_1\le a_2$ by the assumption. Thus any $q_{s'}$ with $0\leq s'\leq \min\{b_1,b_2\}=b_1$ also satisfies the second equation in (\ref{eqdefor}). In addition, such a $q_{s'}$ runs over all  
primitive $p^{s'}$-roots of unity.
 Thus we have 
 $$N(s,s';f_1,f_2)=\frac{\sharp \mathcal{O}^\times_{K,np^s}}{\sharp \mathcal{O}^\times_{K,n}}
\frac{\sharp \mathcal{O}^\times_{K,np^{s'}}}{\sharp \mathcal{O}^\times_{K,n}}\sharp (\Z/p^s\Z)^\times\cdot \sharp (\Z/p^{s'}\Z)^\times$$
for $s,s'$ satisfying $0\le s \le a_1$ and $0\le s'\le b_1$. 

If $a_1<s\le a_2$, then the first equation in (\ref{eqdefor}) gives the equality $s'=s+b_1-a_1$ but there are $p^{b_1}$ numbers of $q_{s'}$ since 
$q^{z_1}_s=q^{w_1}_{s'}=(q^{v_1}_{s'-b_1})^{p^{b_1}}$. 
Notice that $q^{z_2}_s=1$ since $s\le a_2$. On the other hand $s'=s+b_1-a_1\le a_2+b_1-a_1\le b_2$ since 
we have assumed that $a_2+b_1\le a_1+b_2$. Hence the second equation in (\ref{eqdefor}) is automatically fulfilled.    
Thus 
$$N(s,s';f_1,f_2)=\frac{\sharp \mathcal{O}^\times_{K,np^s}}{\sharp \mathcal{O}^\times_{K,n}}
\frac{\sharp \mathcal{O}^\times_{K,np^{s'}}}{\sharp \mathcal{O}^\times_{K,n}}\sharp (\Z/p^s\Z)^\times\cdot p^{b_1}$$
where $s'=s+b_1-a_1$. 

Assume that $s> a_2$. Equation (\ref{eqdefor}) implies $s=s'-b_1+a_1=s'-b_2+a_2$. In particular $a_2+b_1=a_1+b_2$. 
In this case $a_2+b_1\le r$ since $r={\rm ord}_p(z_1w_2-z_2w_1)$. 
As discussed in the previous case there are $p^{b_1}$ numbers of $q_{s'}$. 
With the notation fixed we may write $u_1v_2-u_2v_1=p^{r-(a_2+b_1)}\alpha$ for some $\alpha\in \Z^\times_p$. 
For such a $q_{s'}$ we compute 
$$q^{w_2}_{s'}=q^{v_2}_{s'-b_2}=q^{v_2v_1v^{-1}_1}_{s'-b_2}=
q^{p^{b_2-b_1}v_2v_1v^{-1}_1}_{s'-b_1}=
q^{p^{b_2-b_1}v_2u_1v^{-1}_1}_{s-a_1}=
q^{v_2u_1v^{-1}_1}_{s-a_1-b_2+b_1}=
q^{(v_1u_2+p^{r-(a_2+b_1)}\alpha)v^{-1}_1}_{s-a_2}$$
$$=q^{u_2+p^{r-(a_2+b_1)}\alpha v^{-1}_1}_{s-a_2}=
q^{z_2+p^{r-b_1}\alpha v^{-1}_1}_{s}.$$
It follows from this that if Equation (\ref{eqdefor}) has a solution, then it forces $s$ to be $s\le r-b_1$ 
(note that $q^{z_2}_s=q^{w_2}_{s'}$ and the above equation imply $q^{p^{r-b_1}\alpha v^{-1}_1}_{s}=1$). 
Otherwise (hence when $s>r-b_1$) there is no solution of  (\ref{eqdefor}). 

Next we assume that $b_1<b_2$ and $a_2+b_1>a_1+b_2$. Then $a_1+b_2=r$ since $r={\rm ord}_p(z_1w_2-z_2w_1)$. 
A similar computation done before gives the desired results and hence 
details are omitted. 

In the case when $b_1\ge b_2$ we may replace the role of $s$ with one of $s'$ and carry out a similar computation.  
\end{proof}

\begin{Cor}\label{imp} 
We have the following formula:
$$\sum_{0\le s,s'\le r}\frac{N(s,s';f_1,f_2)}{\sharp {\rm Aut}(\tilde{E}_s)\sharp {\rm Aut}(\tilde{E}'_s)}=
\frac{p^r}{\sharp {\rm Aut}(E)\sharp {\rm Aut}(E')}$$
\end{Cor}
\begin{proof}
By exchanging $f_1$ and $f_2$, if necessary, we may assume that either $a_1$ or $b_1$ is the minimum among $a_1, a_2, b_1, b_2$.
We first assume that $a_1$ is the minimum.
Let us treat the case when $b_1<b_2$ and $a_2+b_1\le a_1+b_2$.
The other  cases can be handled similarly so that we may skip them. 
Using Proposition \ref{nend},  the left hand side turns to be 
$$ \sum_{0\le s\le a_1 \atop 0\le s' \le b_1 }\frac{1}{\sharp \mathcal{O}^\times_{K,n}\sharp \mathcal{O}^\times_{K,n}}\sharp (\Z/p^s\Z)^\times\cdot \sharp (\Z/p^{s'}\Z)^\times+
 \sum_{a_1<s\le r-b_1}\frac{1}{\sharp \mathcal{O}^\times_{K,n}\sharp \mathcal{O}^\times_{K,n}}\sharp (\Z/p^s\Z)^\times\cdot p^{b_1}$$
 $$=\frac{1}{\sharp \mathcal{O}^\times_{K,n}\sharp \mathcal{O}^\times_{K,n}}(p^{a_1+b_1}+p^{b_1}(p^{r-b_1}-p^{a_1}))=
\frac{p^r}{\sharp \mathcal{O}^\times_{K,n}\sharp \mathcal{O}^\times_{K,n}}=\frac{p^r}{\sharp {\rm Aut}(E)\sharp {\rm Aut}(E')}.$$
We next assume that $b_1$ is the minimum.
Let $f_i^{\vee}$ be the dual isogeny of $f_i$.
Then we can see that 
\[
N(s,s';f_1,f_2)=N(s',s;f_1^{\vee},f_2^{\vee}).
\]
Thus the argument used in the above case gives the desired formula.

\end{proof}

\subsection{The intersection number in the ordinary locus}
The intersection number $(T_{m_1,\mathbb{C}},T_{m_2,\mathbb{C}})$  is described  in Proposition 2.4 of  \cite{GK1} and it turns out to be the sum of the Fourier coefficients of the Siegel-Eisenstein series for ${\rm Sp}_{4}/\Q$. 
We first consider the contribution coming from the ordinary part. For a symmetric positive definite $(2\times 2)$-half-integral matrix $T$ (namely 
diagonal entries are integer and anti-diagonal entries are elements in $\frac{1}{2}\Z$), 
we define $\chi_T(p)$ by 
$$
\chi_T(p)=\left\{
\begin{array}{cl}
1 & (\mbox{if $p$ is split in $\Q(\sqrt{-\det(2T)})$}) \\
0 & (\mbox{if $p$ is ramified in $\Q(\sqrt{-\det(2T)})$}) \\
-1 & (\mbox{if $p$ is inert in $\Q(\sqrt{-\det(2T)})$})
\end{array}\right..
$$

\begin{Thm}\label{ordpart}
Assume that $m_1m_2$ is not a square.
Then for any prime number $p$, we have that
$$\sum_{\substack{y=((E,E'),f_1,f_2) \\  (E,E'):{\rm (ord)}\\ {\rm deg}(f_i)=m_i}}\frac{{\rm IM}_{p,y}}{\sharp {\rm Aut}(E) \sharp {\rm Aut}(E')}= 
\frac{1}{288}
 \sum_{T\in {\rm Sym}_{2}(\Z)_{> 0} \atop {\rm diag}(T)=(m_1, m_2), \chi_T(p)=1}c(T).$$ 
Here $c(T)$ is the Fourier coefficient of the Siegel-Eisenstein series of weight 2 with respect to ${\rm Sp}_4(\Z)$ (cf. \cite{nagaoka}). 
\end{Thm}
\begin{proof}
With the observation in CM elliptic curves having good ordinary reduction at $p$
(which we call it \textit{$p$-ordinary} and denote it by \textit{$p$-ord} throughout this proof), we first obtain 
\begin{equation}\label{1st}
\sum_{\substack{((\tilde{E},\tilde{E}'),f_1,f_2) \\  (\tilde{E},\tilde{E}')/\C:{\rm CM}\ p-{\rm ord} \\ {\rm deg}(f_i)=m_i}}\frac{1}{\sharp {\rm Aut}(\tilde{E}) \sharp {\rm Aut}(\tilde{E}')}  =
\frac{1}{288}
 \sum_{T\in {\rm Sym}_{2}(\Z)_{> 0} \atop {\rm diag}(T)=(m_1, m_2), \chi_T(p)=1}c(T)
 \end{equation} 
by (2.19), p. 231 of \cite{GK1}. 

For a pair of $p$-ordinary elliptic curves $(\tilde{E}, \tilde{E}')$ having two isogenies $f_1$ and $f_2$ as in LHS of the above equation,
we already observed in Section \ref{subsec7.4} that
$\tilde{E}=\tilde{E}_s$ and $\tilde{E}'=\tilde{E}'_{s'}$ for suitable $s$ and $s'$, where
$\tilde{E}_{s}$ and $\tilde{E}_{s'}$ 
 are quasi-canonical lifts of level $s$ and $s'$ for ordinary elliptic curves $E$ and $E'$ respectively. 
 Since the reduction of endomorphism groups is injective, the reductions of $f_1$ and $f_2$ (which are also denoted by $f_1$ and $f_2$ respectively) are  isogenies from $E$ to $E'$ with the same degrees.  
On the other hand, any two isogenies $f_1, f_2: E\rightarrow E'$ over $\bF_p$ can be lifted to those having the same degree defined over $\C$ by choosing canonical lifts (cf. Section \ref{subsec7.4}).
Therefore, 
using Corollary \ref{imp}, LHS of Equation (\ref{1st}) turns to be

\begin{eqnarray}\label{2st}
\sum_{\substack{((\tilde{E},\tilde{E}'),f_1,f_2) \\(\tilde{E},\tilde{E}')_{/\C}:{\rm CM}\ p-{\rm ord} \\ {\rm deg}(f_i)=m_i}}\frac{1}{\sharp {\rm Aut}(\tilde{E}) \sharp {\rm Aut}(\tilde{E}')}           
&=&\sum\limits_{s,s'\geq 0}\sum_{\substack{((\tilde{E}_s,\tilde{E}'_{s'}),f_1,f_2)\\ {\rm deg}(f_i)=m_i}}\frac{1}{\sharp {\rm Aut}(\tilde{E}_s) \sharp {\rm Aut}(\tilde{E}'_{s'})} 
\nonumber \\
=\sum_{\substack{((E,E'),f_1,f_2) \\  (E,E')_{/\bar{\F}_p}:{\rm (ord)}\\ {\rm deg}(f_i)=m_i}}
\sum_{0\le s,s'\le r=r_T\atop T=Q(f_1,f_2)}\frac{N(s,s';f_1,f_2)}{\sharp {\rm Aut}(\tilde{E}_s)\sharp {\rm Aut}(\tilde{E}'_s)}
&=&
\sum_{\substack{((E,E'),f_1,f_2) \\  (E,E')_{/\bar{\F}_p}:{\rm (ord)}\\ {\rm deg}(f_i)=m_i}}\frac{p^{r_T}}{\sharp {\rm Aut}(E) \sharp {\rm Aut}(E')}\nonumber \\
=
\sum_{\substack{y=((E,E'),f_1,f_2) \\  (E,E')_{/\bar{\F}_p}:{\rm (ord)}\\ {\rm deg}(f_i)=m_i}}\frac{{\rm IM}_{p,y}}{\sharp {\rm Aut}(E) \sharp {\rm Aut}(E')}. && 
\nonumber
 \end{eqnarray} 
 Here, $a_i$ and $b_i$ are as explained in the paragraph just before Proposition \ref{nend} associated to $(f_1, f_2)$.
 This, combined with Equation (\ref{1st}), completes the proof.
\end{proof}

\  \

\subsection{The intersection number in the supersingular locus}
We next consider the contribution coming from the supersingular part of Equation (\ref{eqssord}). 
Let us remind readers that our definition of the Siegel series is the same as Katsurada's one given in \cite{Kat}. 
This convention has been used in Wedhorn's article \cite{Wed2} and Nagaoka's article \cite{nagaoka}. 
The difference to Kitaoka's Siegel series given in \cite{Kit} is explained in the last line of Section 4.3, p. 43 of \cite{Wed2}.


\begin{Thm}\label{sspart}
Assume that $p$ is odd and that $m_1m_2$ is not a square.
Then we have that\begin{equation}\label{eqsp4final}
\sum_{\substack{y=((E,E'),f_1,f_2) \\  (E,E')_{/\bar{\F}_p}:{\rm (ss)}\\ {\rm deg}(f_i)=m_i}}\frac{{\rm IM}_{p,y}}{\sharp {\rm Aut}(E) \sharp {\rm Aut}(E')}=
\frac{1}{288}\cdot   \sum_{T \in {\rm Sym}_2(\Z)_{>0}  
\atop {\rm diag}(T)=(m_1,m_2), \chi_T(p)=-1,0} 
c(T),
\end{equation}
where $T$ is a half-integral symmetric matrix.
\end{Thm}

\begin{proof} 
We will proceed our proof without assuming $p>2$. 
The assumption will be made later when it is needed.

Let $(E,E')$ be a pair of two supersingular elliptic curves and $f_1,f_2:E\lra E'$ be isogenies with 
${\rm deg}(f_i)=m_i$. Let $T$ be the  $(2\times 2)$ half-integral symmetric matrix associated to  $(f_1, f_2)$. 
Since $(E,E')$ are supersingular elliptic curves and $D:={\rm Hom}(E)\otimes \Q={\rm End}(E')\otimes \Q$ is ramified at $p$, 
the prime $p$ has to be inert or ramified in $\Q(\sqrt{-\det T})$ (cf. Theorem 12, page 182 of \cite{Lang}). Hence $\chi_T(p)=-1$ or 0.   

By the argument explained in page 23 of  \cite{G2} and Proposition \ref{local-im}.(2), we see that 
\begin{equation}\label{fce1}
\sum_{\substack{y=((E,E'),f_1,f_2) \\  (E,E')_{/\bar{\F}_p}:{\rm (ss)}\\ {\rm deg}(f_i)=m_i}}\frac{{\rm IM}_{p,y}}{\sharp {\rm Aut}(E) \sharp {\rm Aut}(E')} 
=\sum_{T \in {\rm Sym}_2(\Z)_{>0}  
\atop {\rm diag}(T)=(m_1,m_2), \chi_T(p)=-1,0} \Big(\sum_{(E,E'):{\rm (ss)}}
\frac{R_{{\rm Hom}(E,E')}(T)}{\sharp {\rm Aut}(E) \sharp {\rm Aut}(E')}\T_{a_1,a_2} \Big),  
\end{equation}
where $(a_1, a_2)=\mathrm{GK}(T\otimes \Z_p)$.

Let $\mathcal{F}'_{T, l}(X)$ be the Siegel series associated to the local completion $T\otimes_{\mathbb{Z}}\mathbb{Z}_l$ for any finite place $l$.

Firstly, by Theorem \ref{thmaniso1} we have
\begin{equation}\label{fce2}
\T_{a_1,a_2}=\frac{-1}{p-1}\cdot\mathcal{F}'_{T, p}(1/p).
\end{equation}
Secondly, by using the argument used in the proof of Theorem 4.3 in \cite{Wed2} and 
the Minkowski-Siegel formula (cf. Theorem 4.2 of loc.cit.), we have
\begin{equation}\label{fce3}
\displaystyle \sum_{(E,E'):{\rm (ss)}}
\frac{R_{{\rm Hom}(E,E')}(T)}{\sharp {\rm Aut}(E) \sharp {\rm Aut}(E')}=
\frac{1}{3^2\cdot 2^4}\cdot \Big(\frac{p-1}{p}\Big)^2\cdot \frac{\pi^{\frac{7}{2}}}{\Gamma(2)\Gamma(3/2)}\cdot 
{\rm det}(T)^{1/2}\cdot \prod_{l<\infty} \alpha_l(T, O_D).
\end{equation}

Here $O_D$ is a maximal order in $D$.
Then $O_D\otimes_{\mathbb{Z}}\mathbb{Z}_p$ is the maximal order in the quaternion division algebra over $\mathbb{Q}_p$ and $O_D\otimes_{\mathbb{Z}}\mathbb{Z}_l$ with $l\neq p$ is isomorphic to $H_2$, the hyperbolic space of rank $4$.

Thus we have
\begin{equation}\label{fce4}
\prod_{l<\infty} \alpha_l(T, O_D)=\alpha_{T, p}(T, O_D)\cdot \prod_{l<\infty, l\neq p} \mathcal{F}_{T, l}(1/l^2). 
\end{equation}

We now plug in Equations (\ref{fce2})-(\ref{fce4}) into Equation (\ref{fce1}) then we have the following:

\begin{multline}\label{eq8.6}
\sum_{\substack{y=((E,E'),f_1,f_2) \\  (E,E')_{/\bar{\F}_p}:{\rm (ss)}\\ {\rm deg}(f_i)=m_i}}\frac{{\rm IM}_{p,y}}{\sharp {\rm Aut}(E) \sharp {\rm Aut}(E')}=\\
\sum_{T \in {\rm Sym}_2(\Z)_{>0}  
\atop {\rm diag}(T)=(m_1,m_2), \chi_T(p)=-1,0} 
\Big(-\frac{1}{3^2\cdot 2^4}\cdot \frac{p-1}{p^2}\cdot \frac{\pi^{\frac{7}{2}}}{\Gamma(2)\Gamma(3/2)}\cdot 
{\rm det}(T)^{1/2}\cdot \alpha_{T, p}(T, O_D)\cdot \mathcal{F}'_{T, p}(1/p)\cdot \prod_{l<\infty, l\neq p} \mathcal{F}_{T, l}(1/l^2)\Big).
\end{multline}

On the other hand, by the functional equation of the Siegel series (cf. Theorem 4.1 in \cite{I1}) combined with Theorem 0.1 of \cite{IK1} for anisotropic binary quadratic lattice, 
we have 
\begin{equation}\label{eq8.7}
\mathcal{F}'_{T, p}(1/p)=\left\{
  \begin{array}{l l}
-\frac{p^{3+|\mathrm{GK}(T_p)|/2}}{(p-1)(p+1)} \mathcal{F}_{T, p}(\frac{1}{p^2}) & \quad  \textit{if  $|\mathrm{GK}(T_p)|$ is even};\\
-\frac{2p^{4+(|\mathrm{GK}(T_p)|-1)/2}}{(p-1)(p+1)^2} \mathcal{F}_{T, p}(\frac{1}{p^2}) & \quad  
\textit{if  $|\mathrm{GK}(T_p)|$ is odd}.
    \end{array} \right.
\end{equation}
Here, $T_p=T\otimes \mathbb{Z}_p$
and $|\mathrm{GK}(T_p)|=a_1+a_2$ if $\mathrm{GK}(T_p)=(a_1, a_2)$.

From now on, we make our assumption of $p>2$, in order to use Lemma \ref{ldrk2aniso} to compute $\alpha_p(L, O_D)$.
Then Equation (\ref{eq8.6}) turns to be (for $p>2$)
\begin{eqnarray}\label{eq8.8}
\sum_{\substack{y=((E,E'),f_1,f_2) \\  (E,E')_{/\bar{\F}_p}:{\rm (ss)}\\ {\rm deg}(f_i)=m_i}}\frac{{\rm IM}_{p,y}}{\sharp {\rm Aut}(E) \sharp {\rm Aut}(E')}&=&
\frac{1}{3^2 2^3}\cdot\frac{\pi^{\frac{7}{2}}}{\Gamma(2)\Gamma(3/2)} \cdot  \sum_{T \in {\rm Sym}_2(\Z)_{>0}  
\atop {\rm diag}(T)=(m_1,m_2), \chi_T(p)=-1,0} 
{\rm det}(T)^{1/2}  \prod_{l<\infty} \mathcal{F}_{T, l}(1/l^2) \nonumber \\
&=&
\frac{\pi^3}{3^2 2^2} \cdot  \sum_{T \in {\rm Sym}_2(\Z)_{>0}  
\atop {\rm diag}(T)=(m_1,m_2), \chi_T(p)=-1,0} 
{\rm det}(T)^{1/2}  \prod_{l<\infty} \mathcal{F}_{T, l}(1/l^2),    
\end{eqnarray}
since $\Gamma(3/2)=\ds\frac{\sqrt{\pi}}{2}$. 
Notice that the product $\ds\prod_{l<\infty}\mathcal{F}_{T, l}(1/l^2)$ coincides with $\ds\lim_{s\to 0}\alpha^{(2)}(2s+2,T)$ 
defined in Section 1.3 of \cite{nagaoka} (note that the limit $\ds\lim_{s\to 0}$ will be taken after 
the analytic continuation given by Kauffold's theorem, cf. Theorem 1.3.1 of \cite{nagaoka}). 
The formula (2.5.6), p.84 of \cite{nagaoka} says 
$$\ds\lim_{s\to 0}\alpha^{(2)}(2s+2,T)=36\pi^{-3}{\rm det}(T)^{-1/2}\frac{1}{288}c(T)$$
where $\ds\frac{1}{288}c(T)=\sum_{d|{\rm cont}(T)}dH(\frac{\det(T)}{d^2})$ in the notation of \cite{nagaoka}. 

In conclusion, we have
\begin{equation}\label{eq8.9}
\sum_{\substack{y=((E,E'),f_1,f_2) \\  (E,E')_{/\bar{\F}_p}:{\rm (ss)}\\ {\rm deg}(f_i)=m_i}}\frac{{\rm IM}_{p,y}}{\sharp {\rm Aut}(E) \sharp {\rm Aut}(E')}=
\frac{1}{3^22^5}\cdot   \sum_{T \in {\rm Sym}_2(\Z)_{>0}  
\atop {\rm diag}(T)=(m_1,m_2), \chi_T(p)=-1,0} 
c(T).
\end{equation}

Here, $c(T)$ is the Fourier coefficient of the Siegel-Eisenstein series for $\mathrm{Sp}_4$ of weight $2$, with respect to the half-integral symmetric matrix $T$.
\end{proof}

\begin{proof}[The proof of Theorem \ref{newid}]
The identity of Theorem \ref{newid} now follows by combining Theorems \ref{ordpart} and  \ref{sspart}. 
\end{proof}

\begin{Lem}\label{ldrk2aniso}
Let $p>2$.
Let $(L, q_L)$ be an anisotropic quadratic $\mathbb{Z}_p$-lattice of rank $2$.
Assume that $(L\otimes_{\mathbb{Z}_p}\mathbb{Q}_p, q_L\otimes_{\mathbb{Z}_p}\mathbb{Q}_p)$ is nondegenerate.
Let $(O_D, q_D)$ be the quadratic lattice, where $O_D$ is the maximal order of the quaternion division algebra $D$ over $\mathbb{Q}_p$ with $q_D$  the reduced norm on $D$.
Then the local density $\alpha_p(L, O_D)$ is given as follows:
\[
\alpha_p(L, O_D)=\left\{
  \begin{array}{l l}
 p^{\frac{-d}{2}}\cdot 2(1+\frac{1}{p})   & \quad  \textit{if $d$ is even};\\
 p^{\frac{-(d-1)}{2}}\cdot(1+\frac{1}{p})^2   & \quad  \textit{if $d=2d'+1$ is odd}.
    \end{array} \right.
\]
Here, $d=|\mathrm{GK}(L)|$.


\end{Lem}

\begin{proof}
Let $L'$ be a sublattice of $L$ of rank $2$.
Then By Theorem 5.6.4.(d) of \cite{Kit}, we have, for any prime $p$ including $p=2$,
\[
\alpha_p(L', O_D)=p^{-[L:L']}\cdot \alpha_p(L, O_D).
\]
Indeed Theorem 5.6.4.(d) of \cite{Kit} says inequality but in our case of anisotropic lattice, the inequality turns to be the equality.

Thus we may assume that $L$ is a maximal lattice so that $\mathrm{GK}(L)=(a_1, a_2)$ has only three possibilities: $(a_1, a_2)=(0,0), (0,1), (1,1)$.
We will use the local density formula of Yang in  \cite{Y0} ($p>2$).

From now on we assume that $p$ is odd.
Then $L$ is diagonalizable so that the exponential order of each diagonal entry is $a_i$. 
Based on Theorem 7.1 of \cite{Y0}, we have
\[
\alpha_p(L, O_D)=\left\{
  \begin{array}{l l}
  2(1+\frac{1}{p})  & \quad  \textit{if $(a_1, a_2)=(0, 0)$};\\
  \frac{2}{p}\Big(1+\frac{1}{p}\Big)  & \quad  \textit{if $(a_1, a_2)=(1, 1)$};\\
  \Big(1+\frac{1}{p}\Big)^2  & \quad  \textit{if $(a_1, a_2)=(0, 1)$}.
    \end{array} \right.
\]
This completes the proof.
\end{proof}

\begin{Rmk}\label{conrad} 
\begin{enumerate}
 \item  B. Conrad  informed us that Theorem \ref{newid} is true 
when $(m_1, m_2)=1$ with any prime $p$
(cf.  \cite{conrad}), by proving that  the scheme ${\rm Spec}\Z[x,y]/(\varphi_{m_1},\varphi_{m_2})$ is flat.

On the other hand, Theorem \ref{newid} is true when $m_1, m_2 \leq 9$ with any prime $p$ by numerical calculation given in Appendix \ref{App:Appendix}.

\item We note that Theorem \ref{newid} is a combination of Theorems \ref{ordpart} and \ref{sspart}, and the identity of Theorem \ref{ordpart} holds for $p=2$.
In the proof of Theorem  \ref{sspart}, the only place we make the assumption $p>2$ is the usage of Lemma \ref{ldrk2aniso}.

On the other hand, one can also compute $\alpha_p(L, O_D)$  of Lemma \ref{ldrk2aniso} with $p=2$ by using the local density formula given in \cite{Y}, which is more complicated than that of \cite{Y0} ($p>2$).
Consequently the explicit computation of $\alpha_p(L, O_D)$ when $p=2$  would directly yield  the identity of Theorem \ref{newid}.



\end{enumerate}
\end{Rmk}

\section{Application 2: Local intersection multiplicities on the special fiber}\label{sectionin} In this section we recall the special cycles on the Shimura variety 
for $\mathrm{GSpin}(n,2),\ 0\le n\le 3$ defined by Kudla and Rapoport with a collaborator Yang. 
We refer the articles \cite{KRY1},\cite{KRY2},\cite{KR2},\cite{KR1} for $n=0,1,2,3$ respectively 
and the readers are supposed to be familiar with these references. 
Let $(n_1,\ldots,n_r)$ be a partition of $n+1$, hence $n_i\ge 1,\ n_1+\cdots+n_r=n+1$. 
We always consider $r=1$ when $n \le 1$. 

Let $V$ be a quadratic form over $\Z$ with the signature 
$(n,2)$ over $\R$ considered in each paper. 
Let $G={\rm GSpin}(V)$ be the generalized spinor group associated to $V$. Let $p$ be an odd prime so that $G$ is smooth over $\Z_p$.  
Then for any neat open compact subgroup $K^p\subset G(\A^p_f)$ and a hyperspecial open compact subgroup $K_p$ of $G(\Q_p)$ 
defined by a suitable structure on $G(\Q_p)$ in each paper (cf. p. 704, line 9 of \cite{KR1} for $n=3$), $d_i\in {\rm Sym}_{n_i}(\Q)$ and open compact subgroups 
$\omega_i\subset V(\A^p_f)^{n_i}$ which are invariant under $K^p$, one can associate 
the special cycles $\mathcal{Z}(d_i,\omega_i)$ and consider the intersection of them: 
$$\mathcal{Z}=\mathcal{Z}(d_1,\omega_1)\times_{\mathcal{M}}\cdots \times_{\mathcal{M}} \mathcal{Z}(d_r,\omega_r)=
\coprod_{T\in {\rm Sym}_{n+1}(\Z_{(p)})_{\ge 0} \atop {\rm diag}(T)=(d_1,\ldots,d_r)}\mathcal{Z}(T,\omega) $$
where $\mathcal{M}$ is the integral model over $\Z_{(p)}$ for the Shimura variety associated to $(G,K^pK_p)$ and 
$\omega=\{\omega_i\}^r_{i=1}$ (cf. Section 2,3 of \cite{KR1} for $n=3$). 
Since $n\le 3$, the Shimura variety $\mathcal{M}$ is of PEL type, namely a moduli space of abelian varieties with endomorphism structure 
by $\mathcal{O}$ which is a maximal order of  $M_2(B_\Q),\ B_F,\ B_\Q,$ or $K$ for $n=3,2,1,0$ respectively. 
Here $B_\Q$ (resp. $B_\F$) is a quaternion algebra over $\Q$ (resp. over a real quadratic field $F$) and $K$ is an imaginary quadratic field.     
Any geometric point $\xi={\rm Spec}\hspace{0.5mm}k$ on $\mathcal{Z}$ in characteristic $p$ consists of quintuple $(A,\iota,\lambda,\overline{\eta}^p,\textbf{j}')$ satisfying the 
following conditions; 
\begin{enumerate}
\item $A$ is an abelian variety of dimension $2^n$ over $k$ considered up to prime to $p$ isogeny;
\item $\iota:\mathcal{O}\otimes \Z_{(p)}\hookrightarrow {\rm End}_k(A)\otimes \Z_{(p)}$ is a ring homomorphism such that 
$$\det (\iota(c):{\rm Lie}(A))=N^0(c)^2$$
for any $c\in \mathcal{O}$ where $N^0$ is the reduced norm of $\mathcal{O}$. 
\item $\lambda$ is a $\Z^\times_{(p)}$-class of a prime to $p$ isogeny $A\lra A^\ast$ 
such that $n'\lambda$ comes from an ample line bundle on $A$ for some $n'\in\Z$. Here $A^\ast$ is the dual 
abelian variety of $A$; 
\item $\textbf{j}'=(\textbf{j}_1,\ldots,\textbf{j}_r)$ and $\textbf{j}_i\in {\rm End}_k(A)^{n_i}$ is a vector of 
special endomorphisms for $i=1,\ldots,r$ such that $q(\textbf{j}_i)=d_i$,
where the quadratic form $q$ is 
defined by  the Rosati-involution $\star$ with $q(x){\rm id}_A=x\circ x^\star$ for any $x\in {\rm End}_k(A)$ (cf. Lemma 2.2 
of \cite{KR1} for $n=3$).
Here a special endomorphism is defined to be an endomorphism $f$ on $A$ which satisfies $f^\star=f$ and ${\rm tr}^0(f)=0$ 
where $\star$ stands for the Rosati-involution with respect to $\lambda$ and ${\rm tr}^0$ means the reduced trace of 
${\rm End}_k((A,\iota))^{{\rm op}}$ (note that special endomorphisms can be regarded as elements in ${\rm End}_k((A,\iota))^{{\rm op}}$ 
(cf. (2.13) of \cite{KR1} for $n=3$)). 
\item  $\overline{\eta}^p=\{\eta^p k\ |\ k\in K^p\}$ is 
a $K^p$-class of a $\mathcal{O}$-linear isomorphism $\eta^p:V^p(A):=\prod_{\ell\not=p}T_\ell(A)\otimes\Q_\ell \stackrel{\sim}{\lra}\mathcal{O}\otimes \A^p_f$. 
Here the action of $K^p$ on $\mathcal{O}$ is defined by the Clifford structure of ${\rm GSpin}(n,2)$ in each case. 
It is known that ${\rm End}_{\mathcal{O}}(\mathcal{O}\otimes_\Z\A_f)$ contains $V(\A_f)$.
We require  
that $(\eta^p)^\ast(\textbf{j}_i):=\eta^p\circ \textbf{j}_i|_{V^p(A)} \circ  (\eta^p)^{-1}$ belongs to $\omega_i$.  
\end{enumerate}

By Theorem 0.1 in \cite{KR1} for $n=3$, Theorem 6.1 in \cite{KR2} for $n=2$, Theorem 3.6.1 in \cite{KRY2} for $n=1$, and 
Proposition 5.9 in \cite{KRY1} for $n=0$ we know a criterion for $T$ which yields  that any geometric point of $\mathcal{Z}(T):=\mathcal{Z}(T,\omega)$ is isolated. From now on we assume this condition. 
For any geometric point $\xi$ on $\mathcal{Z}(T)$ we define the local intersection multiplicity of $\mathcal{Z}(T)$ at $\xi$ by 
$$e(\mathcal{Z}(T),\xi):={\rm length}_{\Z_{(p)}}\mathcal{O}_{\mathcal{Z}(T),\xi}$$
where $\mathcal{O}_{\mathcal{Z}(T),\xi}$ is the localization of $\mathcal{O}_{\mathcal{Z}(T)}$ at $\xi$. By the assumption 
$e(\mathcal{Z}(T),\xi)$ is finite. 

 To compute $e(\mathcal{Z}(T),\xi)$, we need to consider the formal completion of it along $\xi$ to apply 
the deformation theory. 
Put $\textbf{j}'=(f_1,\ldots,f_{n+1})$ for simplicity. 
By the assumption of each reference as above, we  see that $A$ is isomorphic to a product of supersingular elliptic curves $E$ over $\bF_p$. 
Let $G=\widehat{A}$ be the formal group associated to $A$ and we denote by $\widehat{f}_i$ the corresponding special endomorphism 
on $G$ for each special endomorphism $f_i$ via a natural algebra homomorphism 
\begin{equation}\label{endo}
{\rm End}_k((A,\iota))^{{\rm op}}\otimes_\Z\Z_p\hookrightarrow 
{\rm End}((G,F))\subset \mathcal{O}\otimes_\Z \Q_p
\end{equation}
where $F$ is the Frobenius endomorphism and ${\rm End}((G,F))$ stands for the set of 
endomorphisms on $G$ commuting $F$ (cf. (5.13), p. 726 of \cite{KR1} for $n=3$). 
 The $\Z_p$ submodule $L'$ spanned by $\{\widehat{f}_i\}$ in ${\rm End}(G)$ endows with 
the structure as a quadratic space $L'=(L,q_{L'})$ which comes from the Clifford structure. For instance 
$xy+yx=(x,y)_{L'}$ for any $x,y\in L'$.   

Since $p$ is odd, there exists a basis of 
$L'$ such that $q_{L'}$ is isometric  to $T'=\diag(u_1 p^{a_1},\ldots,u_{n+1} p^{a_{n+1}} )$ over $\Z_p$ with integers 
$a_1\le \cdots \le a_{n+1}$ and with units $u_i,\ 1\le i \le n+1$ in $\Z_p$. 
Then the Gross-Keating invariant of $T'$ is given simply by ${\rm GK}(L')={\rm GK}(T')=(a_1,\ldots, a_{n+1})\in \Z^{n+1}_{\ge 0}$. 
Accordingly we can take an optimal basis $\varphi_1,\ldots,\varphi_{n+1}$ in ${\rm End}(G)$ such that 
$\varphi^2_i=q_{L'}(\varphi_i)=u_i p^{a_i}\ (1\le i\le n+1)$. 

Let $R=W(\bF_p)[[t_1,\ldots,t_{n}]]$ be the universal deformation ring of $G$ on ${\rm CLN}_{W(\bF_p)}$ which is isomorphic to 
the strict completion of $\mathcal{O}_{\mathcal{M}}$ at $\xi$. This follows from Serre-Tate theorem (Theorem 1.2.1 of \cite{Katz}), Theorem 2.3.3, p. 242 of \cite{O}, and the fact that 
the Shimura variety $\mathcal{M}$ is a fine moduli space by the assumption on the compact open subgroup $K^p$.  
Let  $\mathcal{G}$ be its universal family. 
Here we make the convention that $R=W(\bF_p)$ when $n=0$. 

We denote by $I=I(\varphi_1,\ldots,\varphi_{n+1})$ 
the minimal ideal of $R$ such that all $\varphi_i$'s are liftable to special endomorphisms on $\mathcal{G}$ modulo $I$. 
By the theorem of Serre and Tate it is easy to see that 
$$e(\mathcal{Z}(T'),\xi)={\rm length}_{W(\bF_p)}R/I$$
(cf. (6.1) of \cite{KR1}).
By Section 6 of \cite{KR1} for $n=3$, the proof of Theorem 6.1 of \cite{KR2} for $n=2$, Theorem 3.6.1 of \cite{KRY2} for $n=1$ , and 
Theorem 5.11 of \cite{KRY1} for $n=0$, it turns out that $e(\mathcal{Z}(T'),\xi)$ depends 
only on ${\rm GK}(T')=(a_1,\ldots,a_{n+1})$.  Hence we may write it for 
\begin{equation}\label{lim}
e(a_1,\ldots, a_{n+1}):=e(\mathcal{Z}(T'),\xi)={\rm length}_{W(\bF_p)}R/I(\varphi_1,\ldots,\varphi_{n+1}).
\end{equation}
This will be checked in Theorem \ref{special2} below. 
When $a_{n+1}\ge 2$, we see that there exist $\varphi'_{n+1}\in {\rm End}(G)$ such that $\varphi_{n+1}=p\varphi'_{n+1}$ and  that
$L'=\langle \varphi_1,\ldots,\varphi_n,\varphi'_{n+1} \rangle$ makes a sublattice of $L$ with ${\rm GK}(L')=(a_1,\ldots,a_{n})\cup (a_{n+1}-2)$ 
where 
$(a_1,\ldots,a_{n})\cup (a_{n+1}-2)$ means the re-ordering of $\{a_1,\ldots,a_{n},a_{n+1}-2\}$ to be 
the non-decreasing sequence.  
Therefore we have 
$${\rm length}_{W(\bF_p)}R/I(\varphi_1,\ldots,\varphi'_{n+1})=e((a_1,\ldots,a_{n})\cup (a_{n+1}-2)).$$
Our interest is to understand the difference between 
$${\rm length}_{W(\bF_p)}R/I(\varphi_1,\ldots,\varphi_{n+1})\  {\rm and}\  
{\rm length}_{W(\bF_p)}R/I(\varphi_1,\ldots,\varphi'_{n+1})$$ in terms of the 
local intersection multiplicity of special cycles over a finite field.  
This motivates us to consider the following situation in special cycles in the special fiber $\mathcal{M}_{\F_p}$. 
Assume that $1\le n\le 3$.  
Let $(n_1,\ldots,n_r)$ be a partition of $n$, hence $n_i\ge 1,\ n_1+\cdots+n_r=n$. 
We always consider $r=1$ when $n=1$. 
For $d_i\in {\rm Sym}_{n_i}(\Z_{(p)})$ ($1\le i\le r$) 
let us consider the closed subscheme in $\mathcal{M}_{\F_p}$ given by  
$$\mathcal{Z}(d_1,\omega_1)_{\F_p}\times_{\mathcal{M}_{\F_p}}\cdots\times_{\mathcal{M}_{\F_p}}  \mathcal{Z}(d_r,\omega_r)_{\F_p}
=\coprod_{T\in {\rm Sym}_{n}(\Z_{(p)})_{\ge 0} \atop {\rm diag}(T)=(d_1,\ldots d_r)}
\mathcal{Z}(T,\omega)_{\F_p}$$
for $1\le n\le 3$. Any geometric point on $\mathcal{Z}(T,\omega)_{\F_p}$ is similarly a quintuple $(A,\iota,\lambda,\overline{\eta}^p,\textbf{j})$ as before but in this case  
we replace $n+1$ endomorphisms with $n$ endomorphisms $\textbf{j}$ which is related to the fourth condition $(4)$.  
In this section, when $n\le 3$ and 
the underlying space $A=A_x$ of a geometric point $x$ is superspecial, we will study that the multiplicity   
$$e(\mathcal{Z}(T,\omega)_{\F_p},x):={\rm length}_{\bF_p}\mathcal{O}_{\mathcal{Z}(T,\omega)_{\F_p},x}$$
is finite and that it depends only on ${\rm GK}(T\otimes \Z_p)$ under some conditions. Let us confirm this as follows. 

Recall that $G$ is the formal group of $A$. 
Let $R_p=\bF_p[[t_1,\ldots,t_{n}]]$ be the universal deformation of $G$ on ${\rm CLN}_{\bF_p}$ which is 
isomorphic to the completion of $\mathcal{O}_{\mathcal{M}_{\bF_p}}$ at $x$. We put $R_p=\bF_p$ when $n=0$. Clearly $\mathcal{G}_p:=\mathcal{G}\otimes \bF_p$ is the local deformation of $G$ over 
$R_p$. Let $\varphi_1,\ldots,\varphi_n$ be an optimal basis of the lattice consisting of $n$ special endomorphisms 
$\textbf{j}$ in the data of $x$ via (\ref{endo}). 

We define the minimal ideal $I_p=I_p(\varphi_1,\ldots,\varphi_{n})$ of $R_p$ such that all $\varphi_i$'s ($1\le i\le n$) are liftable to special endomorphisms on $\mathcal{G}_p$ modulo $I_p$. 
As before we put $I_p=\{0\}$ when $n=0$. By definition we see that 
$$e(\mathcal{Z}(T,\omega)_{\F_p},x)={\rm length}_{\bF_p}R_p/I_p.$$
\begin{Prop}\label{specialcycles}Let $T$ be an element in ${\rm Sym}_{n}(\Z_{(p)})$. For any geometric point $x=(A_x,\iota,\lambda, \overline{\eta}^p,\textbf{j})$ on 
$\mathcal{Z}(T,\omega)_{\F_p}$ let $L=(L,q_L)$ be 
the quadratic space over $\Z_p$ corresponding to the special endomorphisms $\textbf{j}$ with ${\rm GK}(L)={\rm GK}(T\otimes \Z_p)=(a_0,\ldots,a_{n-1})$.  
Assume that $A_x$ is superspecial and $q_L$ represents 1. Then it holds that 
$$e(\mathcal{Z}(T,\omega)_{\F_p},x)=\T_{b_1,b_2}$$
where 
$$(b_1,b_2)=
\left\{
\begin{array}{ll}
(a_1,a_2) &\ {\rm if}\ n=3 \\
(0,a_{n-1}) &\ {\rm if}\ n=1,2  \\
(0,0) & \ {\rm if}\ n=0
\end{array}
\right..
$$  
\end{Prop}

\begin{proof}Let us first consider the case of $n=3$. 
We follow the argument in page 
733 of \cite{KR1}. By our assumption $q_L$ represents 1 over $\Z_p$. This implies $a_0=0$. 
In this case the formal group $G$ of $A_x$ decomposes into $\widehat{A}^4_0$ where $\widehat{A}_0$ is a formal group of dimension 2 and height 4 
equipped with a principal quasi polarization $\lambda_{\widehat{A}_0}$ (cf. Section 4 of \cite{KR1}). 
As explained right after (6.3) of \cite{KR1} there exists $x_0\in L$ such that $q_{L}(x_0)=1$. 
Then the idempotents $e_0=\frac{1}{2}(1+x_0),\ e_1=\frac{1}{2}(1-x_0)$ induce the further decomposition of $\widehat{A}_0$ as 
$$\widehat{A}_0\simeq e_0\widehat{A}_0\times e_1 \widehat{A}_0\simeq G^2_0$$
where $G_0$ is a formal group of dimension 1 and height 2.  Put $M_0=\langle x_0 \rangle^\perp$ in $L$ which is of rank 2 and 
${\rm GK}(M_0)=(a_1,a_2)$. 
By the Clifford structure for the special endomorphisms we see that $xx_0+x_0x=\langle x,x_0 \rangle=0$ for any $x\in M_0$. 
It follows from this that $xe_0=e_1x$. Therefore $M_0$ can be regarded as a sublattice in 
${\rm Hom}( e_0\widehat{A}_0,e_1 \widehat{A}_0)\simeq {\rm End}(G_0)$ which is the maximal order of a unique 
quaternion division algebra over $\Q_p$.  
Since any principal quasi-polarization deforms automatically the deformation problem on ${\rm CLN}_{\bF_p}$ of 
$x$ is same as one of $(G_0,M_0)$. 
This shows that  
\begin{equation}\label{n3}
e(\mathcal{Z}(T,\omega)_{\F_p},x)=\T_{a_1,a_2}. 
\end{equation}
Hence we have the claim. 

The remaining cases will be done similarly in which the arguments of 
Theorem 6.1 of \cite{KR2}, Theorem 3.6.1 of \cite{KRY2}, and  Proposition 5.9 of \cite{KRY2} for $n=2,1$ and $n=0$ should be 
consulted respectively.   
\end{proof}
\begin{Cor}\label{specialfiber} Keep the notation and the assumptions in Proposition \ref{specialcycles}. 
Then the geometric point $x$ is isolated in $\mathcal{Z}(T)_{\bF_p}$. 
\end{Cor}

\begin{Prop}\label{special2}Let $T'$ be an element in ${\rm Sym}_{n+1}(\Z_{(p)})$. For any isolated geometric point $\xi=(A_\xi,\iota,\lambda, \overline{\eta}^p,\textbf{j}')$ of $\mathcal{Z}(T')$ let $L'$ be 
the $\Z_p$-lattice corresponding to the special endomorphisms $\textbf{j}'$ with ${\rm GK}(L')={\rm GK}(T'\otimes \Z_p)=(a_0,\ldots,a_{n})$.  
Then for any $T\in {\rm Sym}_{n}(\Z_{(p)})$ with 
${\rm GK}(T\otimes\Z_p)=(a_0,\ldots,a_{n-1})$ such that $T\otimes\Z_p$ comes from a sublattice of $L'$ and 
a geometric point $x$ of $\mathcal{Z}(T,\omega)_{\F_p}$ whose underlying abelian variety $A_x$ is $A_\xi$, it holds that  
$$e(\mathcal{Z}(T,\omega)_{\F_p},x)=e(a_0,\ldots,a_{n})-e(a_0,\ldots,a_{n}-2).$$
\end{Prop}
\begin{proof}
When $n=3$, by Theorem 0.1 or Corollary 5.15 of \cite{KR1}, we see that $T'\otimes \Z_p$ represents 1 over $\Z_p$. This implies $a_0=0$. 
Let $L'$ be the lattice in $\Q_p$ corresponding to $T'$ over $\Z_p$. Then it turns out that 
its Gross-Keating invariant becomes ${\rm GK}(L')=(0,a_1,a_2,a_3)$ with nondecreasing integers $0\le a_1\le a_2\le a_3$ which 
satisfy the condition that the parities of the three integers never be same. 
Then the argument in p. 733 of \cite{KR1} tells us that  
\begin{equation}\label{n3}
e(0,a_1,a_2,a_3)=\alpha_p(a_1,a_2,a_3)
\end{equation}
where $\alpha_p$ is the intersection number in Proposition 5.4 of \cite{GK1} as mentioned before. 
By Lemma 5.6 of \cite{GK1}  and Proposition \ref{specialcycles}
$$e(0,a_1,a_2,a_3)-e(0,a_1,a_2,a_3-2)=\mathcal{T}_{a_1,a_2}=e(\mathcal{Z}(T,\omega)_{\F_p},x).$$

When $n=2$, by Theorem 6.1 of \cite{KR2}, we see that $T'\otimes \Z_p$ is isometric to ${\rm diag}(1, u_1 p^{a_1},u_2 p^{a_2})$ over $\Z_p$ with units $u_1,u_2$ in $\Z_p$. 
Let $L'$ be the lattice in $\Q_p$ corresponding to $T'$ over $\Z_p$. Then it turns out that 
its Gross-Keating invariant becomes ${\rm GK}(L')=(0,a_1,a_2)$ with nondecreasing integers $0\le a_1\le a_2$ which 
satisfy the condition that the parities of $0,a_2,a_3$ never be same.  
Then the argument in p. 195 loc.cit. shows that  
\begin{equation}\label{n2}
e(0,a_1,a_2)=\alpha_p(0,a_1,a_2)
\end{equation}
and similarly we have $e(0,a_1,a_2)-e(0,a_1,a_2-2)=\mathcal{T}_{0,a_1}=e(\mathcal{Z}(T,\omega)_{\F_p},x)$.  

When $n=1$, by Theorem 3.6.1 of \cite{KRY2}, we see that $T'\otimes \Z_p$ is isometric to 
${\rm diag}(u_0 p^{a_0},u_1 p^{a_1})$ over $\Z_p$ with units $u_0,u_1$ in $\Z_p$. 
Let $L$ be the lattice of rank three in $\Q_p$ corresponding to ${\rm diag}(1,T')$ over $\Z_p$. Then it turns out that 
its Gross-Keating invariant becomes ${\rm GK}(L)=(0,a_0,a_1)$ with nondecreasing integers $0\le a_0\le a_1$ which 
satisfy the condition that the parities of $0,a_0,a_1$ never be same.  
Then Theorem 3.6.1 of loc.cit. shows that  
\begin{equation}\label{n1}
e(a_0,a_1)=\alpha_p(0,a_0,a_1)
\end{equation}
and similarly we have $e(a_0,a_1)-e(a_0,a_1-2)=\mathcal{T}_{0,a_0}=e(\mathcal{Z}(T,\omega)_{\F_p},x)$.

Finally we consider the case when $n=0$.  
By Proposition 5.9 of \cite{KRY2}, we see that $T'$ (it is denoted by $t$ in the reference) satisfies that 
$a_0:={\rm ord}_p(t)\equiv 1$ mod 2. 

Let $L$ be the lattice of rank three in $\Q_p$ corresponding to ${\rm diag}(1,1,T')$ over $\Z_p$. Then it turns out that 
its Gross-Keating invariant becomes ${\rm GK}(L)=(0,0,a_3)$ and it satisfies that $a_0$ is not even.  
Then Theorem 5.11 of loc.cit. shows that  
\begin{equation}\label{n0}
e(a_0)=\alpha_p(0,0,a_0)
\end{equation}
Then $e(a_0)-e(a_0-2)=1=e(\mathcal{Z}(T,\omega)_{\F_p},x)$ by Lemma 5.6 of \cite{GK1} and the claim follows with the convention made when $n=0$.
\end{proof}

Assume that $T\in {\rm Sym}_{n}(\Z_{(p)})$ satisfies the condition in Proposition \ref{special2}. For $n$ and Gross-Keating invariant for $T$ in  Equations (\ref{n3})-(\ref{n0}) 
we take an anisotropic lattice $M$ of rank 2 with 
$${\rm GK}(M)=
\left\{
\begin{array}{ll}
(a_1,a_2) &\ {\rm if}\ n=3 \\
(0,a_{n-1}) &\ {\rm if}\ n=1,2  \\
(0,0) & \ {\rm if}\ n=0
\end{array}
\right.
$$  
Note that in our situation $T\otimes \Z_p$ is always anisotropic and hence we can apply the results in Section \ref{reGK} to $T\otimes \Z_p$. 
Then plugging Proposition \ref{special2} with Section \ref{reGK} (cf. Theorem \ref{thmaniso1}) we have 
\begin{Thm}\label{main-ind}
Keep the assumption in Proposition \ref{special2}. 
Then 
$$e(\mathcal{Z}(T)_{\bF_p},x)=\frac{-1}{p-1}\cdot \mathcal{F}'_M(\frac{1}{p}).   
$$

\end{Thm}

\appendix
\section{The table of intersection numbers} \label{App:Appendix} 
For each positive integer $m$ we denote by $\psi_{m/n^2}$  the polynomial in $\Z[x,y]$ which appears as a factor of $\varphi_m=\ds\prod_{n^2|m}\psi_{m/n^2}$ (cf. p.2 of \cite{V1}). 
 In this appendix we give a table for  
 $${\rm dim}_\C\C[x,y]/(\psi_{m_1},\psi_{m_2}),\ {\rm dim}_{\F_p}\F_p[x,y]/(\psi_{m_1},\psi_{m_2})$$
  for $2\le m_1\le m_2 \le 9$ such that $m_1m_2$ is not a square and $2\le p<50$. 
 S. Yokoyama (cf. \cite{yok}) kindly computed both quantities and checked they coincide directly. 
From this computation with Corollary \ref{reg} and Theorem 2.1 of \cite{V1}  it is easy to see that
$$(T_{m_1,p},T_{m_2,p})=\sum_{n^2_1|m_1,n^2_2|m_2}{\rm dim}_{\F_p}\F_p[x,y]/(\psi_{\frac{m_1}{n^2_1}},\psi_{\frac{m_2}{n^2_2}})$$
$$=
\sum_{n^2_1|m_1,n^2_2|m_2}{\rm dim}_{\C}\C[x,y]/(\psi_{\frac{m_1}{n^2_1}},\psi_{\frac{m_2}{n^2_2}})=(T_{m_1,\C},T_{m_2,\C})$$ 
for $m_1,m_2,p$ above  including the case where $p=2$. 
Put $d(m_1,m_2):={\rm dim}_\C\C[x,y]/(\psi_{m_1},\psi_{m_2})={\rm dim}_{\F_p}\F_p[x,y]/(\psi_{m_1},\psi_{m_2})$. 
We list up all of them as below:

\begin{table}[htbp]
\label{tab1}
\begin{center}
\begin{tabular}{|c|c|c|c|c|c|c|c|c|c|c|c|c|}
\hline
$(m_1,m_2)$ & $(2,3)$ &  $(2,4)$ & $(2,5)$ & $(2,6)$&
$(2,7)$ & $(2,9)$  & $(3,4)$&
$(3,5)$ & $(3,6)$   &  $(3,7)$ & $(3,8)$ & $(3,9)$ \\
\hline
$d(m_1,m_2)$ & 18 & 28 &  30 & 56 & 42 & 62 & 38 & 40 & 78 & 56 & 82 & 84 \\
\hline 
\hline 
$(m_1,m_2)$ & $(4,5)$ &  $(4,6)$ & $(4,7)$ & $(4,8)$&
$(5,6)$ & $(5,7)$  & $(5,8)$&
$(5,9)$ & $(6,7)$   &  $(6,8)$ & $(6,9)$ & $(7,8)$ \\
\hline
${\rm IN}(m_1,m_2)$ & 60 & 118 &  84 & 124 & 122 & 84 & 126 & 128 & 168 & 248 & 252 & 170 \\
\hline 
\hline
$(m_1,m_2)$ & $(7,9)$ &  $(8,9)$ & & &&&&&&&& \\
\hline 
${\rm IN}(m_1,m_2)$ & 172 & 256  & & &&&&&&&&  \\
\hline

\end{tabular}
\end{center}
\caption{}
\end{table}
Let us remark that $d(1,m)=d(m,1)=\ds\sum_{d|m}\max\{d,\frac{m}{d}\}$. By using this we see, for example, that 
$(T_{2,p},T_{4,p})=(T_{2,\C},T_{4,\C})=d(2,1)+d(2,4)=4+28=32$ for any $p<50$.


\begin{thebibliography}{99}

\bibitem[ARGOS07]{A} \textit{Argos Seminar on Intersections of Modular Correspondences}, Held at the University of Bonn, Bonn, 2003-2004. Ast\'erisque No. 312 (2007).



    \bibitem[BLR90]{BLR} S. Bosch, W. L$\ddot{\mathrm{u}}$tkebohmert, and M. Raynaud,  \textit{N$\acute{\mathrm{e}}$ron Models}, Ergeb. Math. Grenzgeb.(3) 21, Springer, Berlin, 1990


\bibitem[BY]{BY} J. Bruinier and T. Yang, \textit{Arithmetic degrees of special cycles and derivatives of Siegel Eisenstein series},   arXiv:1802.09489.


\bibitem[Bou07]{B} I. I. Bouw, \textit{Invariants of ternary quadratic forms},  Ast\'erisque 312 (2007) 113-137.

\bibitem[Cho15]{C1} S. Cho, \textit{Group schemes and local densities of quadratic lattices in residue characteristic 2}, Compositio Math.
 Vol. 151 (2015) 793-827.

\bibitem[Cho16]{C2} S. Cho, \textit{Group schemes and local densities of ramified hermitian lattices in residue characteristic 2 Part I}, Algebra \& Number Theory, 10-3 (2016)  451-532.






\bibitem[CIKY1]{CIKY1} S. Cho, T. Ikeda, H. Katsurada, and T. Yamauchi, \textit{An inductive formula of the Gross-Keating invariant of a quadratic form}, preprint.

 \bibitem[CIKY2]{CIKY2} S. Cho, T. Ikeda, H. Katsurada, and T. Yamauchi, \textit{Remarks on the Extended Gross-Keating data and the Siegel series of a quadratic form}, arXiv:1709.02772.

\bibitem[CL]{CL}P. Clark, Lectures on Shimura curves 1: Endomorphisms of elliptic curves, available at 
http://math.uga.edu/~pete/SC1-endomorphisms.pdf. 


\bibitem[Con17]{conrad}B. Conrad, \textit{Private communication in April 2017}.




\bibitem[Dur44]{Dur} W. H. Durfee, \textit{Congruence of quadratic forms over valuation rings}, Duke Math. J. 11 (1944) 687-697.

\bibitem[GY00]{GY}  W. T. Gan and J.-K. Yu,  \textit{Group schemes and local densities}, Duke Math. J. 105 (2000) 497-524.

\bibitem[Gro86]{G} B-H. Gross, On canonical and quasi-canonical liftings, Invent math 84, 321-326 (1986).

\bibitem[GK93]{GK1} B-H. Gross and K. Keating, \textit{On the intersection of modular correspondences}, Invent math.112 (1993) 225-245.

\bibitem[G\"or07-1]{G1} U. G\"ortz, \textit{A sum of representation numbers},  Ast\'erisque 312 (2007) 9-14.


\bibitem[G\"or07-2]{G2} U. G\"ortz, \textit{Arithmetic intersection numbers},  Ast\'erisque 312 (2007) 15-24.






\bibitem[Han99]{Han} J. P. Hanke, \textit{An exact mass formula for orthogonal groups over number fields}, Ph.D.
dissertation, Princeton Univ. (1999)  Princeton.

\bibitem[Ike17]{I1} T. Ikeda, \textit{On the functional equation of the Siegel series}, J.  Number Theory. 172 (2017) 44-62.

\bibitem[IK1]{IK1} T. Ikeda and H. Katsurada, \textit{On the Gross-Keating invariant of a quadratic form over a non-archimedean local field},
arXiv:1504.07330, to appear in Amer. J. Math.

\bibitem[IK2]{IK2} T. Ikeda and H. Katsurada, \textit{Explicit formula for the Siegel series of a half-integral matrix over the ring of integers in a non-archimedian local field},
arXiv:1602.06617.





\bibitem[Kat99]{Kat} H. Katsurada, \textit{An explicit formula for Siegel series}, American J. Math. 121 (1999) 415-452.


\bibitem[Katz81]{Katz}N. Katz, Serre-Tate local moduli, Surfaces
alg\'ebriques, in: Lect. Notes in Math., vol. 868, Springer-Verlag, Berlin, 1981, pp. 138-202.

\bibitem[KM85]{KM} N. Katz and B. Mazur, Arithmetic of elliptic curves, Annals of Mathematics Studies, 1985. 





\bibitem[Kit83]{K83} Y. Kitaoka, \textit{A note on local densities of quadratic forms}, Nagoya Math. J. 92 (1983) 145-152.



\bibitem[Kit93]{Kit}  Y. Kitaoka, \textit{Arithmetic of Quadratic Forms}, Cambridge Tracts in Math. 106, Cambridge Univ. Press, Cambridge (1993).




\bibitem[Kud97]{Kudla1}S. Kudla, \textit{Central derivatives of Eisenstein series and height pairings}, Ann of math. (2) 146, no.3 (1997) 545-646.


\bibitem[KR99]{KR2} S. Kudla and M. Rapoport, \textit{Arithmetic of Hirzebruch-Zagier cycles},  J. Reine Angew. Math. 515 (1999) 155-244. 

\bibitem[KR00]{KR1} S. Kudla and M. Rapoport, \textit{Cycles on Siegel 3-folds and derivatives of Eisenstein series}, Ann. Sci. \'Ecole Norm. Sup. (4) 33, no. 5 (2000) 695-756. 


\bibitem[KRY99]{KRY1} S. Kudla and M. Rapoport, and T. Yang, \textit{On the Derivative of an Eisenstein Series of Weight One},
Int. Math. Res. Not., 7 (1999) 347-385.

\bibitem[KRY06]{KRY2} S. Kudla and M. Rapoport, and T. Yang, \textit{Modular Forms and Special Cycles on Shimura Curves},
Annals of Mathematics Studies, vol. 161. Princeton University Press, Princeton (2006).

\bibitem[La87]{Lang} S. Lang,  Elliptic functions. With an appendix by J. Tate. Second edition. Graduate Texts in Mathematics, 112. Springer-Verlag, New York, 1987.



\bibitem[LZ]{LZ} C. Li and W. Zhang, \textit{Kudla-Rapoport cycles and derivatives of local densities}, 
arXiv:1908.01701.


\bibitem[Meu07]{Me}V. Meusers, 
Canonical and quasi-canonical liftings in the split case,  
 Ast\'erisque 312 (2007) 87-98.

\bibitem[Nag92]{nagaoka}S. Nagaoka, A note on the Siegel-Eisenstein series of weight 2 on ${\rm Sp}_2(\Z)$. Manuscripta Math. 77 (1992), no. 1, 
71-88.

\bibitem[O'Me00]{O}  O.~T.~O'Meara, \emph{Introduction to Quadratic Forms}, reprint of 1973 ed., Classics\ Math.\ (2000), Springer, Berlin

\bibitem[O70]{O}F. Oort, Finite group schemes, local moduli for abelian varieties, and lifting problems. Compositio Math. 23 (1971), 265-296.

\bibitem[Pop11]{Pop} M. Popa, \textit{Course on p-adic and motivic integration},  UIC lecture note (2011), 
available at https://sites.math. northwestern.edu/$_{\widetilde{~}}$mpopa/571

\bibitem[Rap07]{Rap1} M. Rapoport, \textit{Deformations of isogenies of formal groups},  Ast\'erisque 312 (2007) 139-169.

\bibitem[Ro]{Ro}A-M. Robert, A Course in $p$-adic Analysis, Graduate Texts in Mathematics, 198. Springer-Verlag, New York, 2000.

 \bibitem[Sie35]{Sie} C. Siegel, \textit{$\ddot{U}ber$ die analytische Theorie der quadratischen Formen}, Ann. Math. 36 (1935) 527-606.

\bibitem[Sil09]{Sil}J. Silverman,  \textit{The Arithmetic of Elliptic Curves}, Second edition. Graduate Texts in Mathematics, 106. Springer, Dordrecht, 2009. xx+513 pp.


\bibitem[Tam66]{Tam} T. Tamagawa, \textit{Ad\'eles, Algebraic Groups and Discontinuous Subgroups (Proc. Sympos.
Pure Math., Boulder, Colo., 1965)}, vol. 9, Amer. Math. Soc., Providence, R.I. (1966)
pp. 113-121.

\bibitem[Vog07]{V1} G. Vogel, \textit{Modular polynomials},  Ast\'erisque 312 (2007) 1-7.


\bibitem[Wed07-1]{Wed1} T. Wedhorn, \textit{Calculation of representation densities},  Ast\'erisque 312 (2007) 179-190.

\bibitem[Wed07-2]{Wed2} T. Wedhorn, \textit{The genus of the endomorphisms of a supersingular elliptic curve},  Ast\'erisque 312 (2007) 25-47.

\bibitem[Weil82]{Weil} A. Weil, \textit{Adeles and Algebraic Groups}, Progress in Mathematics, vol. 23, Birkh\"auser,
Boston Mass. (1982) pp. iii+126.
\bibitem[Yan98]{Y0} T. Yang, \textit{An explicit formula for local densities of quadratic forms}, J. Number Theory 72 (1998) 309-356.


\bibitem[Yan04]{Y} T. Yang, \textit{Local densities of 2-adic quadratic forms}, J. Number Theory 108 (2004) 287-345.



\bibitem[Yok17]{yok} S. Yokoyama, \textit{Private communication in March 2017}. 

\bibitem[Yu95]{Yu} J-.K. Yu, \textit{On the moduli of quasi-canonical liftings}, Compositio Math. 96 (1995), no. 3, 293-321. 
\bibitem[Yu08]{Yu1} J-.K. Yu, \textit{Tamagawa number} Purdue lecture note (2008)

\end{thebibliography}
\end{document}